\DeclareSymbolFont{AMSb}{U}{msb}{m}{n}
\documentclass[11pt,a4paper,twoside,leqno,noamsfonts]{amsart}
\usepackage{relsize}
\usepackage[english]{babel}
\usepackage[dvipsnames]{xcolor}
\definecolor{britishracinggreen}{rgb}{0.0, 0.26, 0.15}
\definecolor{cobalt}{rgb}{0.0, 0.28, 0.67}
\usepackage[utopia]{mathdesign}
    \DeclareSymbolFont{usualmathcal}{OMS}{cmsy}{m}{n}
    \DeclareSymbolFontAlphabet{\mathcal}{usualmathcal}
\usepackage[a4paper,top=3cm,bottom=3cm,left=3cm,right=3cm]{geometry}
\usepackage[utf8]{inputenc}
\usepackage{braket,caption,comment,mathtools,stmaryrd}
\usepackage{multirow,booktabs,microtype}
\usepackage{todonotes}
\usepackage[colorlinks,bookmarks]{hyperref} 
   \hypersetup{colorlinks,%
            citecolor=britishracinggreen,%
            filecolor=black,%
            linkcolor=cobalt,%
            urlcolor=black}
   \numberwithin{equation}{section}
           \usepackage{setspace}
\makeatletter
\newenvironment{proofof}[1]{\par
  \pushQED{\qed}%
  \normalfont \topsep6\p@\@plus6\p@\relax
  \trivlist
  \item[\hskip3\labelsep
        \itshape
    Proof of #1\@addpunct{.}]\ignorespaces
}{%
  \popQED\endtrivlist\@endpefalse
}
\makeatother

\def\be{\begin{equation}}    
\def\ee{\end{equation}}
\def\bitem{\begin{itemize}}
\def\eitem{\end{itemize}}
\def\benum{\begin{enumerate}}
\def\eenum{\end{enumerate}}

\def\ra{\rightarrow}

\def\onto{\twoheadrightarrow}

\def\L{\mathbb L}
\def\A{\mathbb A}
\def\C{\mathbb C}
\def\P{\mathbb P}
\def\Q{\mathbb Q}
\def\G{\mathbb G}
\def\L{\mathbb{L}}

\def\Z{\mathbb Z}

\def\HS{\textrm{HS}}
\def\O{\mathscr O}
\def\DDT{\mathsf{DT}}
\def\PPT{\mathsf{PT}}

\def\MF{\mathsf{MF}}

\def\con{\textrm{con}}

\newcommand{\ff}{\mathbf{f}}
\newcommand{\dd}{\mathbf{d}}
\newcommand{\muhat}{\hat{\mu}}
\newcommand{\Db}[1]{\mathrm{D}^{\mathrm{b}}(#1)}
\newcommand{\prsigma}{{}^{\diamond}\!\!\sigma}
\newcommand{\prExp}{{}^{\diamond}\!\!\Exp}
\DeclareMathOperator{\pr}{\diamond}
\DeclareMathOperator{\Quot}{Quot}
\DeclareMathOperator{\mon}{mon}
\newcommand{\phim}[1]{\phi^{\mon}_{#1}}
\DeclareMathOperator{\Hilb}{Hilb}
\DeclareMathOperator{\Var}{Var}

\newcommand{\Staff}[1]{\mathrm{St}^{\mathrm{aff}}_{#1}}
\DeclareMathOperator{\id}{id}
\DeclareMathOperator{\KK}{K}

\DeclareMathOperator{\BBS}{BBS}
\DeclareMathOperator{\fra}{fr}

\DeclareMathOperator{\Tate}{\mathbb{T}}

\DeclareMathOperator{\Coker}{Coker}
\DeclareMathOperator{\Image}{Image}
\DeclareMathOperator{\Rep}{rMod}
\DeclareMathOperator{\Bl}{Bl}

\DeclareMathOperator{\vir}{vir}
\DeclareMathOperator{\HO}{H}
\DeclareMathOperator{\QCoh}{QCoh}
\DeclareMathOperator{\Coh}{Coh}
\DeclareMathOperator{\crit}{crit}
\DeclareMathOperator{\Span}{Span}

\DeclareMathOperator{\Id}{Id}
\DeclareMathOperator{\Spec}{Spec\,}
\DeclareMathOperator{\Proj}{Proj\,}
\DeclareMathOperator{\Supp}{Supp\,}

\DeclareMathOperator{\DT}{DT}
\DeclareMathOperator{\MHM}{MHM}
\DeclareMathOperator{\MMHM}{MMHM}
\DeclareMathOperator{\PT}{PT}
\DeclareMathOperator{\pt}{pt}
\DeclareMathOperator{\twist}{sh}
\DeclareMathOperator{\sym}{Sym}
\DeclareMathOperator{\Boxtimes}{\mathlarger{\mathlarger{\boxtimes}}}
\DeclareMathOperator{\curv}{crv}
\DeclareMathOperator{\GL}{GL}
\DeclareMathOperator{\HC}{\mathbf{hc}}
\DeclareMathOperator{\unip}{uni}

\DeclareMathOperator{\Tr}{Tr}
\DeclareMathOperator{\NCHilb}{NCHilb}
\DeclareMathOperator{\Sym}{Sym}
\DeclareMathOperator{\nilp}{nilp}

\DeclareMathOperator{\Hom}{Hom}
\DeclareMathOperator{\RHom}{\mathcal{R}Hom}
\DeclareMathOperator{\End}{End}

\DeclareMathOperator{\Exp}{Exp}
\DeclareMathOperator{\rk}{rk}

\DeclareMathOperator{\Tot}{Tot}

\theoremstyle{definition}

\newtheorem*{lemma*}{Lemma}
\newtheorem*{theorem*}{Theorem}
\newtheorem*{example*}{Example}
\newtheorem*{fact*}{Fact}
\newtheorem*{notation*}{Notation}
\newtheorem*{definition*}{Definition}
\newtheorem*{prop*}{Proposition}
\newtheorem*{remark*}{Remark}
\newtheorem*{corollary*}{Corollary}

\newtheorem{definition}{Definition}[section]

\newtheorem{example}[definition]{Example}

\newtheorem{notation}[definition]{Notation}
\newtheorem{remark}[definition]{Remark}
\newtheorem{conjecture}[definition]{Conjecture}

\newtheoremstyle{thm} 
        {3mm}
        {3mm}
        {\slshape}
        {0mm}
        {\bfseries}
        {.}
        {1mm}
        {}
\theoremstyle{thm}
\newtheorem{theorem}[definition]{Theorem}
\newtheorem{corollary}[definition]{Corollary}
\newtheorem{lemma}[definition]{Lemma}
\newtheorem{prop}[definition]{Proposition}
\newtheorem{thm}{Theorem}

\usepackage[all]{xy}
\usepackage{tikz}
\usepackage{tikz-cd}
\usepackage{rotating}

\usetikzlibrary{matrix,shapes,arrows,decorations.pathmorphing}
\tikzset{commutative diagrams/arrow style=math font}
\tikzset{commutative diagrams/.cd,
mysymbol/.style={start anchor=center,end anchor=center,draw=none}}
\newcommand\MySymb[2][\square]{%
  \arrow[mysymbol]{#2}[description]{#1}}
\tikzset{
shift up/.style={
to path={([yshift=#1]\tikztostart.east) -- ([yshift=#1]\tikztotarget.west) \tikztonodes}
}
}

\DeclareMathAlphabet{\mathpzc}{OT1}{pzc}{m}{it}

\newcommand*{\defeq}{\mathrel{\vcenter{\baselineskip0.5ex \lineskiplimit0pt
                     \hbox{\scriptsize.}\hbox{\scriptsize.}}}%
                     =}

\usepackage{xifthen}
\usepackage{verbatim}

\newcounter{x}
\newcounter{y}
\newcounter{z}

\newcommand\xaxis{210}
\newcommand\yaxis{-30}
\newcommand\zaxis{90}

\newcommand\topside[3]{
  \fill[fill=cubecolor, draw=black,shift={(\xaxis:#1)},shift={(\yaxis:#2)},
  shift={(\zaxis:#3)}] (0,0) -- (30:1) -- (0,1) --(150:1)--(0,0);
}

\newcommand\leftside[3]{
  \fill[fill=cubecolor, draw=black,shift={(\xaxis:#1)},shift={(\yaxis:#2)},
  shift={(\zaxis:#3)}] (0,0) -- (0,-1) -- (210:1) --(150:1)--(0,0);
}

\newcommand\rightside[3]{
  \fill[fill=cubecolor, draw=black,shift={(\xaxis:#1)},shift={(\yaxis:#2)},
  shift={(\zaxis:#3)}] (0,0) -- (30:1) -- (-30:1) --(0,-1)--(0,0);
}

\newcommand\cube[3]{
  \topside{#1}{#2}{#3} \leftside{#1}{#2}{#3} \rightside{#1}{#2}{#3}
}

\newcommand*\cubecolors[1]{%
  \ifcase#1\relax
  \or\colorlet{cubecolor}{green}%
  \or\colorlet{cubecolor}{green}%
  \or\colorlet{cubecolor}{green}%
  \or\colorlet{cubecolor}{yellow}%
  \or\colorlet{cubecolor}{yellow}%
  \or\colorlet{cubecolor}{yellow}%
  \else
    \colorlet{cubecolor}{yellow}%
  \fi
}

\newcommand*\cubecolorss[1]{%
  \ifcase#1\relax
  \or\colorlet{cubecolor}{green}%
  \or\colorlet{cubecolor}{yellow}%
  \or\colorlet{cubecolor}{yellow}%
  \or\colorlet{cubecolor}{yellow}%
  \or\colorlet{cubecolor}{yellow}%
  \or\colorlet{cubecolor}{yellow}%
  \else
    \colorlet{cubecolor}{yellow}%
  \fi
}

\newcommand\planepartitionn[1]{
 \setcounter{x}{-1}
  \foreach \a in {#1} {
    \addtocounter{x}{1}
    \setcounter{y}{-1}
    \foreach \b in \a {
      \addtocounter{y}{1}
      \setcounter{z}{-1}
      \foreach \c in {1,...,\b} {
        \addtocounter{z}{1}
        \cubecolorss{\c}
        \cube{\value{x}}{\value{y}}{\value{z}}
      }
    }
  }
}

\title[The local motivic DT/PT correspondence]{The local motivic DT/PT correspondence}

\author[B. Davison]{Ben Davison}
\address{School of Mathematics and Hodge Institute,
University of Edinburgh,
United Kingdom}
\email[Ben Davison]{ben.davison@ed.ac.uk}

\author[A. T. Ricolfi]{Andrea T. Ricolfi}
\address{SISSA Trieste,
Via Bonomea 265, 
34136 Trieste,
Italy}
\email[Andrea Ricolfi]{aricolfi@sissa.it}


\begin{document}

\begin{abstract}
We show that the Quot scheme $Q_L^n = \Quot_{\A^3}(\mathscr I_L,n)$ parameterising length $n$ quotients of the ideal sheaf of a line in $\mathbb{A}^3$ is a global critical locus, and calculate the resulting motivic partition function (varying $n$), in the ring of relative motives over the configuration space of points in $\mathbb{A}^3$.  As in the work of Behrend--Bryan--Szendr\H{o}i this enables us to define a virtual motive for the Quot scheme of $n$ points of the ideal sheaf $\mathscr I_C\subset \O_Y$, where $C\subset Y$ is a smooth curve embedded in a smooth 3-fold $Y$, and we compute the associated motivic partition function. The result fits into a motivic wall-crossing type formula, refining the relation between Behrend's virtual Euler characteristic of $\Quot_Y(\mathscr I_C,n)$ and of the symmetric product $\Sym^nC$.  Our ``relative'' analysis leads to results and conjectures regarding the pushforward of the sheaf of vanishing cycles along the Hilbert--Chow map $Q_L^n \ra \Sym^n(\mathbb{A}^3)$, and connections with cohomological Hall algebra representations.
\end{abstract}

\maketitle
{\hypersetup{linkcolor=black}
\tableofcontents}

\section{Introduction}

\subsection{Overview}
Let $C$ be a smooth curve embedded in a smooth $3$-fold $Y$ with ideal sheaf $\mathscr I_C\subset \O_Y$. For an integer $n\geq 0$, the Quot scheme
\[
Q^n_C = \Quot_Y(\mathscr I_C,n)
\]
parameterises closed subschemes $Z\subset Y$ containing $C$ and differing from it by an effective zero-cycle of length $n$. The main purpose of this paper is to construct a \emph{virtual motive} 
\be\label{qncvir}
\bigl[Q^n_C\bigr]_{\vir}\in\mathcal M_\C
\ee
for this Quot scheme, that we view as a $1$-dimensional analogue of the degree zero \emph{motivic Donaldson--Thomas invariant} $[\Hilb^nY]_{\vir}$ defined by Behrend, Bryan and Szendr\H{o}i \cite{BBS}. 

The Quot scheme $Q^n_C$ can be seen as a moduli space of \emph{curves and points} in $Y$, where the curve $C$ is fixed. This geometric situation presents a new feature that was absent in the purely $0$-dimensional case: wall-crossing. More precisely, it is proved in \cite[Prop.~5.1]{LocalDT} that the generating function of the Behrend weighted Euler characteristics $\widetilde\chi(Q^n_C)$ satisfies the wall-crossing type formula
\be\label{WCformula}
\sum_{n\geq 0} \widetilde\chi(Q^n_C)t^n = M(-t)^{\chi(Y)}\cdot (1+t)^{-\chi(C)},
\ee
where $M(t) = \prod_{m\geq 1}(1-t^m)^{-m}$ is the MacMahon function. We show that the motivic partition function encoding the motivic classes \eqref{qncvir} admits a factorisation similar to \eqref{WCformula}, where the \emph{point contribution}, refining the factor $M(-t)^{\chi(Y)}$, is precisely the motivic partition of the Hilbert schemes $\Hilb^nY$ computed in \cite{BBS}. The \emph{curve contribution}, on the other hand, refines the factor $(1+t)^{-\chi(C)}$ and is given by the (shifted) motivic zeta function of the curve $C$, namely
\be\label{curve_contribution}
\sum_{n\geq 0}\,\L^{-\frac{n}{2}}\bigl[\Sym^nC\bigr]_{\vir} t^n.
\ee
Our approach to the problem is a natural extension of the approach of Behrend, Bryan and Szendr\H{o}i, in that our definitions and calculations take place with respect to the natural local model $L\subset \mathbb{A}^3$ given by fixing a line in affine space --- since they consider only finite-dimensional quotients of $\mathscr O_{\mathbb{A}^3}$, their local model is simply $\mathbb{A}^3$.  As in their work, for general embeddings $C\subset Y$ we build $[Q^n_C]_{\vir}$ out of the local model via power structures.  We leave for another day the question of whether this virtual motive accords with the virtual motive one obtains from the machinery of ($-1$)-shifted symplectic stacks, and concentrate on calculating everything in sight for the local theory.  Furthermore, since the key to gluing local models appears to be the direct image of the vanishing cycles sheaf to the configuration space of points on $\mathbb{A}^3$, we prove all of our results in the lambda ring of motives relative to this configuration space.  We conjecture, moreover, that the wall-crossing type identity \eqref{WCformula} can be categorified, i.e.~lifted to an isomorphism between the vanishing cycle cohomologies of the relevant moduli spaces (see Section \ref{further_sec_1}).

Just as Behrend, Bryan and Szendr\H{o}i realise the local model $\Hilb^n\A^3$ as a \emph{critical locus} and show that the associated motivic Donaldson--Thomas invariants $[\Hilb^n \A^3]_{\vir}$ are determined, via power structures, by the motivic weights of the punctual Hilbert schemes $\Hilb^n(\A^3)_0$, we show, for two convenient local models that can also be realised as critical loci, that the induced virtual motives are determined by the motivic contribution of the punctual loci. In our case, we also need to consider the contribution of points \emph{embedded} on the curve $C\subset Y$, and this is what gives rise to the factor \eqref{curve_contribution} in our motivic wall-crossing formula.

The appearance of symmetric products is pretty natural and has a neat interpretation in terms of the ($C$-local) DT/PT correspondence: on a Calabi--Yau $3$-fold $Y$, the symmetric products $\Sym^nC\subset P_{\chi(\O_C)+n}(Y,[C])$ are precisely the $C$-local moduli spaces in the stable pair theory of $Y$, just as $Q^n_C\subset I_{\chi(\O_C)+n}(Y,[C])$ are the $C$-local moduli spaces in Donaldson--Thomas theory.

For a rigid curve $C\subset Y$ in a Calabi--Yau $3$-fold, one can interpret the classes \eqref{qncvir} as motivic Donaldson--Thomas invariants, in the same spirit as in the zero-dimensional case.

\smallbreak
We next give an overview of our main results. The main technical tool we use is a \emph{motivic stratification} technique, that we apply to the study of the (relative) motivic Donaldson--Thomas invariants of the Quot schemes $Q^n_{C_0}$, where $C_0\subset X$ is the exceptional curve in the resolved conifold $X = \Tot(\O_{\P^1}(-1)\oplus \O_{\P^1}(-1))$.

\subsection{Main results}

The first step towards the construction of the motivic classes \eqref{qncvir} consists in setting up a convenient local model. With respect to the local model
\[
L\subset \A^3
\]
we then prove the following as part of Theorem \ref{mainthm1}.

\begin{thm}\label{thm:thm1}
The Quot scheme $Q^n_L$ is a global critical locus.
\end{thm}

An analogous statement is proven in \cite[Prop.~3.1]{BBS} for the Hilbert scheme $\Hilb^n(\A^3)$, which is realised as the critical locus of a function on the  \emph{non-commutative Hilbert scheme}.

Via the theory of motivic vanishing cycles \cite{DenefLoeser1}, Theorem \ref{thm:thm1} produces a relative virtual motive
\[
\mathsf Q_{L/\A^3}^{\mathrm{rel}} = \sum_{n\geq 0}\,(-1)^n\bigl[Q^n_L \xrightarrow{\HC_n} \Sym^n \A^3\bigr]_{\vir}\in\mathcal M_{\Sym(\A^3)}
\]
where the maps $\HC_n$ are Hilbert--Chow morphisms. The following result, proven in Section \ref{sec:Thm_B&C}, follows from Corollary \ref{Omega_curv_calc} and the main calculation of Section \ref{sec:mstr}. 

\begin{thm}\label{thm:newB}
There is an identity 
\be\label{Quot_affine_factorised}
\mathsf Q_{L/\A^3}^{\mathrm{rel}} = \Exp_\cup\left(\sum_{n>0} \Delta_{n\,!}\left(\bigl[\A^3\xrightarrow{\id}\A^3\bigr] \boxtimes \Omega_n^{\mathrm{BBS}} \right)\right) \boxtimes_{\cup} 
\Exp_\cup\left(\Delta_{1\,!}\left([L\hookrightarrow \A^3]\boxtimes \bigl(-\L^{-\frac{1}{2}}\bigr)\right)\right),
\ee
where $\Delta_n\colon \A^3\ra \Sym^n\A^3$ is the small diagonal, and
\[
\Omega_n^{\mathrm{BBS}} = (-1)^n\mathbb{L}^{-\frac{3}{2}}\frac{\L^{\frac{n}{2}}-\L^{-\frac{n}{2}}}{\L^{\frac{1}{2}}-\L^{-\frac{1}{2}}} \in \mathcal M_{\C}.
\]
\end{thm}

Passing to absolute motives, the first factor in \eqref{Quot_affine_factorised} becomes the (signed) motivic partition function of $\Hilb^n(\A^3)$ computed in \cite{BBS} and reviewed in Section \ref{Section:BBS}.  The operation $\Exp_{\cup}$ in Theorem \ref{thm:newB} is a lift to the lambda ring $\mathcal{M}_{\Sym(\mathbb{A}^3)}$ of the usual plethystic exponential for power series with coefficients in the ring of absolute motives.  
These motivic exponentials are reviewed in Section \ref{subsec:Exp}. We let 
\[
\mathsf Q_{L/\A^3}(t) = \sum_{n\geq 0}\,\bigl[Q^n_L\bigr]_{\vir}\cdot t^n
\]
be the absolute partition function. Up to a sign, it is obtained by pushing \eqref{Quot_affine_factorised} forward to a point. The absolute version of Theorem \ref{thm:newB} then reads
\be\label{local_quot}
\mathsf Q_{L/\A^3}(-t) = \Exp\left(\frac{-\L^{\frac{3}{2}}t}{\bigl(1+\L^{-\frac{1}{2}}t\bigr)\bigl(1+\L^{\frac{1}{2}}t\bigr)}-\L^{\frac{1}{2}}t\right).
\ee

\smallbreak
Let $\mathcal P_{\curv}^n\subset Q^n_L$ be the closed subset parameterising quotients $\mathscr I_L\onto \mathscr F$ such that $\mathscr F$ is entirely supported at the origin $0\in L$. We assign a motivic weight 
\[
\bigl[\mathcal P_{\curv}^n\bigr]_{\vir} \in \mathcal M_{\C}
\]
to this locus. The subscript `crv' stands for `curve'. The punctual Hilbert scheme $\mathcal P_{\pt}^n \subset Q^n_L$, which we view as parameterising quotients supported at a single point in $\A^3\setminus L$, also inherits a motivic weight, that agrees (as shown in Proposition \ref{prop:remove_pr_pt}) with the class $[\Hilb^n(\A^3)_0]_{\vir}$ defined in \cite{BBS} starting from the critical structure on $\Hilb^n(\A^3)$. We show in Theorem \ref{Thm:VM_local} that $[Q^n_L]_{\vir}$ is determined by these two types of ``punctual'' motivic classes. They moreover allow us to define a virtual motive $[Q^n_C]_{\vir}\in\mathcal M_{\C}$ for every smooth curve $C\subset Y$ in a smooth quasi-projective $3$-fold $Y$. In other words, the class we define satisfies $\chi[Q^n_C]_{\vir} = \widetilde\chi(Q^n_C)$.

We then consider the generating function
\be\label{quotGF}
\mathsf Q_{C/Y}(t) = \sum_{n\geq 0}\,\bigl[Q^n_C\bigr]_{\vir}\cdot t^n.
\ee
For a smooth quasi-projective variety $X$ of dimension at most $3$, let 
\[
\mathsf Z_X(t) = \sum_{n\geq 0}\,\bigl[\Hilb^nX\bigr]_{\vir}\cdot t^n
\]
be the motivic partition function of the Hilbert scheme of points. In Theorem \ref{thm:C} we prove the following explicit formula, generalising \eqref{local_quot}.

\begin{thm}\label{thm:thm2}
Let $Y$ be a smooth quasi-projective $3$-fold, $C\subset Y$ a smooth curve. Then
\be\label{mainformula}
\bigl[Q^n_C\bigr]_{\vir} = \sum_{j=0}^n\,\bigl[\Hilb^{n-j}Y\bigr]_{\vir}\cdot \bigl[\Sym^jC\bigr]_{\vir}
\ee
in $\mathcal M_\C$. In other words, we have a factorisation
\[
\mathsf Q_{C/Y} = \mathsf Z_Y\cdot \mathsf Z_C,
\]
that, rewritten in terms of the motivic exponential, reads
\[
\mathsf Q_{C/Y}(-t) = \Exp\left(-t[Y]_{\vir}\Exp(-t[\P^1]_{\vir})-t[C]_{\vir}\right).
\]
\end{thm}

In the above formulas, one has $[U]_{\vir} = \L^{-(\dim U)/2}[U] \in \mathcal M_{\C}$ for a smooth scheme $U$. One can view the factorisation $\mathsf Q_{C/Y} = \mathsf Z_Y\cdot \mathsf Z_C$ as a motivic refinement of the identity \eqref{WCformula}. Indeed, we have $M(-t)^{\chi(Y)}=\chi\mathsf Z_Y(t)$, and $(1+t)^{-\chi(C)}=\sum_{n}\widetilde\chi(\Sym^nC)t^n=\chi\mathsf Z_C(t)$. 
The relation \eqref{WCformula} says that Quot schemes and symmetric products are related by a $\widetilde\chi$-weighted wall-crossing type formula, and Theorem \ref{thm:thm2} upgrades this statement to the motivic level.

\subsection{Calabi--Yau $3$-folds}
Let $Y$ be a smooth projective Calabi--Yau $3$-fold. 
For an integer $m\in \Z$ and a homology class $\beta\in H_2(Y,\Z)$, the moduli space $I_m(Y,\beta)$ of ideal sheaves $\mathscr I_Z\subset \O_Y$ with Chern character $(1,0,-\beta,-m)$ carries a symmetric perfect obstruction theory and the Donaldson--Thomas invariant $\DDT^m_{\beta} \in \Z$ is by definition the degree of the associated virtual fundamental class. These invariants are related to the stable pair invariants of Pandharipande--Thomas \cite{PT} by a well known wall-crossing formula \cite{Bri,Toda1}, and the same is true for the $C$-local invariants $\DDT^n_{C}\in\mathbb Z$. The numbers $\DDT^{\bullet}_{C}$ represent the \emph{contribution} of $C$ to the full virtual invariants $\DDT^{\bullet}_{[C]}$. The $C$-local wall-crossing formula \cite[Theorem 1.1]{Ricolfi2018}, written term by term, reads
\be\label{formula:DTnC}
\DDT^n_{C} = \sum_{j=0}^n\DDT^{n-j}_{0}\cdot \PPT^j_{C},
\ee
where $\DDT^k_{0} = \widetilde\chi(\Hilb^kY)$ are the degree zero DT invariants of $Y$, $\PPT^j_{C} = n_{g,C}\cdot \widetilde\chi(\Sym^jC)$ are the $C$-local stable pair invariants of $Y$ and $n_{g,C}$ is the BPS number of $C$ (see \cite{BPS} and Section \ref{Sec:DT}). There is an identity $\chi [Q^n_C]_{\vir} = \DDT^n_{C}$ when $n_{g,C} = 1$ (Corollary \ref{cor_BPS}). Indeed, in this case equation \eqref{formula:DTnC} is equivalent to \eqref{WCformula}. This is especially meaningful from the point of view of motivic DT theory in the situation of the following example.

\begin{example}
Assume $C\subset Y$ is a smooth \emph{rigid} curve, i.e.~$H^0(C,N_{C/Y})=0$. Then $C$ has BPS number $1$, the Quot scheme $Q^n_C$ is a \emph{connected component} of the Hilbert scheme $I_{\chi(\O_C)+n}(Y,[C])$, and the motivic class $[Q^n_C]_{\vir}$ is a \emph{motivic Donaldson--Thomas invariant} in the sense that its Euler characteristic computes the degree of the virtual fundamental class of $Q^n_C$. In this case, the formula $\mathsf Q_{C/Y} = \mathsf Z_Y\cdot \mathsf Z_C$ of Theorem \ref{thm:thm2} can be regarded as a $C$-local motivic DT/PT correspondence, refining the enumerative correspondence $\DDT_C = \DDT_0\cdot \PPT_C$ spelled out in \eqref{formula:DTnC}.
\end{example}

\subsection{Organisation of contents}
The paper is organised as follows. In Section \ref{sec:Prelim} we recall foundational material on rings of motivic weights and we revisit the main formula of \cite{BBS} expressing the virtual motive of $\Hilb^nX$ for $3$-folds. In Section \ref{CriticalLocus} we prove Theorem \ref{thm:thm1} by restricting the critical structure on $Q^n_{C_0}$, where $C_0 \cong \P^1$ is the exceptional curve in the resolved conifold $X = \Tot(\O_{\P^1}(-1)\oplus \O_{\P^1}(-1))$. In Section \ref{sec:mstr} we prove that the virtual motives of $Q^n_L$ and $Q^n_{C_0}$ are determined by motivic classes  $\Omega_{\pt}^n$, $\Omega_{\curv}^n$ expressing the contributions of ``fully punctual loci'' (cf.~Definition \ref{defin:Fully_Punctual} and Theorem \ref{punctual_thm}). By explicitly calculating these motives in Section \ref{sec:Thm_B&C} we finally prove Theorem \ref{thm:newB}.  We then use these classes to \emph{define} (cf.~Definition \ref{def:VM_arbitrary_curve}) a virtual motive of $Q^n_C$ for every smooth curve $C$ in a smooth $3$-fold $Y$, and in Section \ref{sec:Thm_B&C} we also prove Theorem \ref{thm:thm2}.

\subsection*{Acknowledgements}
This collaboration was supported by the National Science Foundation under Grant No 1440140, while BD was in residence at the
Mathematical Sciences Research Institute in Berkeley and AR was visiting.   During the writing of the paper, BD was supported by
the starter grant ``Categorified Donaldson-Thomas theory'' No.~759967 of the European Research Council, which also enabled AR to visit Edinburgh.  BD was also supported by a Royal Society 
university research fellowship.  Much of the work was done while BD was a guest of the University of British Columbia in the 
summer of 2018, supported by a Royal Society research fellows enhancement award.  He would like to thank Jim Bryan and Kai 
Behrend for helping to make this stay so enjoyable and productive. AR would like to thank Martin Gulbrandsen and Lars Halle for the many helpful discussions related to the subject of this paper. He would also like to thank Max-Planck Institut f\"{u}r Mathematik (Bonn) and SISSA (Trieste) for the excellent working conditions offered during the completion of this project.

\section{Background material}\label{sec:Prelim}

In this section we set up the notation and introduce the main tools that will be used in the rest of the paper. 

\subsection{Grothendieck rings of varieties}\label{K_Rings}

\begin{definition}
Let $S$ be a locally finite type algebraic space over $\mathbb C$. 
\bitem
\item [(i)] If $S$ is a variety, the \emph{Grothendieck group of $S$-varieties} is the free abelian group $K_0(\Var_S)$ generated by isomorphism
classes $[X\ra S]$ of finite type varieties over $S$, modulo the scissor relations, namely the identities 
$[p\colon Y\ra S]=[p|_X\colon X\ra S]+[p|_{Y\setminus X}\colon Y\setminus X\ra S]$ whenever $X\hookrightarrow Y$ is a closed $S$-subvariety of $Y$.  For general $S$, we impose the \textit{locality} relation $[f\colon X\xrightarrow{} S]=[g\colon X'\xrightarrow{} S]$ if for all varieties $U\subset S$ there is an identity $[f\lvert_U\colon X\times_S U\xrightarrow{} U]=[g\lvert_U\colon X'\times_S U\xrightarrow{} U]$ in $K_0(\Var_U)$. The group $K_0(\Var_S)$ is a ring via $[Y\ra S]\cdot [Z\ra S]=[Y\times_SZ\ra S]$. 
\item [(ii)] We denote by $\L=[\A^1_S]\in K_0(\Var_S)$ the \emph{Lefschetz motive}, the class of the affine line over $S$.
\item [(iii)]
The \emph{Grothendieck group of $S$-stacks} is the free abelian group $K_0(\Staff{S})$ generated by isomorphism
classes $[X\ra S]$ of locally finite type Artin $S$-stacks $X\ra S$ with affine stabilizers, modulo the scissor and locality relations, and the following additional relation: if $f\colon X\rightarrow S$ is an $S$-stack, such that $f$ factors as $g\circ \pi$ for $g\colon Y\rightarrow S$ an $S$-stack and $\pi$ the projection from the total space of a rank $r$ vector bundle, then
\[
[X\xrightarrow{f}S]=\mathbb{L}^r\cdot [Y\xrightarrow{g}S].
\]
\item [(iv)]
Define the group $K(\Var_S)=\Image\left(K_0(\Var_S)\rightarrow K_0(\Staff{S})\right)$, and give it the induced ring structure.
\eitem
\end{definition}
Where $S=\coprod_{i\in I} S_i$ is a possibly infinite union of algebraic spaces, we will write 
\[
\sum_{i\in I}\,\bigl[X_i\xrightarrow{f_i}S_i\bigr]\defeq \left[\coprod_{i\in I}(X_i\xrightarrow{f_i}S_i)\right].
\]
By results of Kresch \cite[Section 4]{kreschcycle}, we have
\[
K_0(\Staff{S})=K_0(\Var_S)\bigl[\L^{-1},(\mathbb{L}^{n}-1)^{-1} \,\big|\, n\geq 1\bigr]
\]
and so we can alternatively define $K(\Var_S)$ as the quotient of $K_0(\Var_S)$ by the ideal
\be\label{ideal_J}
\mathsf J_S = \ker(\cdot \L) + \sum_{n\geq 1} \ker(\cdot (\L^n-1)) \subset K_0(\Var_S).
\ee

For $S$ and $S'$ two varieties, there is an external product
\[
K(\Var_{S})\times K(\Var_{S'})\xrightarrow{\boxtimes} K(\Var_{S\times S'}),
\]
defined on generators by
\[
[g\colon Y\ra S]\boxtimes [h\colon Z\ra S'] = [g\times h\colon Y\times Z \ra S\times S'].
\]
In particular, $K(\Var_S)$ is a $K(\Var_\C)$-module.  When we are considering the action of absolute motives on relative motives, we will often abbreviate
\begin{align*}
[X][X'\xrightarrow{f} S']&=[X\rightarrow \pt]\boxtimes[X'\xrightarrow{f} S']
\\ &=[X\times X'\xrightarrow{f\circ \pi_{X'}} S'].
\end{align*}
Often for a relative motive $[X\ra S]\in K(\Var_S)$ we will denote it by $[X]_S$, retaining the subscript to at least remind the reader of which motivic ring it lives in.  

Given a morphism $f\colon S\ra T$ of varieties, there is an induced \emph{pullback} map
\[
f^\ast\colon K(\Var_T)\ra K(\Var_S)
\]
which is a ring homomorphism given by $f^\ast[X]_T=[X\times_TS]_S$ on generators. 
Composition with $f$ defines a \emph{direct image} 
homomorphism $f_!\colon K(\Var_S)\ra K(\Var_T)$, which is $K(\Var_T)$-linear. 

If $S$ comes with an associative map $\nu\colon S\times S\ra S$, we define the convolution ring structure via $\boxtimes_{\nu}=\nu_!\circ\boxtimes$, i.e.~we set
\begin{equation}\label{def:boxtimes_nu}
A \boxtimes_\nu B = \nu_!(A\boxtimes B) \in  K(\Var_S).
\end{equation}
The resulting associative product on $K(\Var_S)$ is commutative if $\nu$ commutes with the symmetrising isomorphism.

The ring
\[
\mathcal M_S=K(\Var_S)\bigl[\L^{-\frac{1}{2}}\bigr]
\]
is called the \emph{ring of motivic weights} over $S$. The structures $f^*$, $f_!$, $\boxtimes$ and $\boxtimes_{\nu}$ carry over to $\mathcal M_S$ without change. 
When $f\colon S\ra \Spec \C$ is the structure morphism of $S$, we use the special notation $\int_S$ for the pushforward $f_!$.

\begin{definition}\label{effectiveness_defn}
We define $S_0(\Var_S)\subset \mathcal{M}_S$ to be the sub semigroup of \textit{effective} motives, i.e.~the subset of sums of elements of the form
\[
(-\mathbb{L}^{\frac{1}{2}})^n[X\rightarrow S].
\]
\end{definition}
\begin{remark}
By Definition \ref{effectiveness_defn} the motive $-\mathbb{L}^{1/2}$ is effective, as opposed to $\mathbb{L}^{1/2}$.  This is dictated by the fact that in the language of lambda rings (Section \ref{sec:lambda_rings}), we make definitions so that $\mathbb{L}^{1/2}$ is not a line element, while $-\mathbb{L}^{1/2}$ is.
\end{remark}

\subsection{Equivariant $K$-groups, quotient and power maps}\label{sec:equivariant}
Let $G$ be a finite group.
A $G$-action on a variety $X$ is said to be \emph{good} if every point of $X$ has a $G$-invariant affine open neighborhood; all actions are assumed to be good throughout. For instance, any $G$-action on a quasi-projective variety is good. Moreover, for a good $G$-action, an orbit space $X/G$ exists as a variety. 

\begin{definition}
Let $S$ be a variety with good $G$-action. We let $\widetilde{K}_0^{G}(\Var_S)$ denote the abelian group generated by isomorphism classes $[X\ra S]$ of $G$-equivariant $S$-varieties, modulo the $G$-scissor relations (over $S$).  The \emph{equivariant Grothendieck group} $K_0^{G}(\Var_S)$ is defined by imposing the further relations $[V\ra X\ra S]= [\A^r_X]$,
whenever $V\ra X$ is a $G$-equivariant vector bundle of rank $r$, with $X\ra S$ a $G$-equivariant $S$-variety.
The element $[\A^r_X]$ in the right hand side is taken with the $G$-action induced by the trivial action on $\A^r$ and the isomorphism $\A^r_X=\A^r\times X$.
\end{definition}

There is a natural ring structure on $\widetilde K_0^{G}(\Var_S)$ given by 
fibre product; if $X$ and $Y$ are $G$-equivariant $S$-varieties we give $X\times_S Y$ the diagonal $G$-action.

We shall consider the quotient rings
\[
\widetilde{K}_0^{G}(\Var_S) \onto \widetilde{K}^{G}(\Var_S), \quad K_0^{G}(\Var_S) \onto K^{G}(\Var_S)
\]
obtained by modding out the ideal 
\[
\widetilde{\mathsf J}^G = \ker(\cdot \L) + \sum_{n\geq 1} \ker(\cdot (\L^n-1)) \subset \widetilde{K}_0^{G}(\Var_S)
\]
and its image $\mathsf J^G \subset K_0^{G}(\Var_S)$, respectively.
We let 
\[
\widetilde{\mathcal M}^G_S = \widetilde{K}^{G}(\Var_S)\bigl[\L^{-\frac{1}{2}}\bigr],\quad \mathcal M^G_S = K^{G}(\Var_S)\bigl[\L^{-\frac{1}{2}}\bigr]
\]
be the rings of $G$-\emph{equivariant motivic weights}.

There exists a natural ``quotient map'' 
\be\label{map:quotient_map}
\pi_G\colon \widetilde{K}_0^{G}(\Var_{S})\ra K_0(\Var_{S/G}),
\ee
defined on generators by taking the orbit space:
\[
\pi_G[X\rightarrow S]=[X/G\rightarrow S/G].
\]
If the $G$-action on $S$ is trivial, $\widetilde{K}_0^{G}(\Var_S)$ becomes a $K_0(\Var_S)$-algebra, and $\pi_G$ is $K_0(\Var_S)$-linear. More generally, we have the following:

\begin{lemma}\label{pi_linearity_lemma}
The map \eqref{map:quotient_map} is $K_0(\Var_{S/G})$-linear.
\end{lemma}

\begin{proof}
The action of a generator $u = [U\ra S/G]\in K_0(\Var_{S/G})$ on a $G$-equivariant motive $x = [h\colon X\ra S] \in \widetilde{K}_0^{G}(\Var_{S})$ is given by 
\[
u\cdot x = h_!h^*q^*(u) = [U\times_{S/G}X\xrightarrow{\mathrm{pr}_2} X\xrightarrow{h}S],
\]
where $q\colon S\ra S/G$ is the quotient map. We have 
\[
u\cdot \pi_G(x) = u\cdot [X/G\ra S/G] = [U\times_{S/G}X/G \ra X/G \ra S/G],
\]
and this is the same motive as $\pi_G(u\cdot x) = \pi_G[U\times_{S/G}X \ra X \ra S]$, since $G$ does not act on $U$.
\end{proof}

By Lemma \ref{pi_linearity_lemma} the map \eqref{map:quotient_map} sends the ideal $\widetilde{\mathsf J}^G$ onto the ideal  $\mathsf J_{S/G} \subset K_0(\Var_{S/G})$ defined in \eqref{ideal_J}, therefore it descends to a $K(\Var_{S/G})$-linear map 
\[
\pi_G\colon \widetilde{K}^{G}(\Var_S)\ra K(\Var_{S/G}).
\]
This map extends to a map $\widetilde{\mathcal M}^G_S\rightarrow \mathcal M_{S/G}$, still denoted $\pi_G$, by setting $\pi_G(\mathbb{L}^{n/2}\cdot [X\rightarrow S])=\mathbb{L}^{n/2}\cdot \pi_G([X\rightarrow S])$. 

Furthermore, by \cite[Lemma 3.2]{Bittner05}, if the $G$-action on $S$ is free,\footnote{Without freeness, the naive quotient map may fail to respect the relation identifying $G$-bundles with trivial $G$-bundles.} $\pi_G$ descends to a $K(\Var_{S/G})$-linear map
\[
\pi_{G}\colon K^{G}(\Var_S)\rightarrow K(\Var_{S/G}),
\]
which again extends to a morphism $\pi_G\colon \mathcal M_S^G \ra \mathcal M_{S/G}$.

\smallbreak
Let $\mathfrak S_n$ be the symmetric group on $n$ elements.

\begin{lemma}[{\cite[Lemma $2.4$]{BBS}}]\label{lemma:majdgap}
For every $n>0$ there exists a $n$-th power map
\[
(\,\cdot\,)^{\otimes n}\colon \mathcal M_S\ra \widetilde{\mathcal M}_{S^n}^{\,\mathfrak S_n}
\]
where $S^n$ carries the natural $\mathfrak{S}_n$-action, defined by the property that for 
\[
T=(-\mathbb{L}^{\frac{1}{2}})^{\alpha}\cdot [A\xrightarrow{f} S]+\mathbb{L}^\beta\cdot [B\xrightarrow{g} S]-(-\mathbb{L}^{\frac{1}{2}})^{\gamma}\cdot [C\xrightarrow{h}S]-\mathbb{L}^\delta\cdot[D\xrightarrow{i}S]\in \mathcal M_S,
\]
we have 
\[
T^{\otimes n}\defeq \sum_{a+b+c+d=n}(-1)^{c+d}(-\mathbb{L}^{\frac{1}{2}})^{a\alpha/2+b\beta+c\gamma/2+d\delta}[X_{a,b,c,d}\rightarrow S^n]
\]
where $X_{a,b,c,d}$ is the space of homomorphisms of schemes 
\[
s:\{1,\ldots, n\}\rightarrow A\cup B\cup C\cup D
\]
with the domain considered as a scheme with $n$ points, $a$ of which are sent to $A$, $b$ of which are sent to $B$, and so on.  We consider this variety as a $\mathfrak{S}_n$-equivariant variety over $S^n$, sending $s$ to the point $(j(s(1)),\ldots, j(s(n)))$, where $j\colon A\cup B\cup C\cup D\rightarrow S$ is the union of the maps $f,g,h,i$.
\end{lemma}
The above lemma is proved in \cite{BBS} in the case $S=\Spec \C$, but the proof for the general case is the same.  We remark that, by definition, there is an identity 
\begin{equation}
    \label{SymSign}
(-[A\rightarrow S])^{\otimes n}=(-1)^n[A\rightarrow S]^{\otimes n}.
\end{equation}

\subsubsection{The monodromic motivic ring}\label{sec:Monodromic}
Let $\mu_n=\Spec \C[x]/(x^n-1)$ be the group of $n$-th roots of unity. We define a good action of the procyclic group
\[
\hat\mu=\varprojlim \mu_n
\]
as an action that factors through a good $\mu_n$-action for some $n$.
The additive group 
\[
\mathcal M_S^{\hat\mu}
\]
carries a commutative bilinear associative product `$\star$' called the \emph{convolution product}. 
See \cite[Section $5$]{DenefLoeser1}
or \cite[Section $7$]{LooijengaMM} for its definition. The product `$\star$'
provides an alternative ring structure on the group of $\hat\mu$-equivariant motivic weights, and it restricts to the usual 
product `$\cdot$' on the subring
\[
\mathcal M_S\subset\mathcal M_S^{\hat\mu}
\]
of classes with trivial $\hat\mu$-action. The main role of `$\star$' will be played through the motivic Thom--Sebastiani theorem (cf.~Theorem \ref{mtseb190}).

\subsection{Lambda ring structures}\label{sec:lambda_rings}

Let $A\in\mathcal{M}_{S}$. We define 
\[
\prsigma^n(A)=\pi_{\mathfrak{S}_n}(A^{\otimes n}) \in \mathcal M_{S^n/\mathfrak S_n}.
\]
The \textit{lambda ring} operations on $K_0(\Var_{\mathbb{C}})[\mathbb{L}^{-1/2}]$ are defined by setting $\sigma^n(A)=\prsigma^n(A)$ for $A$ effective, and then taking the unique extension to a lambda ring on $K_0(\Var_{\mathbb{C}})[\mathbb{L}^{-1/2}]$, determined by the relation
\begin{equation}
\label{lambda_rel}
\sum_{i=0}^n\sigma^i([X]-[Y])\sigma^{n-i}([Y])=\sigma^n([X]).
\end{equation}
Note that $\sigma^n(-\L^{1/2})=(-\L^{1/2})^n$.  By \cite[Rem.~3.5\,(4)]{BenSven3} these operations induce a lambda ring structure on the localization $K_0(\Staff{\mathbb{C}})[\mathbb{L}^{-1/2}]$, and thus a lambda ring structure on $\mathcal{M}_{\mathbb{C}}$.
\begin{remark}
Note that by definition, $\prsigma^n(A)=\sigma^n(A)$ for $A$ effective.  The logical structure of the paper is such that we will often end up proving relations involving $\prsigma^n$ first, and then using them to prove that the motives we consider are effective, so that we can state those same relations in terms of the more well-behaved operations $\sigma^n$.
\end{remark}

If $S$ comes with a commutative associative map $\nu\colon S\times S\ra S$, and $A \in \mathcal M_S$, we likewise define 
\[
\prsigma_{\nu}^n(A)=\nu_! \left(\prsigma^n(A)\right)=\nu_!\left(\pi_{\mathfrak{S}_n}(A^{\otimes n})\right)\in \mathcal M_S,
\]
where we abuse notation by denoting by $\nu$ the map $S^n/\mathfrak{S}_n\rightarrow S$.  As above, using the analogue of the relation (\ref{lambda_rel}) there is a unique set of lambda ring operators $\sigma_{\nu}^n$ agreeing with $\prsigma_{\nu}$ on effective motives. 

As a special case, we obtain operations $\prsigma^n$ and $\sigma^n$ on $\mathcal{M}_{\C}\llbracket t\rrbracket$ via the isomorphism
\be\label{Iso_Power_series}
\mathcal{M}_{\C}\llbracket t\rrbracket \,\widetilde{\ra}\,\mathcal{M}_{\mathbb{N}}
\ee
defined by
\be
\sum_{n\geq 0}\,[X_n]t^n\mapsto \left[\coprod_{n\in\mathbb{N}}X_n\rightarrow \{n\}\right]
\ee
for $X_0,X_1,\ldots$ varieties, and then extending by linearity.  Here, $\mathbb{N}$ is a considered as a scheme by identifying each natural number with a distinct closed point, and this scheme is considered as a commutative monoid under the addition map.

\subsection{Motivic measures}

Ring homomorphisms with source $K(\Var_\C)$ or $\mathcal M_\C$ 
are frequently called \emph{motivic measures},
realisations, or generalised Euler characteristics. We recall 
some of them here. 

Let $K_0(\HS)$ be the Grothendieck ring of the abelian category $\HS$ of Hodge structures.  For a complex variety $X$, taking its \emph{Hodge characteristic}
\[
\chi_{\mathrm{h}}(X)=\sum_{i\geq 0}(-1)^i\bigl[\mathrm{H}^i_c(X,\Q)\bigr]\in K_0(\HS)
\]
defines a motivic measure. The $E$-\emph{polynomial} is the specialisation
\[
E(X)=\sum_{p,q\geq 0}(-1)^{p+q}h^{p,q}\bigl(\mathrm{H}^{p+q}_c(X,\Q)\bigr)u^pv^q\in \Z[u,v].
\]
As $E(\A^1_\C)=uv$, the $E$-polynomial can be extended to a motivic measure
\[
E\colon \mathcal M_\C\ra \Z\,\bigl[u,v,(uv)^{-\frac{1}{2}}\bigr]
\]
satisfying $E(\L^{1/2})=-(uv)^{1/2}$.
The further specialisations $u=v=q^{1/2}$, $(uv)^{1/2}=q^{1/2}$
define the \emph{weight polynomial} $W\colon \mathcal M_\C\ra \Z[q^{\pm 1/2}]$ and one has $W(\L)=q$. Finally, specialising to $q^{1/2}=1$ recovers the Euler characteristic $\chi\colon K(\Var_{\C})\ra \Z$,
extending to $\chi\colon \mathcal{M}_{\mathbb{C}}\ra \Z$ after setting
\[
\chi(\L^{-\frac{1}{2}})=-1.
\]
See \cite[Section $2$]{DenefLoeser1} for a natural extension to a ring homomorphism 
\[
\chi\colon \mathcal M_\C^{\hat\mu}\ra \Z.
\] 
\begin{remark}\label{Rmk:Conventions_BBS}
Our sign conventions differ slightly from \cite{BBS}.  We have chosen them so that all specialisations are homomorphisms of pre $\lambda$-rings.  Note that, putting all the changes together, our convention that $\chi(\L^{-1/2})=-1$ is the same as theirs.
\end{remark}

\subsection{Power structures and motivic exponentials}\label{Sec:Power_Structures}

We recall the notion of a \emph{power structure} on a commutative ring $R$, mainly following \cite{GLMps,GLMHilb}.

\begin{definition}\label{def:power1026}
A \emph{power structure} on a ring $R$ is a map 
\begin{align*}
(1+tR\llbracket t\rrbracket)\times R&\ra 1+tR\llbracket t\rrbracket\\
(A(t),X)&\mapsto A(t)^X
\end{align*}
satisfying the following conditions: 
\begin{enumerate}
    \item $A(t)^0=1$,
    \item $A(t)^1=A(t)$, 
    \item $(A(t)\cdot B(t))^X=A(t)^X\cdot B(t)^X$, 
    \item $A(t)^{X+Y}=A(t)^X\cdot A(t)^Y$, 
    \item $A(t)^{XY}=(A(t)^X)^Y$, 
    \item $(1+t)^X=1+Xt+O(t^2)$, 
    \item $A(t)^X\big{|}_{t\ra t^k}=A(t^k)^X$.
\end{enumerate}  
\end{definition}

\begin{notation}
If $\alpha$ is a partition of an integer $n$, which we indicate $\alpha\vdash n$, by writing $\alpha=(1^{\alpha_1}\cdots i^{\alpha_i}\cdots r^{\alpha_r})$ we mean that there are $\alpha_i$ parts of size $i$, so that we recover $n$ as the sum $|\alpha| = \sum_ii\alpha_i$. The number of distinct parts of $\alpha$ is denoted $l(\alpha) = \sum_{i}\alpha_i$. The automorphism group of $\alpha$ is the product of symmetric groups $G_\alpha=\prod_i\mathfrak S_{\alpha_i}$. 
\end{notation}

Let us focus on $R = K_0(\Var_\C)$. If $X$ is a variety and $A(t)=1+\sum_{n>0}A_nt^n$ is a power series in $K_0(\Var_\C)\llbracket t\rrbracket$, we define
\be\label{power12}
(A(t))_{\pr}^{[X]}=1+\sum_{\alpha}\pi_{G_\alpha}\Biggl(\Biggl[\prod_i X^{\alpha_i}\setminus \Delta\Biggr]\cdot \prod_i A_i^{\otimes \alpha_i}\Biggr)t^{|\alpha|}.
\ee
In the above formula, $\Delta\subset \prod_i X^{\alpha_i}=X^{l(\alpha)}$  is the ``big diagonal'' (where at least two entries are equal), and the class
\[
\Biggl[\prod_i X^{\alpha_i}\setminus \Delta\Biggr]\cdot \prod_i A_i^{\otimes \alpha_i} \in \widetilde K_0^{G_\alpha}(\Var_\C)
\]
is $G_\alpha$-equivariant thanks to the  ``power map'' of Lemma \ref{lemma:majdgap}. 
Gusein-Zade, Luengo and Melle-Hern{\'a}ndez have proved \cite[Thm.~2]{GLMps} that there is a unique power structure on $K_0(\Var_{\C})$ for which the restriction to the case where all $A_i$ are effective is given by the formula (\ref{power12}), for every variety $X$.  
Moreover, by \cite[Thm.~1]{GLMps}, such a power structure is determined by the relation
\[
(1-t)^{-[X]}=\zeta_{[X]}(t)
\]
where
\be\label{eqn:kapranov1}
\zeta_{[X]}(t)=\sum_{n\geq 0}\,\bigl[\Sym^nX\bigr]\cdot t^n\in K_0(\Var_\C)\llbracket t\rrbracket
\ee
is the Kapranov motivic zeta function of $X$.  Since we always consider effective exponents when taking powers, we just recall the recipe for dealing with general $A(t)$ and effective exponent $[X]$. First, note that for any such $A(t)$ there is an effective $B(t)$ such that $A(t)\cdot B(t)=C(t)$ is effective. Then we have
\[
A(t)^{[X]} \defeq (C(t))^{[X]}_{\pr}\cdot ((B(t))^{[X]}_{\pr})^{-1},
\]
where both factors in the right hand side are defined via \eqref{power12}.
\begin{lemma}
\label{pr_inj_lemma}
Let $[X]\in K_0(\Var_{\mathbb{C}})$ be invertible in $K_0(\Staff{\C})$.  Then $(-)_{\pr}^{[X]}$ and $(-)^{[X]}$ are injective maps.
\end{lemma} 
\begin{proof}
By \cite[Rem.~3.7]{BenSven2} the power structure can be extended to $K_0(\Staff{\C})$, and so the second statement follows from $(A(t)^{[X]})^{[X]^{-1}}=A(t)$.  

Next we consider the first statement.  Assume that $A(t)_{\pr}^{[X]}=B(t)_{\pr}^{[X]}$. Write $A(t)=\sum_{i\geq 0} A_i t$ and $B(t)=\sum_{i\geq 0} B_i t$, where $A_0 = B_0 = 1$,  and assume that we have shown that $A_i=B_i$ for $i<n$.  Let $\alpha\vdash n$. Comparing the contributions from $\alpha$ in the $t^n$ coefficients of $A(t)_{\pr}^{[X]}$ and $B(t)_{\pr}^{[X]}$, by assumption they agree for $\alpha\neq (n)$, since these terms only involve $A_i$ and $B_i$ for $i<n$.  We deduce that the terms for $\alpha=(n^1)$ agree, and so $[X]\cdot A_n=[X]\cdot B_n$, and the result follows by injectivity of $ [X]\cdot$.
\end{proof}

As noted in \cite{BBS}, there is an extension of the power structure to $\mathcal{M}_{\C}$ uniquely determined by the substitution rules 
\[
A((-\mathbb{L}^{\frac{1}{2}})^nt)^{[X]}=A(t)^{(-\mathbb{L}^{\frac{1}{2}})^n[X]}=A(t)\big|_{t\mapsto(-\mathbb{L}^{\frac{1}{2}})^nt}.
\]

\subsubsection{Motivic exponential}\label{subsec:Exp}
It is often handy to rephrase motivic identities in terms of the motivic exponential, which is a group isomorphism\footnote{The group structures are the additive one on the source and the multiplicative one on the target.}
\[
\Exp\colon t\mathcal{M}_\C\llbracket t\rrbracket\,\widetilde{\ra}\,1+t\mathcal{M}_\C\llbracket t\rrbracket.
\]
Under \eqref{Iso_Power_series}, this can be seen as an inclusion of groups 
\[
K_0(\Var_{\mathbb N\setminus 0})\,\hookrightarrow\,K_0(\Var_{\mathbb N})^\times.
\]
First, define $\prExp=\sum_{n\geq 0}\prsigma^n$, relative to the monoid $(\mathbb{N},+)$.  Then if $A$ and $B$ are effective classes, we set 
\[
\Exp(A-B)=\prExp(A)\cdot \prExp(B)^{-1}.
\]
As in the proof of Lemma \ref{pr_inj_lemma}, $\prExp$ and $\Exp$ are injective.

Now if $(S, \nu\colon S\times S\rightarrow S)$ is a commutative monoid in the category of schemes, with a submonoid $S_+$ such that the induced map $\coprod_{n\geq 1}S_+^{\times n}\rightarrow S$ is of finite type, we similarly define
\[
\prExp_{\nu}(A)=\sum_{n\geq 0} \prsigma^n_\nu(A),
\]
and for $A$ and $B$ effective classes, we define
\[
\Exp_{\nu}(A-B)=\prExp_{\nu}(A)\cdot \prExp_{\nu}(B)^{-1}.
\]
The principal example will be 
\[
S = \Sym(U) = \coprod_{n\geq 0}\Sym^n(U)
\]
for $U$ a variety, and $S_+=\coprod_{n\geq 1}\Sym^n(U)$.  We define 
\[
\cup:\Sym(U)\times\Sym(U)\rightarrow \Sym(U)
\]
to be the morphism sending a pair of sets of unordered points with multiplicity to their union.

Note that $\Exp_{\nu}$ sends effective motives to effective motives, as the same is true of $\sigma_{\nu}^n$ for each $n$.

In order to recover a formal power series from a relative motive over $\Sym(U)$, we consider the operation 
\[
\#_!\left(\sum_{n\geq 0}\,\bigl[A_n\xrightarrow{f_n}\Sym^n(U)\bigr]\right)\defeq \sum_{n\geq 0}\, [A_n]t^n.
\]
In other words we take the direct image along the ``tautological'' map $\#\colon \Sym(U)\rightarrow\mathbb{N}$ which sends $\Sym^n(U)$ to the point $n$ --- recall that via \eqref{Iso_Power_series} we consider power series in $t$ with coefficients in $\mathcal{M}_{\C}$ as the same thing as elements of $\mathcal{M}_{\mathbb{N}}$.

\begin{prop}
\label{exp_prop}
Let $U,V$ be varieties.  Set $S=\Sym(U\times V)=\coprod_{n\geq 0}\Sym^n(U\times V)$, and for $i\in\mathbb{N}$ denote by $\cup\colon \Sym(U\times V)^i\rightarrow \Sym(U\times V)$ the map taking $i$ sets of points (with multiplicity) to their union (with multiplicity). Let 
\[
\tilde{\iota}_n\colon U\times\Sym^n(V)\rightarrow \Sym^n(U\times V)
\]
be the inclusion of the $n$-tuples $((u_1,v_1),\ldots,(u_n,v_n))$ such that $u_1=\cdots=u_n$.  Write $B=1+\sum_{n>0}B_n=\prExp_{\cup}\left(\sum_{n>0}A_n\right)=\prExp_{\cup}(A)$ for some set of $A_n,B_n\in K(\mathrm{Var}_{\Sym^n(V)})$.  Define the $S$-motive
\[
\mathsf Z=\sum_{n\geq 0}\sum_{\alpha \vdash n}\cup_!\pi_{G_{\alpha}}j^*_{\alpha}\left(\underset{i\lvert\alpha_i\neq 0}{\Boxtimes}\tilde{\iota}_{i,!}\bigl([U\xrightarrow{\id} U]\boxtimes B_i\bigr)^{\otimes \alpha_i}\right)
\]
where $j_{\alpha}$ is the $G_{\alpha}$-equivariant inclusion from the space of points in $\prod_{i\lvert \alpha_i\neq 0}\Sym^i(U\times V)^{\alpha_i}$ that are not sent to the big diagonal after projection to $\prod_{i\lvert \alpha_i\neq 0}\Sym^i(U)^{\alpha_i}$. Then
\[
\mathsf Z=\prExp_{\cup}\left(\sum_{n>0} \tilde{\iota}_{n,!}\bigl([U\xrightarrow{\id} U]\boxtimes A_n\bigr)  \right).
\]
and if $A$ is effective one has
\[
\#_!\mathsf Z=(\#_!B)^{[U]}.
\]
\end{prop}

\begin{proof}
The second statement follows directly from the definition of the power structure.  The first arises from the decomposition of the right hand side according to incidence partition in the $U$ factor.
\end{proof}

\subsection{The virtual motive of a critical locus}\label{sec:900}
Let $X$ be a complex scheme of finite type, and let $\nu_X\colon X(\C)\ra \Z$ be the canonical constructible function introduced by Behrend \cite{Beh}. The \emph{weighted} (or \emph{virtual}) \emph{Euler characteristic} of $X$ is defined via $\nu_X$ as
\[
\widetilde\chi(X) = \int_X\nu_X\,\mathrm{d}\chi =
\sum_{r\in\Z}r\cdot\chi(\nu_X^{-1}(r)).
\]
When $X$ is a proper moduli space of stable sheaves on a Calabi--Yau $3$-fold, this number agrees with the Donaldson--Thomas invariant of $X$ by the main result of \emph{loc.~cit}.
The following definition is central to this paper.

\begin{definition}[{\cite{BBS}}]\label{def:vm}
A \emph{virtual motive} of a scheme $X$ is a motivic weight $\xi\in \mathcal M_\C^{\hat\mu}$ such that $\chi(\xi) = \widetilde\chi(X)$.
\end{definition}

\begin{definition}
A scheme $X$ is a \emph{critical locus} if it is of the form 
\[
X = \crit(f) = Z(\mathrm{d} f),
\]
where $f\colon U\ra \A^1$ is a regular function on a smooth scheme $U$.
\end{definition}

The Behrend function of a critical locus $X = \crit(f) \subset U$ agrees with the Milnor function $\mu_f$, the function counting the number of vanishing cycles \cite[Cor. 2.4 (iii)]{APPP1}. In particular, $\nu_X(x) = (-1)^{\dim U-1}(\chi(\textrm{MF}_{f,x})-1)$, where $\textrm{MF}_{f,x}$ is the Milnor fibre of $f$ at $x$. More globally, one can write
\[
\nu_X = \chi\left(\Phi_f[\dim U-1]\right),
\]
where $\Phi_f[\dim U-1]\in \mathrm{Perv}(X)$ is the perverse sheaf of vanishing cycles, the image of the constant perverse sheaf $\underline{\mathbb{Q}}_U[\dim U]$ under the vanishing cycle functor $\varphi_f[-1]\colon \mathrm{Perv}(U)\ra \mathrm{Perv}(U_0)$. Here $U_0=f^{-1}(0)$ denotes the hypersurface determined by $f$.
The pair $(U,f)$ also determines a canonical \emph{relative} virtual motive
\be\label{def:relvm}
\MF_{U,f} = \L^{-\frac{d}{2}}\left[-\phi_f\right]_X\in\mathcal M_{X}^{\hat\mu} \subset \mathcal{M}^{\hat{\mu}}_{U}
\ee
where $d = \dim U$ and $[\phi_f]_X$ is the relative motivic vanishing cycle class introduced by Denef and Loeser \cite{DenefLoeser1}. It is a class in $K_0^{\muhat}(\Var_{U_0})$, supported on $X=\crit(f)$, that we view as an element of $\mathcal M_{X}^{\hat\mu}$. We write $[\phi_f]$ for the pushforward of $[\phi_f]_X$ to a point.  We will repeatedly use the following proposition, due to Bittner \cite{Bittner05}.
\begin{prop}
\label{BittnerProp}
Let $G$ be a finite group acting freely on a smooth variety $U$, let $q\colon U\rightarrow U/G$ be the quotient map, and let $f$ be a regular function on $U/G$.  Then 
\begin{enumerate}
    \item 
    There is a well defined \textit{equivariant} motivic vanishing cycle $[\phi_{fq}]^G_{U_0}\in K_0^{G\times \muhat}(\Var_{U_0})$ such that the relative motive in $K_0^{\muhat}(\Var_{U_0})$ induced by forgetting the $G$-action is $[\phi_{fq}]_{U_0}$.
    \item
    There is an equality of motives 
    \[
    \pi_{G}\left([\phi_{fq}]^G_{U_0}\right)=[\phi_{f}]_{U_0/G} \in K_0^{\hat\mu}(\Var_{U_0/G}).
    \]
\end{enumerate}
\end{prop}

\begin{notation}\label{notation:vir}
If $X = \crit(f)\rightarrow Y$ is a morphism of varieties, we denote by
\[
\left[\crit(f)\rightarrow Y\right]_{\vir}=(\crit(f)\rightarrow Y)_!\MF_{U,f}\in\mathcal{M}^{\hat{\mu}}_Y
\]
the induced \textit{relative} virtual motive.
More generally, if $\iota\colon Z\hookrightarrow \crit(f)$ is a locally closed subscheme and $Z\ra Y$ is a morphism, we let
\[
[Z \ra Y]_{\vir} = (Z \ra Y)_!\iota^\ast \MF_{U,f} \in \mathcal{M}^{\hat{\mu}}_Y.
\]
When $Y=\Spec \C$ we denote $[Z\ra \Spec \C]_{\vir}$ simply by $[Z]_{\vir}$.
\end{notation}

Since the fibrewise Euler characteristic of $\MF_{U,f}$ equals $\nu_X$ as a function on $X$ \cite[Prop.~2.16]{BBS}, the absolute class
\be\label{eqn:kjadhjka}
[X]_{\vir} = \int_X \MF_{U,f} = \L^{-\frac{d}{2}}\cdot \left[-\phi_f\right]\in\mathcal M_{\C}^{\hat\mu}
\ee
is a virtual motive for $X$ in the sense of Definition \ref{def:vm}. 

\begin{remark}
A critical locus $X = \crit(f)\subset U$ has a canonical virtual fundamental class $[X]^{\vir}\in A_0X$, attached to the symmetric perfect obstruction theory determined by the Hessian of $f$.
When $X$ is proper, Behrend's theorem \cite{Beh} can be phrased as
\[
\chi [X]_{\vir} = \int_{[X]^{\vir}}1 \in \Z.
\]
\end{remark}

\begin{remark}\label{rmk:smvm}
If $X$ is a smooth scheme, it can be considered as a critical locus via the zero function $f = 0 \in  \Gamma(\O_X)$. The associated virtual motive is
\[
[X]_{\vir} = \L^{-\frac{\dim X}{2}}\cdot [X]\in \mathcal M_\C.
\]
Via the stated sign conventions, we see that $\chi [X]_{\vir}$ recovers the virtual Euler characteristic of the smooth scheme $X$, namely $\widetilde\chi(X) = (-1)^{\dim X}\chi(X)$.
\end{remark}

We end this subsection with two results that are of crucial importance in calculations involving motivic vanishing cycles. Recall from Section \ref{sec:Monodromic} that the groups $\mathcal M_{S}^{\hat\mu}$ carry the convolution product `$\star$' besides the ordinary product.

\begin{theorem}[Motivic Thom--Sebastiani \cite{DenefLoeser2,LooijengaMM}]\label{mtseb190}
Let $f\colon U\ra \A^1$ and $g\colon V\ra \A^1$ be regular functions on smooth  varieties $U$ and $V$. Consider the function $f\oplus g\colon U\times V\ra \A^1$ given by $(x,y)\mapsto f(x)+g(y)$.
Let $i\colon U_0\times V_0\ra (U\times V)_0$ be the inclusion of the zero fibres, and let $p_U$ and $p_V$ be the projections from $U_0\times V_0$. Then one has
\[
i^\ast\bigl[\phi_{f\oplus g}\bigr]_{(U\times V)_0}=p_U^\ast\bigl[\phi_f\bigr]_{U_0}\star p_V^\ast\bigl[\phi_g\bigr]_{V_0}\in \mathcal M_{U_0\times V_0}^{\hat\mu}.
\]
\end{theorem}

The following result will be used in Propositions \ref{thmequiva792} and \ref{prop:mcv10}.

\begin{theorem}[{\cite[Thm.~B.1]{BBS}}]\label{thm:equivfam738}
Let $f\colon U\ra \A^1$ be a regular function on a smooth complex quasi-projective  variety, with critical locus $X$. 
Assume $U$ is acted on by a connected complex torus $\mathbb T$ in such a way that $f$ is
$\mathbb T$-equivariant with respect to a primitive character $\chi\colon \mathbb T\ra \G_m$. 
\bitem
\item [\normalfont{(i)}] If there is a one parameter subgroup $\G_m\subset \mathbb T$ such that the induced action is circle compact, then
\[
\bigl[\phi_f\bigr]=\bigl[f^{-1}(1)\bigr]-\bigl[f^{-1}(0)\bigr]\in\mathcal M_\C\subset \mathcal M_{\C}^{\hat\mu}.
\]
\item [\normalfont{(ii)}] Let $\tau \colon X\ra Y$ be a map to an affine variety. 
If, in addition to the assumption in $\mathrm{(i)}$, the hypersurface $f^{-1}(0)\subset U$ is reduced, then the relative class $[\phi_f]_{Y}=\tau_![\phi_f]_{X}$ lies in the subring $\mathcal M_{Y}\subset \mathcal M_{Y}^{\hat\mu}$ of classes with trivial monodromy.
\eitem
\end{theorem}
\begin{remark}
The original statement of this theorem in \cite{BBS} fixed $Y$ to be the affinisation of $X$ --- the statement above then follows from the fact that $\tau$ must factor through the affinisation, and the direct image of a monodromy-free motive is monodromy-free.
\end{remark}

\subsection{The virtual motive of the Hilbert scheme of points}\label{Section:BBS}
Quivers with potentials provide a large class of examples of critical loci. For instance, consider the framed $3$-loop quiver $Q_{\mathrm{BBS}}$ (studied by Behrend--Bryan--Szendr\H{o}i) depicted in Figure \ref{3LoopQuiver}. The arrow $1 \ra \infty$ is called a \emph{framing}, and $\infty$ is the \textit{framing vertex}.  Throughout the paper, the vertices of a framed quiver are ordered so that the framing vertex is last.

\begin{figure}[ht]
\begin{tikzpicture}[>=stealth,->,shorten >=2pt,looseness=.5,auto]
  \matrix [matrix of math nodes,
           column sep={3cm,between origins},
           row sep={3cm,between origins},
           nodes={circle, draw, minimum size=7.5mm}]
{ 
|(A)| \infty & |(B)| 1 \\         
};
\tikzstyle{every node}=[font=\small\itshape]
\path[->] (B) edge [loop above] node {$x$} ()
              edge [loop right] node {$y$} ()
              edge [loop below] node {$z$} ();            
\draw (B) -- (A) node {};
\end{tikzpicture}
\caption{The framed $3$-loop quiver $Q_{\mathrm{BBS}}$.}\label{3LoopQuiver}
\end{figure}
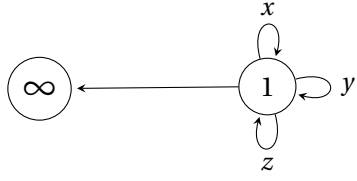

The space of $(n,1)$-dimensional right $\mathbb{C}Q_{\mathrm{BBS}}$-modules is the affine space $\End(\C^n)^3\times \C^n$ parameterising triples of $n\times n$ matrices $(A,B,C)$ and vectors $v\in \C^n$. Consider the potential $W = x[y,z]$, viewed as an element of the path algebra $\C\langle x,y,z\rangle$ of the (unframed) $3$-loop quiver. Then by \cite[Prop.~3.1]{BBS} one has, as schemes,
\[
\Hilb^n(\A^3) = \crit(\Tr W) \subset \NCHilb^n,
\]
where $\NCHilb^n$ is the \emph{non-commutative Hilbert scheme}, defined as follows. The open subscheme
\[
V_n \subset \End(\C^n)^3 \times \C^n = \Rep_{(n,1)}(\mathbb{C}Q_{\mathrm{BBS}}),
\]
parameterising tuples $(A,B,C,v)$ such that $v\in \C^n$ generates the $\C\langle x,y,z\rangle$-module defined by the triple $(A,B,C)$, carries a free $\GL_n$-action, and $\NCHilb^n = V_n/\GL_n$ is a smooth quasi-projective variety of dimension $2n^2+n$. The generating function $\mathsf Z_{\A^3}(t)$ for which the $t^n$ coefficient is the virtual motive
\[
\left[\Hilb^n(\A^3)\right]_{\vir} = \L^{-\frac{2n^2+n}{2}}\left[-\phi_{\Tr W}\right] \in \mathcal M_{\C}
\]
was computed in \cite[Thm.~3.7]{BBS}. The result is the equation
\[
\mathsf Z_{\A^3}(t) = \prod_{m\geq 1}\prod_{k=0}^{m-1}\,\bigl(1-\L^{k+2-\frac{m}{2}}t^m\bigr)^{-1} \in \mathcal M_{\C}\llbracket t \rrbracket.
\]
Let 
\be\label{Punctual_Hilb_scheme}
\left[\Hilb^n(\A^3)_0\right]_{\vir} = \int_{\Hilb^n(\A^3)_0} \iota^\ast \MF_{\NCHilb^n,\Tr W}\in \mathcal M_{\C}
\ee
be the virtual motive of the \emph{punctual Hilbert scheme} $\iota\colon \Hilb^n(\A^3)_0\hookrightarrow \Hilb^n(\A^3)$ (cf.~Notation \ref{notation:vir}), the closed subscheme parameterising subschemes entirely supported at the origin $0\in \A^3$. Then the generating series
\[
\mathsf Z_0(t) = \sum_{n\geq 0}\,\left[\Hilb^n(\A^3)_0\right]_{\vir}\cdot t^n
\]
satifies the relation
\begin{equation}
    \label{punctual_BBS}
\mathsf Z_0(-t) = \Exp\left(\frac{-\L^{-\frac{3}{2}}t}{\bigl(1+\L^{-\frac{1}{2}}t\bigr)\bigl(1+\L^{\frac{1}{2}}t\bigr)} \right).
\end{equation}
\begin{remark}
\label{punctualBBS_effective}
As a corollary of Formula $(\ref{punctual_BBS})$ the motive $(-1)^n\left[\Hilb^n(\A^3)_0\right]_{\vir}$ is effective, i.e.~it belongs to the sub semigroup $S_0(\Var_{\C})\subset \mathcal M_{\C}$.
\end{remark}

Behrend, Bryan and Szendr\H{o}i also define a virtual motive $[\Hilb^nX]_{\vir}$ for arbitrary smooth $3$-folds. The motivic partition function
\[
\mathsf Z_X(t) = \sum_{n\geq 0} \,\left[\Hilb^nX\right]_{\vir} \cdot t^n \in \mathcal M_{\C}\llbracket t \rrbracket
\]
is again fully determined by the punctual contributions, i.e.~by \cite[Prop.~4.2]{BBS} one has
\begin{equation}\label{Hilb_Partition_Function}
\mathsf Z_X(-t) = \mathsf Z_0(-t)^{[X]}.
\end{equation}

\begin{remark}[Related work on Quot schemes]
The identity \eqref{Hilb_Partition_Function}, as well as its reformulation in terms of the motivic exponential, has been generalised  in \cite{Ob_Mot} to the case of Quot schemes $\Quot_Y(F,n)$ where $F$ is an arbitrary locally free sheaf on a smooth $3$-fold. See also \cite{mot_quot} for the ``non-virtual'' setup. In higher rank, the starting point of motivic DT theory is the observation that $\Quot_{\mathbb A^3}(\mathscr O^{\oplus r},n)$ is a global critical locus \cite[Theorem 2.6]{BR18}. This has also been exploited to prove a plethystic formula (the Awata--Kanno conjecture in String Theory) for the partition function of higher rank K-theoretic DT invariants \cite{FMR_K-DT}.
\end{remark}


\section{The local model as a critical locus}\label{CriticalLocus}
For a smooth curve $C$ embedded in a smooth $3$-fold $Y$ with ideal sheaf $\mathscr I_C\subset \O_Y$, we let
\[
Q^n_C = \Quot_Y(\mathscr I_C,n) = \Set{\mathscr I_C\onto \mathscr F | \dim (\Supp \mathscr F) = 0,\,\chi(\mathscr F)=n}
\]
denote the Quot scheme of $n$ points of $\mathscr I_C$. Given a surjection $\theta\colon\mathscr I_C\onto \mathscr F$, we can consider its kernel $\mathscr I_Z\subset \mathscr I_C$, and thus think of $[\theta]\in Q^n_C$ as a closed one-dimensional subscheme $Z\subset Y$ containing $C$ as its maximal purely one-dimensional subscheme. We will switch between these two interpretations of $Q^n_C$ without further comment. Note that we generally suppress the ambient $3$-fold $Y$ from the notation.

\begin{remark}
When $Y$ is projective, the association $[\theta]\mapsto \ker\theta$ defines a closed immersion into the moduli space of ideal sheaves
\[
Q^n_C\hookrightarrow I_{\chi(\O_C)+n}(Y,[C]),
\]
as proved in \cite[Lemma 5.1]{LocalDT}. This closed immersion can be generalised to higher rank sheaves, see \cite[Prop.~5.1]{BR18}.
\end{remark}

Let now $L\subset\mathbb{A}^3$ be a line; for concreteness set $\mathscr{I}_L=(x,y)\subset \mathbb{C}[x,y,z]$. The scheme $Q^n_L$ parameterises surjections $\mathscr{I}_L\onto N$ of $\mathbb{C}[x,y,z]$-modules, where $N$ is $n$-dimensional as a $\C$-vector space.  As such there is a forgetful map from $Q_L^n$ to the stack of $n$-dimensional $\mathbb{C}[x,y,z]$-modules, and postcomposing with the affinization map for this stack, a ``Hilbert-to-Chow'' morphism
\begin{equation}
\label{HC_def}
\HC_{\A^3}\colon Q_L^n\rightarrow \sym^n(\mathbb{A}^3).
\end{equation}

This map is a special case of \cite[Cor.~$7.15$]{Rydh1}.

The goal of this section is to prove the following result, which is a one-dimensional counterpart of the analogous statement for $\Hilb^n(\A^3)$, considered in \cite{BBS}.  Part (A) is Theorem \ref{thm:thm1} from the introduction.

\begin{theorem}\label{mainthm1}
Consider the embedding $L\subset \mathbb A^3$ from above. Then 
\bitem
\item [(A)] the Quot scheme $Q^n_{L}$ is a global critical locus, i.e.~there is a smooth variety $U$ and a function $f\colon U\ra \A^1$ such that $Q_L^n\cong \crit(f)$.
\item [(B)] the relative motive $[\phi_f]_{\Sym^n(\mathbb{A}^3)}$ is an element of the subring $\mathcal M_{\sym^n(\mathbb{A}^3)}\subset \mathcal M_{\sym^n(\mathbb{A}^3)}^{\hat\mu}$.
\eitem
\end{theorem}

Part (A) is proved at the end of Section \ref{conifold_section}, and part (B) is proved in Proposition \ref{prop:mcv10}.  The main step in the proof consists of realising $Q^n_L$ as a suitable open 
subscheme of the Quot scheme $Q^n_{C_0}$, where $C_0\subset X$ is the exceptional curve in the resolved conifold $X$. In the next subsection we review the critical locus structure on $Q^n_{C_0}$, and more generally, the noncommutative Donaldson--Thomas theory of the conifold as introduced in \cite{MR2403807}.

\subsection{Conifold Geometry}
\label{conifold_section}
Given a quiver $Q = (Q_0,Q_1)$ with potential $W$ and a field $K$ we define $K(Q,W)$ to be the associated Jacobi algebra, i.e.~the quotient of the free path algebra $KQ$ by the noncommutative derivatives $\partial W/\partial a$ for $a\in Q_1$ ranging over the arrows of $Q$.  

Given a dimension vector $\dd = (\dd(i))_{i\in Q_0}\in \mathbb{N}^{Q_0}$, we define
\[
\Rep_{\dd}(\C Q)=\prod_{(i\ra j)\in Q_1}\Hom(\C^{\dd(j)},\C^{\dd(i)}) 
\]
to be the affine space of $\dd$-dimensional right $\C Q$-modules.  Given a King stability condition $\zeta\in\mathbb{Q}^{Q_0}$ we denote by 
\[
\Rep_{\dd}^{\zeta}(\C Q)\subset \Rep_{\dd}(\C Q)
\]
the open subscheme of $\zeta$-semistable $\C Q$-modules.  Both schemes are acted on by the gauge group 
\[
\GL_{\dd}\defeq \prod_{i\in Q_0}\GL_{\dd(i)}
\]
by change of basis.

\smallbreak
Throughout the paper, we let $X$ denote the \emph{resolved conifold}, namely the total space of the rank $2$ locally free sheaf $\O_{\P^1}(-1)\oplus\O_{\P^1}(-1)$. We denote by $C_0\subset X$ the zero section of this vector bundle, so $C_0\cong \mathbb{P}^1$.
Let 
\[
\pi\colon X\ra C_0
\]
denote the projection onto the zero section. The local Calabi--Yau $3$-fold $X$ is the crepant resolution of the conifold singularity 
\[
\Spec \C[x,y,z,w]/(xy-zw)\subset \A^4,
\]
and $C_0$ is the exceptional curve, the only strictly positive dimensional proper subvariety of $X$. Since $C_0$ is a \emph{rigid} smooth rational curve in $X$, the Hilbert scheme $I_{n+1}(X,[C_0])$ of one-dimensional subschemes $Z\subset X$ with homology class $[C_0]$ and $\chi(\O_Z) = n+1$ agrees (as a scheme) with the Quot scheme $Q^n_{C_0}$. This space was realised as a global critical locus by Nagao and Nakajima \cite{NagNak}, by studying representations of the Jacobi algebra $\mathbb{C}(\tilde{Q}_{\con},W_{\con})$. Here, $\tilde{Q}_{\con}$ is the (framed) \emph{conifold quiver} depicted in Figure \ref{FramedConifoldQuiver}, and 
\be\label{KWpotential}
W_{\con} = a_1b_1a_2b_2-a_1b_2a_2b_1
\ee
is the Klebanov--Witten potential \cite{KW1}. We denote by $Q_{\con}$ the quiver obtained by removing from $\tilde{Q}_{\con}$ the framing vertex $\infty$ and the arrow $\iota$.  Since we consider right modules, a $K\tilde{Q}_{\con}$-module $\tilde{\rho}$ is defined by the following data.
\begin{enumerate}
    \item A right $KQ_{\con}$-module $\rho$, with $K Q_{\con}$-action defined via the $K\tilde{Q}_{\con}$-action on $\tilde{\rho}$ and the inclusion $KQ_{\con}\hookrightarrow K\tilde{Q}_{\con}$.
    \item A linear map $V\rightarrow \rho_1$, with $V=\tilde{\rho}_\infty$.
\end{enumerate}

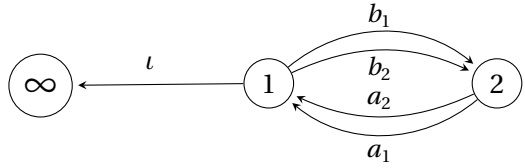
\begin{figure}[ht]
\begin{tikzpicture}[>=stealth,->,shorten >=2pt,looseness=.5,auto]
  \matrix [matrix of math nodes,
           column sep={3cm,between origins},
           row sep={3cm,between origins},
           nodes={circle, draw, minimum size=3.5mm}]
{ 
|(I)| \tiny{\infty} & |(A)| 1 & |(B)| 2 \\         
};
\draw (A) -- (I) node {}; 
\tikzstyle{every node}=[font=\small\itshape]
\node [anchor=west,right] at (1.2,0.93) {$b_1$};            
\node [anchor=west,right] at (1.2,0.23) {$b_2$};            
\node [anchor=west,right] at (1.2,-0.2) {$a_2$};
\node [anchor=west,right] at (1.2,-0.9) {$a_1$}; 
\node [anchor=west,right] at (-1.7,0.3) {$\iota$};

\draw (A) to [bend left=25,looseness=1] (B) node {};
\draw (A) to [bend left=40,looseness=1] (B) node {};
\draw (B) to [bend left=25,looseness=1] (A) node {};
\draw (B) to [bend left=40,looseness=1] (A) node {};
\end{tikzpicture}
\caption{The framed conifold quiver $\tilde{Q}_{\con}$.}\label{FramedConifoldQuiver}
\end{figure}

\begin{remark}
The orientation of $\tilde{Q}_{\con}$ differs from Figure 4 of \cite{NagNak}, but our notion of framing is the same as \emph{loc.~cit.} (see (1) and (2) above) --- note that we are considering \emph{right} $\mathbb{C}\tilde{Q}_{\con}$-modules throughout --- see Remark \ref{OrientationRemark}.
\end{remark}

We identify quasicoherent $\mathscr{O}_X$-modules with triples $(\mathscr{F},\alpha_1,\alpha_2)$ where $\mathscr{F}\in\QCoh(\mathbb{P}^1)$ and $\alpha_1$, $\alpha_2\in\Hom_{\mathscr{O}_{\mathbb{P}^1}}(\mathscr{F},\mathscr{F}(-1))$ commute in the sense that $\alpha_1(-1)\circ \alpha_2=\alpha_2(-1)\circ \alpha_1$.  Then the above noncommutative conifold is the natural enhancement of Beilinson's noncommutative $\mathbb{P}^1$: given a complex of quasicoherent sheaves $\mathscr{F}$ on $\mathbb{P}^1$ we set $\rho_1$ and $\rho_2$ to be the complexes of vector spaces $\RHom_{\mathscr{O}_{\mathbb{P}^1}}(\mathscr{O}_{\mathbb{P}^1},\mathscr{F})$ and $\RHom_{\mathscr{O}_{\mathbb{P}^1}}(\mathscr{O}_{\mathbb{P}^1}(1),\mathscr{F})$.  We let 
\[
\rho(b_1),\rho(b_2)\colon \RHom_{\mathscr{O}_{\mathbb{P}^1}}(\mathscr{O}_{\mathbb{P}^1}(1),\mathscr{F})\rightarrow \RHom_{\mathscr{O}_{\mathbb{P}^1}}(\mathscr{O}_{\mathbb{P}^1},\mathscr{F})
\]
be the maps induced by the two sections $x$ and $y$ of  $\mathscr{O}_{\mathbb{P}^1}(1)$. We set $\rho(a_i)=\alpha_i$ for $i=1,2$, and it is easy to check that the commutativity conditions are given precisely by the superpotential relations for $W_{\con}$, so that we obtain in this way a right $\mathbb{C}(Q_{\con},W_{\con})$-module.  

This description of the noncommutative conifold makes the translation of various geometrically defined functors rather transparent.  For instance, the direct image along the projection map $\pi_*\colon\Db{\QCoh(X)}\rightarrow \Db{\QCoh(\P^1)}$ becomes the forgetful map from the noncommutative conifold to Beilinson's noncommutative $\mathbb{P}^1$, forgetting the action of the arrows $a_1,a_2$, and the direct image along the inclusion $C_0\hookrightarrow X$ becomes the extension by zero functor.

Let 
\[
\mathscr{E}=\pi^*(\mathscr{O}_{\mathbb{P}^1}\oplus\mathscr{O}_{\mathbb{P}^1}(1))
\]
and let 
\begin{align*}
A_{\con}&=\End_{\mathscr{O}_X}(\mathscr{E})
\\&\cong \mathbb{C}(Q_{\con},W_{\con}).
\end{align*} 
We denote by 
\[
\Phi=\RHom(\mathscr{E},\bullet)\colon \Db{\Coh X}\,\widetilde{\rightarrow}\, \Db{A_{\con}\textrm{-mod}} 
\]
the equivalence of derived categories, where on the right hand side of the equivalence we have the derived category of right $A_{\con}$-modules with finitely generated total cohomology.  We denote by 
\[
\Psi\colon M\rightarrow M\overset{\mathbf{L}}{\otimes}_{A_{\con}} \mathscr{E}
\]
the quasi-inverse.
\begin{remark}\label{OrientationRemark}
As one sees from the above, the general setup leads us to consider \textit{right} $A_{\con}$-modules.  On the other hand, for an arbitrary quiver $Q$, there is an equivalence of categories between $K$-linear $Q$-representations and \textit{left} $KQ$-modules.  Since we consider right modules (as in \cite{NagNak}), if one wants to think of modules over algebras such as $A_{\con}$ as quiver representations, one should reverse the orientation of the underlying quiver.
\end{remark}
The chamber decomposition of the space of stability parameters for $\tilde{Q}_{\con}$ was worked out by Nagao and Nakajima in \cite{NagNak}, where the DT and PT chambers were precisely characterised. For a generic stability condition
\[
\zeta=\left(\zeta_1,\zeta_2,-(\zeta_1(n+1)+\zeta_2n)\right)\in \mathbb R^3
\]
in the ``DT region'' for $X$, defined by the conditions $\zeta_1<\zeta_2$ and $\zeta_1+\zeta_2<0$, we consider the moduli space
\begin{equation}
    \label{Ndef}
\mathcal N_n \defeq \Rep^\zeta_{(n+1,n,1)}(\tilde Q_{\con})/\GL_{n+1}\times \GL_n
\end{equation}
of $\zeta$-stable framed representations of $Q_{\con}$, having dimension vector $(n+1,n,1)$. Here the dimension vector $(n+1,n,1)$ refers to the vertices ordered as $(1,2,\infty)$. 

The work of Nagao--Nakajima then implies that $Q^n_{C_0}$ is isomorphic to the subscheme of $\mathcal N_n$ defined by the defining relations of $A_{\con}$.  Since these are exactly the noncommutative derivatives of the Klebanov--Witten potential \eqref{KWpotential}, it follows from \cite[Prop.~3.8]{Seg} that $Q^n_{C_0}$ is identified with the critical locus of the function 
\[
g_n\colon \mathcal N_n\ra \A^1
\]
given by taking the trace of \eqref{KWpotential}.

Consider the open subset 
\[
\mathcal N_n^\circ\subset \mathcal N_n
\]
parameterising stable representations $\rho$ such that the linear map $\rho(b_2)\colon\C^n\ra \C^{n+1}$ is injective, and let 
\[
f_n\colon \mathcal N_n^\circ \ra \A^1
\]
be the restriction of the function $g_n$. We now prove that
\[
Q^n_L \cong \crit(f_n) \subset \mathcal N_n^\circ.
\]

\begin{proofof}{Theorem \ref{mainthm1}\,(A)}
For skyscraper sheaves of points $\mathscr{O}_x$, corresponding to representations of dimension vector $(1,1)$ under $\Phi$, injectivity of $\rho(b_2)$ corresponds to the condition that $\pi(x)\neq\infty$.  We identify $\mathbb{A}^3$ with $X\setminus \pi^{-1}(\infty)$, and we have the following Cartesian diagram
\[
\begin{tikzcd}
L\MySymb{dr}\arrow[hook]{r}{}\arrow[hook]{d}{}&\mathbb{A}^3\arrow[hook]{d}{u}\\
C_0\arrow[hook]{r}{}&X
\end{tikzcd}
\]
where the horizontal maps are closed inclusions, and vertical maps open inclusions.  

Let $\mathscr{L}$ be the tautological $\mathbb{C}(Q_{\con},W_{\con})\otimes \mathscr{O}_{\crit(g_n)}$-module.  Let $\mathscr{G}$ be the submodule generated by
\[
\rho(b_1)\mathscr{L}+\rho(b_2)\mathscr{L}.
\]
Consider the exact sequence 
\[
0\rightarrow \mathscr{G}\rightarrow \mathscr{L}\rightarrow V\rightarrow 0
\]
of $\mathbb{C}(Q_{\con},W_{\con})\otimes \mathscr{O}_{\crit(g_n)}$-modules.  Let $\Spec(K)\hookrightarrow \mathcal{N}_n$ be a geometric point of $\crit(g_n)$, corresponding to a $K(Q_{\con},W_{\con})$-module $\rho$.  By the above-mentioned result of Nagao--Nakajima, $\Psi(\rho)$ corresponds to a $K$-point of $Q_C^n$, and so admits a unique (up to scalar) surjective map to $\mathscr{O}_{C_0}\otimes K$, with kernel a coherent sheaf $\mathscr{F}$ with zero-dimensional support.  It follows that $\rho$ admits a unique (up to scalar) nonzero map to the nilpotent simple module at vertex $1$, with kernel isomorphic to $\Phi(\mathscr{F})$.  In particular, the space spanned by the image of $\rho(b_1)$ and $\rho(b_2)$ is $n$-dimensional.  It follows that $\dim(V_K)=1$ for all $K$-points, and so $V$ is a locally free $\O_{\crit(g_n)}$-module of rank 1, and thus  $\mathscr{G}$ is locally free of rank $(n,n)$.  

Let $\Spec(K)\hookrightarrow \crit(g_n)$ be the inclusion of a point, and let $i\colon X\times \Spec(K)\hookrightarrow X\times \crit(g_n)$ be the induced inclusion.  The $\mathscr{O}_{X\times\crit(g_n)}$-module $\Psi(\mathscr{G})$ is a coherent sheaf after pullback along $i$, and so is a coherent sheaf by Lemma 4.3 of  \cite{BriFMT}.  The coherent sheaf $\Psi(\mathscr{L})$ is equipped with a tautological map from $\mathscr{O}_{X\times\crit(g_n)}$, surjective by stability, inducing the tautological surjective map 
\[
\mathtt{fr}\colon \mathscr{I}_{C_0}\boxtimes \mathscr{O}_{\crit(g_n)}\onto \mathscr{G}.
\]

The condition on $\rho(b_2)$ defining $\mathcal{N}_n^{\circ}$ implies that each $i^*\Psi(\mathscr{G})$ has support away from $\pi^{-1}(\infty)$.  Since $\crit(f_n)=\crit(g_n)\cap \mathcal{N}_n^{\circ}$, the inverse image sheaf of $\Psi(\mathscr{G})$ on $\crit(f_n)\times\mathbb{A}^3$ is a vector bundle on $\crit(f_n)$, equipped with a surjection from $\mathscr{I}_L\boxtimes\mathscr{O}_{\crit(f_n)} $.  This defines the map $\crit(f_n)\rightarrow Q_L^n$.  The inverse is defined similarly: given a family of surjections $\mathscr{I}_L\onto \mathscr{F}$ we obtain a family of surjections $\mathscr{I}_{C_0}\rightarrow u_*\mathscr{I}_L\ra u_*\mathscr{F}$.
\end{proofof}

This completes the proof of Theorem \ref{thm:thm1} from the Introduction.

\subsection{Relative virtual motives}
Via the coherent sheaf $\Psi(\mathscr{G})$ of $\O_{X\times \crit(g_n)}$-modules constructed in the proof we obtain the map
\be\label{map:qc}
\HC_{X}\colon Q_{C_0}^n\rightarrow \Sym^n(X)
\ee
extending the map \eqref{HC_def}. It is again a special case of the Quot-to-Chow map \cite[Cor.~$7.15$]{Rydh1}.  In particular, we can write
\begin{equation}
\label{HCX_desc}
Q_L^n \cong Q_{C_0}^n\times_{\Sym^n(X)} \Sym^n(\mathbb{A}^3)
\end{equation}
via the map $\HC_X$ and the inclusion $u\colon \A^3\ra X$.  Via projection to the $\Sym^n(\mathbb{A}^3)$-factor, we recover the map
\[
\HC_{\mathbb{A}^3}\colon Q_L^n\rightarrow \Sym^n(\mathbb{A}^3).
\]
Where it is clear which of the two Hilbert--Chow maps we mean, we will drop the subscript.  
\begin{remark}
\label{hcRem}
In a little more detail, at a $K$-valued point of $Q_{C_0}^n$, the corresponding $K(Q_{\con},W_{\con})$-module admits a filtration by $(1,1)$-dimensional $(\zeta_1,\zeta_2)$-stable $K(Q_{\con},W_{\con})$-modules, i.e.~quadruples $(\alpha_1,\alpha_2,\beta_1,\beta_2)\in\mathbb{C}^4$ such that $\beta_1\neq 0$ or $\beta_2\neq 0$, modulo the equivalence relation 
\[
(\alpha_1,\alpha_2,\beta_1,\beta_2)\sim (z\cdot \alpha_1,z\cdot \alpha_2,z^{-1}\cdot \beta_1,z^{-1}\cdot\beta_2)
\]
for $z\in\mathbb{C}^*$.  This is the fine moduli space of point sheaves on $X$, and we identify it with $X$.  Then taking the union of the supports of the subquotients in the filtration gives the map \eqref{map:qc} to $\Sym^n(X)$. 
\end{remark}

Theorem \ref{mainthm1} (A) has the following immediate consequence.

\begin{corollary}\label{cor:virn}
The function $f_n\colon \mathcal N_n^\circ\ra \A^1$ induces a relative virtual motive
\[
\L^{-\frac{2n^2+3n}{2}}\left[-\phi_{f_n}\right]_{Q^n_L}\in \mathcal M_{Q^n_L}^{\hat\mu}
\]
that is the pullback under the inclusion $Q_L^n\hookrightarrow Q_{C_0}^n$ of the relative virtual motive 
\[
\L^{-\frac{2n^2+3n}{2}}\left[-\phi_{g_n}\right]_{Q_{C_0}^n}\in \mathcal{M}_{Q^n_{C_0}}^{\hat\mu}.
\]
\end{corollary}

\begin{proof}
The space $\mathcal N_n$ is a smooth scheme of dimension $2n^2+3n$, so the machinery recalled in Section \ref{sec:900} gives the relative motive 
\[
\MF_{\mathcal N_n,g_n} = \L^{-\frac{2n^2+3n}{2}}\left[-\phi_{g_n}\right]_{Q_{C_0}^n}
\]
according to \eqref{def:relvm}. The statement follows since $f_n$ is defined as the restriction of $g_n$ to the smooth open subscheme $\mathcal{N}^\circ_n\subset \mathcal{N}_n$.
\end{proof}

According to Notation \ref{notation:vir}, we can write down the absolute motives
\begin{equation}
\label{vC_def}
\begin{split}
\left[Q_{C_0}^n\right]_{\vir}&=\int_{Q_{C_0}^n} \L^{-\frac{2n^2+3n}{2}}\left[-\phi_{g_n}\right]_{Q_{C_0}^n}\\
\left[Q_{L}^n\right]_{\vir}&=\int_{Q_{L}^n} \L^{-\frac{2n^2+3n}{2}}\left[-\phi_{f_n}\right]_{Q_{L}^n}.
\end{split}
\end{equation}

We end this subsection with two examples of the geometric and motivic behavior of $Q^n_L$ for low $n$.

\begin{example}\label{ex:singptquot}
The scaling action of the torus $\G_m^3$ on $\A^3$ lifts to an action on $Q^n_L$. Let us consider the Quot scheme $Q^2_L$.
We will exhibit a singular point belonging to the torus fixed locus.
First of all, we have $\dim Q^2_L=6$. Consider the point $p\in Q^2_L$ corresponding to 
\[
\mathscr I_Z=(x^2,y^2,xy,xz,yz)\subset \mathscr I_L\subset \mathbb C[x,y,z].
\]
This is depicted in Figure \ref{fig:M1} below.
It is easy to check that $\Hom_{\A^3}(\mathscr I_Z,\mathscr I_L/\mathscr I_Z)$, the tangent space of $Q^2_L$ at $p$, is $10$-dimensional, so that $p$ is a singular point.
\begin{figure}[h]
\captionsetup{width=0.68\textwidth}
\centering
\begin{tikzpicture}[scale=0.26]
\planepartitionn{{9,1},{1}}
\end{tikzpicture}
\caption{\small{A singular point of the Quot scheme $Q^2_L$.}} \label{fig:M1}
\end{figure}
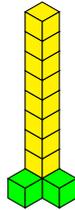
\end{example}

\begin{example}\label{ex:sjfoqh}
The Hilbert scheme $\Hilb^n\A^3$ is nonsingular if $n\leq 3$ and singular otherwise, whereas $Q^n_L$ is already singular if $n = 2$, by Example \ref{ex:singptquot}. Let us fix $n=1$, so that both trace functions (the ones giving rise to $\Hilb^1\A^3$ and to $Q^1_L$ respectively) vanish, and the virtual motives are a shift of the naive motives by $\L^{-3/2}$. On the Hilbert scheme side we have, using Remark \ref{rmk:smvm},
\[
\bigl[\Hilb^1\A^3\bigr]_{\vir}=\L^{-\frac{3}{2}}\cdot\L^3=\L^{\frac{3}{2}},
\]
while on the Quot scheme side we find, using that $Q^1_L = \Bl_L\A^3$,
\begin{align*}
\bigl[Q^1_L\bigr]_{\vir}&=\L^{-\frac{3}{2}}\cdot\bigl[\Bl_L\A^3\bigr]\\
&=\L^{-\frac{3}{2}}\cdot\bigl([\A^3\setminus L]+[L\times \P^1]\bigr)\\
&=\L^{\frac{3}{2}}+\L^{\frac{1}{2}}.
\end{align*}
This is the first instance of formula \eqref{mainformula}.
\end{example}

\subsection{Equivariance of the family}\label{sec:equiv3729}

Let $\mathcal R_n=\Rep_{(n+1,n,1)}(\C \tilde Q_{\con})$ be the affine space parameterising framed $\tilde{Q}_{\con}$-modules of dimension vector $(n+1,n,1)$.
Let us set $G = \GL_{n+1}\times \GL_n$ and let $S \subset \Gamma(\mathscr O_{\mathcal R_n})$ be the subalgebra of functions scaling, under the $G$-action, as a positive power of the given character realising framed stability. Then by general GIT we have $\mathcal N_n=\Proj S$, and the natural inclusion $\Gamma(\mathscr O_{\mathcal R_n})^{G}\subset S$ induces a \emph{projective} morphism 
\be\label{projmor}
\mathsf p_n\colon\mathcal N_n\ra Y_0 = \Spec \Gamma(\mathscr O_{\mathcal R_n})^{G},
\ee
where the affine scheme $Y_0$  can be viewed as the GIT quotient $\mathcal R_n\sslash_{0} G$ at the trivial character.
\begin{prop}\label{thmequiva792}
There is an identity
\begin{align*}
\bigl[\phi_{g_n}\bigr]&=\bigl[g_n^{-1}(1)\bigr]-\bigl[g_n^{-1}(0)\bigr]\in \mathcal M_{\C}\subset \mathcal M_{\C}^{\hat\mu}.
\end{align*}
In particular, $[Q^n_{C_0}]_{\vir}$ lies in $\mathcal M_\C$. 
\end{prop}

\begin{proof}
The four-dimensional torus $T=\mathbb G_m^4$ acts on $\mathcal N_n$ by 
\[
t\cdot (A_1,A_2,B_1,B_2,v)=(t_1A_1,t_2A_2,t_3B_1,t_4B_2,t_1t_2t_3t_4v).
\]
Moreover, the trace function $g_n\colon \mathcal N_n\ra\A^1$ is $T$-equivariant with respect to the primitive character $\chi(t)=t_1t_2t_3t_4$. This means that for all $P\in \mathcal N_n$, we have $g_n(t\cdot P)=\chi(t)g_n(P)$. We claim that the induced action on $\mathcal N_n$ by the diagonal torus $\mathbb G_m\subset T$ is circle compact, that is, it has compact fixed locus and the limits $\lim_{t\ra 0}t\cdot P$ exist in $\mathcal N_n$ for all $P\in\mathcal N_n$.
Following the proof of~\cite[Lemma $3.4$]{BBS}, we notice that \eqref{projmor} is a projective $\mathbb G_m$-equivariant map, where $Y_0$ has a unique $\mathbb G_m$-fixed point, and moreover all orbits have this point in their closure. In other words, limits exist in $Y_0$. Therefore, by properness of $\mathsf p_n$, we conclude that the $\mathbb G_m$-fixed locus in $\mathcal N_n$ is compact and limits exist. This proves the claim.  

Then the equation involving $g_n$ follows by part (i) of Theorem \ref{thm:equivfam738}, proved by Behrend, Bryan and Szendr\H{o}i. In particular, the absolute virtual motive of $Q_{C_0}^n$ carries no monodromy,
\begin{align*}
\bigl[Q^n_{C_0}\bigr]_{\vir}&=\L^{-\frac{2n^2+3n}{2}}\bigl[-\phi_{g_n}\bigr]\in \mathcal M_\C.\qedhere
\end{align*}
\end{proof}

\subsection{A direct critical locus description}
\label{direct_QW}
There is a way of writing down the above critical locus description of $Q_L^n$ that does not involve pulling back from a moduli space of representations for the noncommutative conifold.\footnote{Although, to be precise, for the proof that the description really does recover $Q_L^n$, the only method we offer will rely on the geometry of the noncommutative conifold.}  Consider again the space $\mathcal{N}^{\circ}_n$.  In the definition, we have imposed an open condition on representations in $\mathcal N_n$, namely
\begin{enumerate}
\item
$\rho(b_2)$ is injective.  
\end{enumerate}
We have seen above that the points of $\crit(f_n)$ correspond to $\tilde{Q}_{\con}$-representations $\rho$ satisfying the extra condition that the short exact sequence of vector spaces
\[
0\rightarrow \Image(\rho(b_2))\rightarrow \rho_1\rightarrow\Coker(\rho(b_2))\rightarrow 0
\]
is induced by a short exact sequence of $\mathbb{C}(Q_\con,W_{\con})$-modules, and in particular, $\Image(\rho(b_1))\subset \Image(\rho(b_2))$.  Stability then imposes the extra open condition
\begin{enumerate}
\setcounter{enumi}{1}
\item $\Span(\Image(\rho(\iota)),\Image(\rho(b_2)))=\rho_1$.
\end{enumerate}
Let $\mathcal{N}_n^{\circ\circ}\subset \mathcal{N}_n$ be the open substack defined by the open conditions (1) and (2) above.  For a $\mathbb{C}\tilde{Q}_\con$-module parameterised by a point in $\mathcal{N}_n^{\circ\circ}$, there is a canonical direct sum decomposition $\rho_1\cong\Image(\rho(b_2))\oplus\Image(\rho(\iota))$ and moreover an identification between $\Image(\rho(b_2))$ and $\rho_2$ via the action of $\rho(b_2)$, and an identification of $\Image(\rho(\iota))$ and $\rho_{\infty}$ via $\rho(\iota)$.  It follows that there is an isomorphism 
\be\label{Isom:Qr}
\Gamma\colon \mathcal{N}^{\circ\circ}_n\,\widetilde{\ra}\, \Rep^{\zeta'}_{(n,1)}(Q_r)/\GL_n
\ee
with the fine moduli space of $\zeta'$-semistable $(n,1)$-dimensional $\C Q_r$-representations of the ``reduced'' quiver $Q_r$ depicted in Figure \ref{quiverQr}.

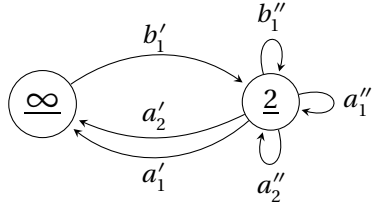
\begin{figure}[ht]
\begin{tikzpicture}[>=stealth,->,shorten >=2pt,looseness=.5,auto]
  \matrix [matrix of math nodes,
           column sep={3cm,between origins},
           row sep={3cm,between origins},
           nodes={circle, draw, minimum size=7.5mm}]
{ 
|(A)| \underline{\infty} & |(B)| \underline{2} \\         
};
\tikzstyle{every node}=[font=\small\itshape]
\path[->] (B) edge [loop above] node {$b_1''$} ()
              edge [loop right] node {$a_1''$} ()
              edge [loop below] node {$a_2''$} ();
\node [anchor=west,right] at (-0.3,0.88) {$b_1'$};              
\node [anchor=west,right] at (-0.3,-0.15) {$a_2'$};              
\node [anchor=west,right] at (-0.3,-0.96) {$a_1'$};              
\draw (B) to [bend left=25,looseness=1] (A) node [midway,above] {};
\draw (B) to [bend left=40,looseness=1] (A) node [midway] {};
\draw (A) to [bend left=35,looseness=1] (B) node [midway,below] {};
\end{tikzpicture}
\caption{The quiver $Q_r$.}\label{quiverQr}
\end{figure}

In \eqref{Isom:Qr} we have put $\zeta'=(-1,n)$.  In a little more detail, given $\rho$ a $\mathbb{C}\tilde{Q}_\con$-module corresponding to a point in $\mathcal{N}^{\circ\circ}_n$,  we set 
\begin{align*}
\rho_{\underline{\infty}}&=\Image(\rho(\iota))\\
\rho_{\underline{2}}&=\rho_2\quad (\cong\Image(\rho(b_2))).
\end{align*}
Then for $s=1,2$ we set
\begin{align*}
\rho(a'_s)=&\rho(a_s)\lvert_{\Image(\rho(\iota))}\\
\rho(b'_1)=&\pi_{\Image(\rho(\iota))}\circ \rho(b_1)\\
\rho(b''_1)=&\pi_{\Image(\rho(b_2))}\circ \rho(b_1)\\
\rho(a''_s)=&\rho(a_s)\lvert_{\Image(\rho(b_2))},
\end{align*}
where $\pi$ denotes the projection maps.
Then 
\be\label{function_hn}
h_n \defeq  f_n\big|_{\mathcal N_n^{\circ\circ}} = \Tr(W_r)\circ \Gamma,
\ee
for $W_r$ the potential
\[
W_r=a''_1b'_1a'_2-a''_2b'_1a'_1+a''_1b''_1a''_2-a''_2b''_1a''_1.
\]

\begin{prop}\label{prop:critQL}
There is an isomorphism of schemes
\[
Q_L^n\cong \crit(h_n).
\]
\end{prop}
The result follows directly from our analysis of the noncommutative conifold.  From that description, we see that in the stable locus, we have $\rho(b_1')=0$.  Then the superpotential relations become
\begin{align*}
a'_2a''_1&=a'_1a''_2\\
a''_1a''_2&=a''_2a''_1\\
a''_1b''_1&=b''_1a''_1\\
a''_2b''_1&=b''_1a''_2
\end{align*}
and the stable locus is identified with the moduli space of triples $(v_x,v_y,N)$ where $N$ is a $\mathbb{C}[x,y,z]$-module, $v_x,v_y\in N$ generate under the action of $\mathbb{C}[x,y,z]$, and $xv_y=yv_x$.  This is the same data as a $\mathbb{C}[x,y,z]$-linear surjection $\theta\colon (x,y)\onto N$, by setting $v_x=\theta(x)$ and $v_y=\theta(y)$.

\smallbreak
The next result contains the statement of Theorem \ref{mainthm1} (B).

\begin{prop}\label{prop:mcv10}
There is an identity
\[
\bigl[\phi_{f_n}\bigr]=\bigl[h_n^{-1}(1)\bigr]-\bigl[h_n^{-1}(0)\bigr]\in \mathcal M_{\C}\subset \mathcal M_{\C}^{\hat\mu}.
\]
In particular, $[Q^n_L]_{\vir}$ lies in $\mathcal M_\C$.  Moreover, the relative motive $[\phi_{f_n}]_{\Sym^n(\mathbb{A}^3)}$ belongs to the subring $\mathcal M_{\Sym^n(\mathbb{A}^3)}\subset \mathcal M_{\Sym^n(\mathbb{A}^3)}^{\hat\mu}$ of monodromy-free motives.
\end{prop}
\begin{proof}
The first statement is proved the same way as Proposition \ref{thmequiva792}, noting that by construction, $[\phi_{f_n}]=[\phi_{h_n}]$, and by (\ref{Isom:Qr}), $h_n$ is a potential on a GIT moduli space of quiver representations admitting a contracting $\mathbb{C}^*$-action for which $\Tr(W_r)$ has weight three.

Next, we claim that the hypersurface 
$h_n^{-1}(0)\subset \mathcal N^{\circ\circ}_n$ is reduced.  
This will follow from the claim that the variety cut out by the function $\Tr(W_\con)$ on $\Rep^\zeta_{(n+1,n,1)}(\tilde{Q}_{\con})$ is reduced, which is in turn weaker than the claim that the variety cut out by $\Tr(W_{\con})$ on $\Rep_{(n+1,n,1)}(\tilde{Q}_{\con})$ is reduced.  This variety is affine, and so it is enough to show that the function $\Tr(W_\con)$ is reduced.  With respect to the $\mathbb{T}$-grading this function has degree $(1,1,1,1)$, and so it cannot be factorised with a repeated factor.

The statement regarding $[\phi_{f_n}]_{\Sym^n(\mathbb{A}^3)}$ then follows from part (ii) of Theorem \ref{thm:equivfam738}.  
\end{proof}

The proof of Theorem \ref{mainthm1} is complete.

\section{Relative DT theory of the conifold}\label{sec:mstr}

The goal of this section is to express the motive $[Q^n_{C_0}]_{\vir}$ defined in \eqref{vC_def} in terms of motivic contributions coming from the ``punctual loci'' inside $Q^n_{C_0}$. For $n\in\mathbb{N}$ we define:
\bitem
\item [(i)] $\mathcal{P}_{\pt}^{n}\subset Q^n_{C_0}$, the subvariety parameterising quotients $\mathscr I_{C_0}\onto \mathscr F$ such that $\mathscr F$ is set-theoretically supported at a \textit{fixed} single point, away from $C_0$, and
\item [(ii)] $\mathcal{P}_{\curv}^n\subset Q^n_{C_0}$, the subvariety parameterising quotients $\mathscr I_{C_0}\onto \mathscr F$ such that $\mathscr F$ is set-theoretically supported at a \textit{fixed} single point on the curve $C_0$.
\eitem
Items (i), (ii) respectively will lead (cf.~Definition \ref{defin:Fully_Punctual}) to the definition of motivic weights
\[
\Omega_{\pt}^n,\,\,\,\Omega_{\curv}^n\,\in \,\mathcal M_{\C},
\]
which will be the basic building blocks for the construction of a virtual motive of the Quot scheme $Q^n_{C}$ for an arbitrary curve $C\subset Y$ in a $3$-fold.

\subsection{Motivic stratifications I: partially punctual strata}\label{section:qnl}

In what follows, we streamline proofs by deducing results for the embedded curve $C_0\subset X$ from the embedded curve $L\subset \mathbb{A}^3$.  In order to achieve this we introduce two automorphisms of $\mathcal{N}_n$ (see \eqref{Ndef}):
\begin{align*}
\overline{\sigma}_\lambda\cdot \left(\rho(\iota),\rho(a_1),\rho(a_2),\rho(b_1),\rho(b_2)\right)&=\left(\rho(\iota),\rho(a_1),\rho(a_2),\rho(b_1)+\lambda\rho(b_2),\rho(b_2)\right),\\
\overline{\xi}_\lambda\cdot \left(\rho(\iota),\rho(a_1),\rho(a_2),\rho(b_1),\rho(b_2)\right)&=\left(\rho(\iota),\rho(a_1),\rho(a_2),\rho(b_1),\rho(b_2)+\lambda\rho(b_1)\right).
\end{align*}
Note that both of these automorphisms preserve $\Tr(W_\con)$ and the $\zeta$-stable locus, and so they induce automorphisms of $Q^n_{C_0}$ preserving $[\phi_{g_n}]_{Q^n_{C_0}}$.  The two automorphisms of $\mathbb{C}^4$ defined by $\sigma_\lambda(x,y,z,w)=(x,y,z+\lambda w,w)$ and $\xi_{\lambda}(x,y,z,w)=(x,y,z,w+\lambda z)$ induce automorphisms of $X$ commuting with $\HC_X$ in the sense that $\Sym^n(\sigma_{\lambda})\circ\HC_X=\HC_X\circ \overline{\sigma}_\lambda$ and $\Sym^n(\xi_{\lambda})\circ\HC_X=\HC_X\circ\overline{\xi}_\lambda$ (see Remark \ref{hcRem}).
\begin{lemma}
\label{res_lemma}
Let $\alpha\in \mathcal{M}_{\Sym^n(X)}^{\hat\mu}$ satisfy 
\begin{itemize}
\item [(1)]
$(\Sym^n(\mathbb{A}^3)\hookrightarrow \Sym^n(X))^*\alpha=0$,
\item [(2)]
$\sigma_{\lambda}^*\alpha=\alpha$ for all $\lambda\in\mathbb{C}$,
\item [(3)]
$\xi_{\lambda}^*\alpha=\alpha$ for all $\lambda\in\mathbb{C}$.
\end{itemize}
Then $\alpha=0$.
\end{lemma}
\begin{proof}
The space $\Sym^n(X)$ has an open cover by open subsets of the form $U_t = \Sym^n(X\setminus \pi^{-1}(t))$ for $t\in\mathbb{P}^1$: if $\gamma\in\Sym^n(X)$ is in the complement of $\cup_{t\in T}U_t$ for $|T| > n$ then the support of $\gamma$ is spread across more than $n$ fibres of the projection $\pi\colon X\rightarrow \mathbb{P}^1$.  By our assumptions, the restriction of $\alpha$ to $U_t$ is zero.  By the scissor relations, it follows that $\alpha=0$.
\end{proof}
\begin{corollary}
\label{res_cor}
Let $\alpha,\beta\in\mathcal{M}^{\hat{\mu}}_{\Sym^n(X)}$ satisfy conditions (2) and (3) from Lemma \ref{res_lemma}, and $(\Sym^n(\mathbb{A}^3)\hookrightarrow \Sym^n(X))^*\alpha=(\Sym^n(\mathbb{A}^3)\hookrightarrow \Sym^n(X))^*\beta$.  Then $\alpha=\beta$.
\end{corollary}
\smallbreak
The next lemma is an incarnation of the fact that taking box sum with a quadratic function (locally) does not change the vanishing cycle complex, while for global triviality one has to be mindful of monodromy.  The implication is that we can pass to a ``minimal'' potential at the expense of keeping track of some extra monodromy data, which in the Kontsevich--Soibelman framework, and then elsewhere, is called orientation data.  In the language of potentials on 3-Calabi--Yau categories, one can think of the proof of part (2) of Lemma \ref{fixed_BBS} as working by proceeding to a ``partially minimized'' potential.
\begin{lemma}
\label{quad_lemma}
Let $\pi\colon \Tot(V)\ra X$ be the projection from the total space of a vector bundle on a smooth connected variety $X$, and let $f\colon \Tot(V)\rightarrow \mathbb{C}$ be a function that is quadratic in the fibres, i.e.~$f(z\cdot v)=z^2f(v)$ for $z\in\mathbb{C}$, where we have given $\Tot(V)$ the scaling action of  $\mathbb{C}^*$.  Assume $X=\crit(f)$, where we have identified $X$ with the zero section of $\Tot(V)$.
\begin{enumerate}
\item
For $x\in X$, there is an equality 
\begin{equation}
[x]_{\vir}=\L^{-\frac{\dim(X)}{2}}\in \mathcal{M}_{\C}\subset\mathcal{M}_{\C}^{\muhat}.
\end{equation}
\item
Assume that $V\cong V_-\oplus V_+$ where $f\lvert_{\Tot(V_-)}=f\lvert_{\Tot(V_+)}=0$.  Then there is an equality
\begin{equation}
\label{glob_no_mon}
\bigl[-\phi_f\bigr]_X=\bigl[\id\colon X\rightarrow X\bigr]\in\mathcal{M}_X\subset \mathcal{M}_X^{\muhat}.
\end{equation}
\item
Under the conditions of (2), let $G$ be a finite group acting freely on $X$, let $V\cong V_-\oplus V_+$ be a direct sum decomposition of $G$-equivariant vector bundles, and assume that $f$ is $G$-invariant.  Then 
\[
\bigl[-\phi_f\bigr]_X^G=\bigl[\id\colon X\rightarrow X\bigr]\in\mathcal{M}^G_X\subset \mathcal{M}^{G\times\muhat}_X.
\]
\end{enumerate}
\end{lemma}

In part (3) we include the assumption that $G$ acts freely so that we may apply Proposition \ref{BittnerProp}.

\begin{proof}
Part (1): Zariski locally, we can write $\Tot(V)=X\times \mathbb{A}^m$, and $f=\sum_{1\leq i,j\leq m}f_{ij}t_it_j$ for $F=(f_{ij})_{1\leq i,j\leq m}$ a matrix of functions on $X$, and $t_1,\ldots,t_m$ coordinates on $\mathbb{A}^m$.  We can, and will, assume that $f_{ij}=f_{ji}$ for all $i$ and $j$.  For every closed point $x\in X$, the functions $(\partial/\partial t_i) f$ and $(\partial/\partial s) f$, for $s$ local coordinates at $x$, generate 
\[(t_1,\ldots,t_m)\mathscr{O}_{X,x}[t_1,\ldots,t_m]/(t_1,\ldots,t_m)^2\mathscr{O}_{X,x}[t_1,\ldots,t_m].
\]
On the other hand, $(\partial/\partial s)f=0$ in this quotient.  It follows that $\det(f_{ij})_{1\leq i,j\leq m}$ is an invertible function on $X$.  In an analytic open neighborhood of $x$ we may thus apply a change of coordinates and obtain $f_{ij}=\delta_{i-j}$, and then the first result follows from direct calculation, or the explicit formula of Denef and Loeser \cite[Theorem 4.3.1]{DenefLoeser1}.
\smallbreak
Part (2+3): By nondegeneracy of $F$, if $V\cong V_-\oplus V_+$ is a decomposition as in the statement of the lemma, then $\dim(V_-)=m/2$, and the symmetric bilinear form $F$ establishes an isomorphism of vector bundles $V_+\cong (V_-)^*$.  Replacing $V_+$ by $(V_-)^*$ we obtain
\[
F=\left(\begin{array}{cc}0& \mathrm{Id}_{m/2\times m/2}\\ \mathrm{Id}_{m/2\times m/2} & 0\end{array}\right).
\]
Letting $\mathbb{C}^*$ act by scaling $V_+$, the function $f$ is equivariant of weight one.  The proof of \cite[Theorem 5.9]{BenSven3} shows that if we have a smooth variety $Y$, an integer $r$ and a function $g=\sum g_i t_i$ on $Y\times \mathbb{A}^r$ then $\pi_{Y,!}([\phi_g]_{g^{-1}(0)})=[V(g_1,\ldots,g_r)]$.  Moreover the proof generalises without modification to the case in which $\pi_Y\colon Y\times \mathbb{A}^r\rightarrow Y$ is a projection from the total space of a $G$-equivariant vector bundle.  The second and the third parts follow, putting $Y=X\times \Tot(V_-)$.
\end{proof}
\begin{remark}
If we relax the condition on $f$, part (2) may fail.  For instance consider the function $f=zx^2$ on $\Spec \mathbb{C}[z^{\pm 1},x]\cong \mathbb{C}^*\times\mathbb{A}^1_{\mathbb{C}}$.  It is easy to check that the associated virtual motive satisfies $[\mathbb{C}^*]_{\vir}=0$.
\end{remark}
\begin{remark}
Likewise, if we relax the condition on the $G$-action on $V$, part (3) may fail.  For instance let $G$ be the cyclic group of order 2, let $\Tot(V)=\Spec \mathbb{C}[x,y]\cong \mathbb{A}^2$, with $f=xy$ and $G$ swapping $x$ and $y$.  Then it is easy to check that for the associated equivariant critical structure we have $[\pt]_{\vir}=[G]-[\pt]$, where on the right hand side $[G]$ is a pair of points, permuted by the $G$-action, and $\pt$ carries the trivial $G$-action.  This is a consequence of the fact that the $G$-action on $\mathrm{H}(\mathbb{A}^2,\Phi_{xy})$ is the sign representation, see for instance \cite[Lemma 4.1]{RS17}.
\end{remark}
The ultimate goal of this section is to show that as a relative motive over $\Sym(X)$, the virtual motive 
\be\label{relative_virtual_conifold}
\sum_{n\geq 0} \,(-1)^n\bigl[Q^n_{C_0}\xrightarrow{\HC_X} \Sym^n(X)\bigr]_{\vir}
\ee
is generated under $\Exp_{\cup}$ by motives that are supported on the punctual locus, and constant away from $C_0$, as well as constant on $C_0$.  To get to this point will require some work, and we break the proof up by showing first that \eqref{relative_virtual_conifold} is at least generated by motives supported on a ``partially punctual'' locus.  By Lemma \ref{res_lemma} it is enough to prove the analogous result for $\sum_n (-1)^n [Q^n_{L}\xrightarrow{\HC} \Sym^n(\mathbb{A}^3)]_{\vir}$.

We explain here what we mean by the partially punctual locus.  Consider again the map 
\[
\HC\colon Q_L^n\rightarrow \Sym^n(\mathbb{A}^3).
\]
The embedding of the $xy$-plane in $\A^3$ induces an embedding of varieties
\be
\label{iota_def}
\iota_n\colon \Sym^n(\mathbb{A}^2)\hookrightarrow \Sym^n(\mathbb{A}^3)
\ee
and we denote by $Q^{\bullet,\bullet,\nilp,n}_L$ the preimage of $\Sym^n(\mathbb{A}^2)$ under $\HC$, i.e.
\be\label{Def:Nilp}
Q^{\bullet,\bullet,\nilp,n}_L = \HC^{-1}(\Sym^n(\mathbb{A}^2)) \subset Q^n_L.
\ee
The map $\iota_n$ is the inclusion of the subspace of configurations of points which all have $z$-coordinate $0$, which explains the notation $(\bullet,\bullet,\nilp)$ --- the scheme-theoretic version of this condition is that the operator corresponding to the action of $z$ is nilpotent.  So, ordering the operators corresponding to $x,y,z$ alphabetically, the first two are unconstrained, and the third is nilpotent.  

\begin{notation}
More generally, for $\#_x,\#_y,\#_z \in\{\unip,\nilp,\bullet\}$ we define $Q_L^{\#_x,\#_y,\#_z,n}\subset Q^n_L$ by imposing the closed conditions that for $w\in \set{x,y,z}$ the operator $\cdot w$ is nilpotent if $\#_w=\nilp$, or unipotent if $\#_w=\unip$ --- so for instance $Q^n_L=Q_L^{\bullet,\bullet,\bullet,n}$.
\end{notation}

There is an action of $\mathbb{A}^1$ on $\Sym^n(\mathbb{A}^3)$ via simultaneous addition on the $z$ coordinate of all points in a configuration, and we let
\be
\label{tiota_def}
\tilde{\iota}_n\colon \Sym^n(\mathbb{A}^2)\times\mathbb{A}^1\rightarrow \Sym^n(\mathbb{A}^3)
\ee
be the restriction of this action.  It is again an embedding, this time of the subvariety of $n$-tuples of points which all have the same $z$-coordinate (not necessarily zero).  It is this locus that we call ``partially punctual''.  Finally, consider the morphism
\[
q_z\colon \Sym^n(\mathbb{A}^3)\rightarrow \Sym^n(\mathbb{A}^1)
\]
obtained by projecting onto the $z$ coordinate.  We define $\HC^z=q_z\circ \HC$.  More generally, for $a_1,\ldots,a_r$ distinct elements of $\{x,y,z\}$ and $T\subset Q_L^n$ we denote by 
\[
\HC^{a_1\cdots a_r}:T\rightarrow \Sym^n(\mathbb{A}^r)
\]
the map given by composing the restriction of $\HC$ to $T$ with the projection $\Sym(\mathbb{A}^3)\rightarrow \Sym(\mathbb{A}^r)$ induced by the projection $\mathbb{A}^3\rightarrow \mathbb{A}^r$ defined by forgetting the coordinates not contained in $a_1,\ldots,a_r$.  

The space $\Sym^n(\mathbb{A}^1)$ is stratified according to partitions of $n$, and for $\alpha\vdash n$ we denote by $Q^{\bullet,\bullet,\alpha}_L\subset Q_L^n$ the corresponding stratum of the stratification of $Q_L^n$ given by pulling back along $\HC^z \colon Q^n_L\ra\Sym^n(\A^1)$.  So for instance
\[
Q^{\bullet,\bullet,(n)}_L=(\Sym^n(\mathbb{A}^2)\times\mathbb{A}^1)\times_{\Sym^n(\mathbb{A}^3)} Q^n_L\subset Q^n_L
\]
is the fibre product of $\tilde{\iota}_n$ and $\HC$.
\begin{lemma}
\label{fixed_BBS}
\begin{enumerate}
\item
There is an equality
\[
\tilde\iota_n^\ast\bigl[Q_L^{\bullet,\bullet,(n)}\xrightarrow{\HC} \Sym^n(\mathbb{A}^3)\bigr]_{\vir}=\iota_n^*\bigl[Q_L^n\xrightarrow{\HC } \Sym^n(\mathbb{A}^3)\bigr]_{\vir}\boxtimes \bigl[\mathbb{A}^1\xrightarrow{\id_{\mathbb{A}^1}}\mathbb{A}^1\bigr]
\]
in $\mathcal{M}_{\Sym^n(\mathbb{A}^2)\times\mathbb{A}^1}$.  In other words, the motive on the left hand side is constant in the $\mathbb{A}^1$-factor.
\item
Moreover, for $\alpha\vdash n$ there is an equality
\begin{align}
\label{second_eq}
\left[Q^{\bullet,\bullet,\alpha}_L\rightarrow \Sym^n(\mathbb{A}^3)\right]_{\vir}=\cup_!\pi_{G_{\alpha}}j^*_{\alpha}\left(\underset{i\lvert\alpha_i\neq 0}{\Boxtimes}\left[Q^{\bullet,\bullet,(i)}_L \rightarrow \Sym^i(\mathbb{A}^2)\times\mathbb{A}^1\right]^{\otimes \alpha_i}_{\vir}\right)
\end{align}
where $j_{\alpha}$ is the $G_{\alpha}$-equivariant inclusion of the complement of the pullback of the big diagonal under the $G_{\alpha}$-equivariant projection onto $z$-coordinates: 
\[
\prod_{i\lvert \alpha_i\neq 0}(\Sym^i(\mathbb{A}^2)\times\mathbb{A}^1)^{\alpha_i}\rightarrow \mathbb{A}^{\sum \alpha_i}.
\]
Here $G_{\alpha}$ is the automorphism group for the partition type $\alpha$, and $\cup$ is the union of points map on $\Sym(\mathbb{A}^3)$.
\end{enumerate}
\end{lemma}

Before we begin the proof, we give some guidance for how to read the right hand side of \eqref{second_eq}.  Firstly, recall from Section \ref{Sec:Power_Structures} that the infinite union of algebraic varieties $\Sym(\mathbb{A}^3)$ has a symmetric monoidal structure $\cup$, given by taking unions of unordered points with multiplicity.  We consider $\Sym^i(\mathbb{A}^2)\times\mathbb{A}^1$ as a subvariety of $\Sym^i(\mathbb{A}^3)$ via $\tilde{\iota}_i$.  We abuse notation by writing $\cup$ again for the map $\Sym(\mathbb{A}^3)^m\rightarrow \Sym(\mathbb{A}^3)$  taking an $m$-tuple of sets of unordered points with multiplicity to their union.  The term in big round brackets on the right is a $G_{\alpha}$-equivariant motive via Lemma \ref{lemma:majdgap}.

\begin{proof}
Consider again the space $\mathcal{N}^{\circ\circ}_n\subset \mathcal N_n$ from Section \ref{direct_QW}.  We define the subspace $\mathcal{T}_n\subset \mathcal{N}^{\circ\circ}_n$ by the condition that $\Tr(\rho(b''_1))=0$.  Then there is an isomorphism
\[
\mathcal{T}_n\times\mathbb{A}^1\,\widetilde{\ra}\,\mathcal{N}^{\circ\circ}_n
\]
given by
\begin{multline*}
    (\rho(a''_1),\rho(a''_2),\rho(b''_1),\rho(a'_1),\rho(a'_2),\rho(b'_1),t)\\
    \mapsto 
    (\rho(a''_1),\rho(a''_2),\rho(b''_1)+t\cdot \Id_{n\times n},\rho(a'_1),\rho(a'_2),\rho(b'_1)),
\end{multline*}
and the function $h_n = f_n|_{\mathcal N_n^{\circ\circ}}$ (cf.~\eqref{function_hn} and Proposition \ref{prop:critQL}) is pulled back from a function $\overline{h}_n$ on $\mathcal{T}_n$.  The stratification by partition type of $\rho(b_1'')$ for $Q_L^n\subset \mathcal N_n^{\circ\circ}$ is induced by a stratification of $\mathcal{T}_n$: for $\alpha$ a partition of $n$, define $\mathcal{T}_{\alpha}\subset \mathcal{T}_n$ to be the locally closed subvariety whose $\mathbb{C}$-points correspond to $Q_r$-representations for which the partition type of the generalised eigenvalues of $\rho(b_1'')$ are given by $\alpha$.   Then $\crit(\overline{h}_n)\cap \mathcal{T}_{(n)}=Q_L^{\bullet,\bullet,\nilp,n}$, i.e. the isomorphism $\mathcal T_n \times \A^1\widetilde{\ra}\mathcal N_n^{\circ\circ}$ sends $(\crit(\overline{h}_n)\cap \mathcal{T}_{(n)})\times \{0\}$ onto $Q_L^{\bullet,\bullet,\nilp,n}$.  

Let $\Sym^n_0(\mathbb{A}^1)\subset \Sym^n(\mathbb{A}^1)$ be the closed subvariety of $n$-tuples summing to zero.  Let 
\[
\lambda:\mathcal N_n^{\circ\circ}\rightarrow \Sym^n(\mathbb{A}^1)
\]
be the morphism taking a module $\rho$ to the eigenvalues (with multiplicity) of $\rho(b''_1)$.  Note that $\lambda\lvert_{Q^n_L}=\HC^z$.
Then the first equality follows from the commutativity of 
\[
\begin{tikzcd}[row sep=large]
\mathcal T_n \times \A^1 \arrow{rrrr}{\cong} \arrow{d}{\lambda\lvert_{\mathcal T_n}\times \id_{\mathbb{A}^1}} & & & &
\mathcal N_n^{\circ\circ} \arrow{d}{\lambda} \\
\Sym_0^n(\mathbb{A}^1)\times \mathbb{A}^1 \arrow{rrrr}{((t_1,\ldots,t_n),t)\mapsto (t_1+t,\ldots,t_n+t)} & & & &
\Sym^n(\mathbb{A}^1)
\end{tikzcd}
\]
and the fact that, pulling back along the top isomorphism, the function $h_n$ becomes $\overline{h}_n\oplus 0$.

Let $\alpha=(1^{\alpha_1}\cdots r^{\alpha_r})\vdash n$ be a partition.  We write $l(\alpha)=\sum_{i\leq r} \alpha_i$ for the total number of parts of $\alpha$.  For the proof of the second part of the lemma we adapt the proof of \cite[Prop 2.6]{BBS}, via the following commutative diagram, which will take some time to define and describe.  
\[
\begin{tikzcd}
\prod_i(\mathcal N_i^{\circ\circ})^{\alpha_i} &
\tilde{P}\arrow{r}\arrow[hook']{l} &
\tilde{S}^\circ\arrow[hook]{r} &
\tilde{S}\arrow[hook]{rrr}\arrow{dd} &
&
&
\mathcal N_\alpha\arrow{dd}{s} 
\\
V\arrow{rr}{\tiny{\textrm{closed}}}\arrow{dd}{\pi} & & %
\tilde{U}\arrow[hook,crossing over]{rr}\arrow{dd}\arrow[dotted]{ur}\arrow[dotted]{ul}\arrow[dotted]{u} & &
\prod_{i}(Q_L^i)^{\alpha_i} & \\
& & & S\arrow[hook]{rrr} & & & \mathcal N_n^{\circ\circ} \\
Q^\alpha_L\arrow{rr}{\tiny{\textrm{closed}}} & & U\arrow[hook]{rr}\arrow[dotted]{ur} & & Q^n_L\arrow[dotted]{urr} & 
\end{tikzcd}
\]

All of the dotted arrows correspond to the inclusion of the critical locus of a function.  All hooked arrows denote \textit{open} inclusions.  

The open subspace $\tilde{U}\subset \prod_i (Q^i_L)^{\alpha_i}$ is defined by the condition that the set of eigenvalues of $\rho(b''_1)$ in each of the $l(\alpha)$ factors of the product are distinct from the set of eigenvalues for $\rho(b''_1)$ in any other factor in the product.  The open subset $U\subset Q^n_L$ is the image of the \'etale map $\tilde U\ra Q^n_L$ given by sending the $l(\alpha)$-tuple $\{\mathscr{I}_L\twoheadrightarrow \mathscr{F}_j\}_{j\leq l(\alpha)}$ to 
\[
\mathscr{I}_L\twoheadrightarrow \bigoplus_{j\leq l(\alpha)}\mathscr{F}_j. 
\]
The closed subset $V\subset \tilde{U}$ is defined by the condition that in each of the $l(\alpha)$ factors in the product decomposition there is only one eigenvalue. It is in fact the preimage of $Q^\alpha_L=\HC^{-1}(\Sym^\alpha(\A^3))\subset U$ under the map $\tilde U\ra U$. 
\smallbreak
We form the quiver $Q_{\alpha}$ (Figure \ref{quiver_Qalpha}) and dimension vector $\dd_{\alpha}$ as follows.  Set 
\begin{align*}
(Q_{\alpha})_0&=\Set{\underline{2}_1,\ldots,\underline{2}_{l(\alpha)},\underline{\infty}}\\
(Q_{\alpha})_1&=\Set{a'_{1,i},a'_{2,i},b'_{1,i},b''_{1,i}}_{1\leq i\leq l(\alpha)}\cup \Set{a''_{1,i,j},a''_{2,i,j}}_{1\leq i,j\leq l(\alpha)}.
\end{align*}
We set 
\begin{align*}
t(a'_{1,i})=t(a'_{2,i})=s(b'_{1,i})&=\underline{\infty}\\
s(a'_{1,i})=s(a'_{2,i})=t(b'_{1,i})=t(b''_{1,i})=s(b''_{1,i})&=\underline{2}_i\\
t(a''_{1,i,j})=t(a''_{2,i,j})&=\underline{2}_j\\
s(a''_{1,i,j})=s(a''_{2,i,j})&=\underline{2}_i.
\end{align*}
\begin{figure}[ht]
\begin{tikzpicture}[>=stealth,->,shorten >=2pt,looseness=.5,auto]
  \matrix [matrix of math nodes,
           column sep={3cm,between origins},
           row sep={3cm,between origins},
           nodes={circle, draw, minimum size=7.5mm}]
{ 
|(A)| \underline{2}_1 & |(B)| \underline{2}_2 \\      
|(I)| \underline{\infty} &  \\
};
\tikzstyle{every node}=[font=\small\itshape]

\path[->] (B) edge [loop right] node {$b_{1,2}''$} ();
\path[->] (A) edge [loop left] node {$b_{1,1}''$} ();
\path[->>] (A) edge [loop above] node {$a_{k,1,1}''$} ();
\path[->>]  (B) edge [loop above] node {$a_{k,2,2}''$} ();

\node [anchor=west,right] at (1,-0.5) {$b_{1,2}'$};              
\node [anchor=west,right] at (-0.3,-0.1) {$a_{k,2}'$};
\node [anchor=west,right] at (-0.4,2.25) {$a_{k,1,2}''$};              
\node [anchor=west,right] at (-0.4,0.75) {$a_{k,2,1}''$};              
\node [anchor=west,right] at (-2.6,0.1) {$a_{k,1}'$};              
\node [anchor=west,right] at (-1.2,0.1) {$b_{1,1}'$};        

\draw[->>] (A) to [bend left=25,looseness=1] (B) node [midway,below] {};
\draw[->>] (B) to [bend left=25,looseness=1] (A) node [midway,below] {};
\draw (I) to [bend right=65] (B) node {};
\draw[->>] (B) to [bend left=35] (I) node {};
\draw[->>] (A) to [bend right=25] (I) node {};
\draw (I) to [bend right=25] (A) node {};
\end{tikzpicture}
\caption{The quiver $Q_\alpha$, for the case in which $l(\alpha)=2$. The index $k$ varies in $\set{1,2}$.}\label{quiver_Qalpha}
\end{figure}
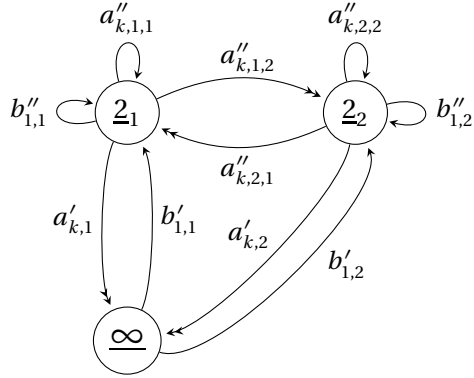
We give $Q_{\alpha}$ the dimension vector $\dd_{\alpha}=(1^{\alpha_1},\ldots,r^{\alpha_r},1)$, i.e.~we always assign the vertex $\underline{\infty}$ dimension 1, and then the rest of the dimension vector is given by the ordered set of numbers $\alpha$.  We give this quiver the stability condition
\[
\zeta_{\alpha}=(-1,\ldots,-1,n).
\]
A ($\zeta_{\alpha}$-stable) $\C Q_{\alpha}$-module is the same data as a ($\zeta'$-stable) $\C Q_r$-module $\rho$, along with an ordered vector space decomposition of the vector space $\rho_{\underline{2}}$ that is preserved by the operation $\rho(b''_1)$.  As such, there is a forgetful functor from $\C Q_{\alpha}$-modules to $\C Q_r$-modules, inducing a generically finite morphism
\[
s\colon \mathcal{N}_{\alpha} = \Rep^{\zeta_{\alpha}}_{\dd_\alpha}(Q_{\alpha})/\GL_{\alpha}\rightarrow \mathcal{N}_n^{\circ\circ}
\]
which at the level of points amounts to setting
\[
s(\rho_{\underline{2}})=\bigoplus_{i\leq l(\alpha)} \rho_{\underline{2}_i}.
\]
The scheme is $G_{\alpha}$-equivariant, via the obvious $G_{\alpha}$-action on the quiver $Q_{\alpha}$.

There is a unique potential $W_{\alpha}$ on $Q_{\alpha}$ such that $\Tr(W_{\alpha})$ is the pullback of $\Tr(W_r)$ under this forgetful map.  Precisely, we define
\[
W_{\alpha}=
\sum_{1\leq i,j\leq l(\alpha)}\left(a''_{1,j,i}b'_{1,j}a'_{2,i}-a''_{2,j,i}b'_{1,j}a'_{1,i}\right)+\sum_{1\leq i,j\leq l(\alpha)}\left(a''_{1,j,i}b''_{1,j}a''_{2,i,j}-a''_{2,j,i}b''_{1,j}a''_{1,i,j}\right).
\]
We define the open subscheme $\tilde{S}\subset \mathcal{N}_{\alpha}$ by the condition that for every $i,j\leq l(\alpha)$ the endomorphisms $\rho(b''_{1,i})$ and $\rho(b''_{1,j})$ share no eigenvalues.  The map
\[
s\colon \tilde{S}\rightarrow \mathcal{N}^{\circ\circ}_n
\]
is quasi-finite, and factors through a finite morphism $\tilde{S}\rightarrow S$ where $S\subset \mathcal{N}_n^{\circ\circ}$ is an open subscheme.

Let $\rho\in \tilde{S}\cap \crit(\Tr(W_{\alpha}))$.  Then $s(\rho)(b'_1)=0$ by stability, and so from the superpotential relations we deduce that $a''_{1,i,j}b''_{1,i}=b''_{1,j}a''_{1,i,j}$ and $a''_{2,i,j}b''_{1,i}=b''_{1,j}a''_{2,i,j}$ for all $i, j$ and so from our condition on the eigenvalues of the $b''_{1,i}$ we deduce that for all $1\leq i\neq j\leq l(\alpha)$ we have $\rho(a''_{1,i,j})=\rho(a''_{2,i,j})=0$.  As such, in calculating the relative vanishing cycle $[\phi_{\Tr(W_{\alpha})}]_{\tilde{S}}$ we can restrict to the set $\tilde{S}^{\circ}\subset \tilde{S}$ defined by the condition that $\rho$ remains stable after setting all $\rho(a''_{1,i,j})=\rho(a''_{2,i,j})=0$ for $i\neq j$.  So both the inclusions $\tilde{U}\rightarrow \tilde{S}$ and $\tilde{U}\rightarrow \tilde{S}^{\circ}$ are the inclusions of the critical locus of the function $\Tr(W_{\alpha})$ on the respective targets.

The space $\tilde{S}^{\circ}$ is a vector bundle over
\begin{equation}
\label{PZS}
\tilde{P}\subset \prod_{i}(\mathcal{N}^{\circ\circ}_{i})^{\alpha_i},
\end{equation}
the open subset defined by the condition that the generalised eigenvalues of the $\rho(b''_1)$-operators from different factors are distinct.  The projection from $\tilde{S}^{\circ}$ to $\tilde{P}$ is given by forgetting the values of $\rho(a''_{k,i,j})$ for $k=1,2$ and $i\neq j$.

The map $\pi$ is a Galois cover with Galois group $G_{\alpha}$, and the map $\tilde{S}^{\circ}\rightarrow S$ is a Galois cover in a formal neighbourhood of the morphism $\pi$.  

The space $\tilde{P}$ carries the \textit{free} $G_{\alpha}$-action inherited from $\mathcal{N}_{\alpha}$, and furthermore the vector bundle $\tilde S^{\circ}$ has a direct sum decomposition $V_{-}\oplus V_+$, where $V_-$ keeps track of the entries of $\rho(a''_{1,i,j})$ for $i\neq j$ and $V_+$ keeps track of the entries of $\rho(a''_{2,i,j})$ for $i\neq j$, and so the decomposition is preserved by the $G_{\alpha}$-action.

Let $m=\rk(V_+)=\rk(V_-)$.  If we let $G_1=\mathbb{C}^*$ act on $\tilde{S}^{\circ}$ by scaling $V_+$ with weight $1$ and $V_-$ with weight $-1$, then $\Tr(W_{\alpha})$ is $\mathbb{C}^*$-invariant.  Let $G_2=\C^*$ act by scaling both $V_+$ and $V_-$ with weight one.  Then since each term in the potential contains at most two instances of $a''_{k,i,j}$, for $k=1,2$, for every $\rho\in \tilde{S}^{\circ}$, there is a fixed constant $C$ for which $|\Tr(W_{\alpha})(z\cdot_2 \rho)| \leq C \lvert z\lvert^2$ where the action is via the $G_2$-action.  We deduce that on $\tilde{S}^{\circ}$ we can write
\[
\Tr(W_{\alpha})=g_0+g_1
\]
where $g_0$ is a function pulled back from the projection to $\tilde{P}$ and $g_1$ is a $G_1$-invariant function on $\tilde{S}^{\circ}$, quadratic in the fibres.  After passing to a $G_{\alpha}$-invariant Zariski open subset we can trivialize the vector bundle $\tilde{S}^{\circ}$ and write
\[
g_1=\sum_{1\leq i,j\leq m}g_{i,j}t_is_j
\]
where $t_i$ and $s_j$ are coordinates on the fibre of $\Tot(V_+)$ and $\Tot(V_-)$ respectively.  From the equality $\crit(\Tr(W_{\alpha}))=\crit(g_0)$, arguing as in the proof of the second part of Lemma \ref{quad_lemma}, the matrix $\{g_{ij}\}_{1\leq i,j\leq m}$ is invertible, and after a change of coordinates on $\Tot(V_+)$ we may assume that $g_{ij}=\delta_{i-j}$.  By the Thom--Sebastiani isomorphism, and the third part of Lemma \ref{quad_lemma} we deduce that $[\phi_{\Tr(W_{\alpha})}]^{G_{\alpha}}_{\tilde{S}^{\circ}}=[\phi_{g_0}]^{G_{\alpha}}_{\tilde{P}}\in\mathcal{M}^{G_{\alpha}\times\muhat}_{\tilde{P}}$. Finally, we note that $g_0$ is the sum of the potentials on the factors of $\prod_i(\mathcal{N}_i^{\circ\circ})^{\alpha_i}$, and the result follows from Proposition \ref{BittnerProp} and the Thom--Sebastiani theorem.
\end{proof}
Recall from \eqref{Def:Nilp} the subvarieties $Q_{L}^{\bullet,\bullet,\nilp,n} \subset Q^n_L$, relative over $\Sym^n(\A^2)$ via the map $\HC^{xy}$.
The next corollary follows from Proposition \ref{exp_prop}. 
\begin{corollary}
\label{firstc}
Define classes $\Phi_n\in \mathcal{M}_{\Sym^n(\mathbb{A}^2)}$ via
\[
\sum_{n\geq 0}\, (-1)^n\bigl[Q_{L}^{\bullet,\bullet,\nilp,n}\xrightarrow{\HC^{xy}}\Sym^n(\mathbb{A}^2)\bigr]_{\vir}=\prExp_{\cup}\left(\sum_{n\geq 1}\Phi_n \right).
\]
Then
\begin{multline*}
\sum_{n\geq 0}\,(-1)^n\bigl[Q^n_L\xrightarrow{\HC}\Sym^n(\mathbb{A}^3)\bigr]_{\vir} \\
=\prExp_{\cup}\left(\sum_{n\geq 1}\left(\Sym^n(\mathbb{A}^2)\times\mathbb{A}^1\xrightarrow{\tilde{\iota}_n} \Sym^n(\mathbb{A}^3)\right)_!\bigl(\Phi_n\boxtimes \bigl[\mathbb{A}^1\xrightarrow{\id}\mathbb{A}^1\bigr]\bigr)\right).
\end{multline*}
\end{corollary}
\begin{proof}
By Lemma \ref{fixed_BBS}\,(2) and (\ref{SymSign}) we deduce that for $\alpha\vdash n$
\[
(-1)^n\left[Q^{\bullet,\bullet,\alpha}_L\rightarrow \Sym^n(\mathbb{A}^3)\right]_{\vir}=\cup_!\pi_{G_{\alpha}}j^*_{\alpha}\left(\underset{i\lvert\alpha_i\neq 0}{\Boxtimes}(-1)^i\left[Q^{\bullet,\bullet,(i)}_L \rightarrow \Sym^i(\mathbb{A}^2)\times\mathbb{A}^1\right]^{\otimes \alpha_i}_{\vir}\right).
\]
The result then follows from Proposition \ref{exp_prop} and
\[
\left[Q^n_L\rightarrow \Sym^n(\mathbb{A}^3)\right]_{\vir}=\sum_{\alpha\vdash n}\left[Q^{\bullet,\bullet,\alpha}_L\rightarrow \Sym^n(\mathbb{A}^3)\right]_{\vir}.\qedhere
\]
\end{proof}

\subsection{Motivic stratifications II: fully punctual strata}
Corollary \ref{firstc} says that, considered as relative motives over $\Sym(\mathbb{A}^1)$, via projection to the $z$ coordinate, the DT invariants are generated by classes on the small diagonals, which are moreover \textit{constant} as relative motives over $\mathbb{A}^1$. 
In other words, there are relative motives over $\Sym^n(\mathbb{A}^2)$ that generate the motivic DT partition function under taking the external product with the constant motive on $\mathbb{A}^1$ and taking exponentials. 

We next show that these relative motives are themselves supported on the small diagonal (although they are non-constant, with a ``jump'' at the intersection $\mathbb{A}^2\cap L$) so that considered as a relative motive over $\Sym(\mathbb{A}^3)$, the virtual motive of $Q^n_{L}$ is generated on the small diagonal.

For $n'\leq n$ and $\alpha$ a partition of $n-n'$ we define 
\[
\mathcal{Q}^{n',\alpha}=Q_L^{\bullet,\bullet,\nilp,n}\cap (\HC^y)^{-1}(A_{n',\alpha})
\]
where $A_{n',\alpha}$ is the subvariety of $\Sym^n(\mathbb{A}^1)$ defined by the condition that $0$ occurs $n'$ times, and the partition defined by the $(n-n')$-tuple of points away from $0$ is $\alpha$.  We denote by 
\begin{align*}
\kappa^1_n\colon \Sym^n(\mathbb{A}^1)&\hookrightarrow \Sym^n(\mathbb{A}^2) & \kappa^0_n\colon \Sym^n(\mathbb{A}^1)&\hookrightarrow \Sym^n(\mathbb{A}^2)\\
(z_1,\ldots,z_n)&\mapsto ((z_1,1),\ldots,(z_n,1)) & (z_1,\ldots,z_n)&\mapsto ((z_1,0),\ldots,(z_n,0))
\end{align*}
and
\begin{align*}
\tilde{\kappa}_n:\Sym^n(\mathbb{A}^1)\times\mathbb{A}^1&\hookrightarrow \Sym^n(\mathbb{A}^2)\\
((z_1,\ldots,z_n),z)&\mapsto ((z_1,z),\ldots,(z_n,z))
\end{align*}
the inclusions analogous to $\iota_n$ and $\tilde{\iota}_n$ defined in \eqref{iota_def} and \eqref{tiota_def}.  Let $l_n\colon\Sym^n(\mathbb{A}^1)\times\mathbb{G}_m\hookrightarrow\Sym^n(\mathbb{A}^1)\times\mathbb{A}^1$ be the inclusion.

\begin{lemma}
\label{twothirds}
\begin{enumerate}
    \item 
    There is an equality
    \[
    l_n^*\tilde{\kappa}_n^*\bigl[Q_L^{\bullet,\bullet,\nilp,n}\xrightarrow{\HC^{xy}} \Sym^n(\mathbb{A}^2)\bigr]_{\vir}=\kappa_n^{1,*}\bigl[Q_L^{\bullet,\bullet,\nilp,n}\xrightarrow{\HC^{xy}} \Sym^n(\mathbb{A}^2)\bigr]_{\vir}\boxtimes \bigl[\mathbb{G}_m\xrightarrow{\id}\mathbb{G}_m\bigr]
    \]
    in $\mathcal{M}_{\Sym^n(\mathbb{A}^1)\times\mathbb{G}_m}$. In other words, away from $0$, the motive 
    \[
    \tilde{\kappa}_n^*\bigl[Q_L^{\bullet,\bullet,\nilp,n}\xrightarrow{\HC^{xy}} \Sym^n(\mathbb{A}^2)\bigr]_{\vir}
    \]
    is constant along the $\mathbb{A}^1$-factor.
    \item
    More generally, there is an equality
    \begin{align*}
    \left[\mathcal{Q}^{n',\alpha}\xrightarrow{\HC^{xy}}\Sym^n(\mathbb{A}^2)\right]_{\vir}=&\left[\mathcal{Q}^{n',\emptyset}\xrightarrow{\HC^{xy}}\Sym^{n'}(\mathbb{A}^2)\right]_{\vir}\boxtimes_{\cup}\\&\cup_!\pi_{G_{\alpha}}j_{\alpha}^*\left(\underset{i\lvert\alpha_i\neq 0}{\Boxtimes}\left[\mathcal{Q}^{0,(i)} \xrightarrow{\HC^{xy}} \Sym^i(\mathbb{A}^1)\times\mathbb{A}^1\right]^{\otimes \alpha_i}_{\vir}\right).
    \end{align*}
    where $j_{\alpha}$ is the inclusion of the complement to the preimage of the big diagonal in $\mathbb{A}^{\sum_i\alpha_i}$ under the projection $\prod_i (\Sym^i(\mathbb{A}^1)\times\mathbb{A}^1)^{\alpha_i} \ra \mathbb{A}^{\sum_i\alpha_i}$.
\end{enumerate}
\end{lemma}

See the remarks following the statement of Lemma \ref{fixed_BBS}, as well as Equation \eqref{def:boxtimes_nu}, for some guidance on how to read the right hand side of the second equation.
\begin{proof}
Let $\mathbb{G}_m$ act on $\mathcal{N}_n^{\circ\circ}$ via 
\begin{multline*}
z\cdot (\rho(a'_1),\rho(a'_2),\rho(b'_1),\rho(a''_1),\rho(a''_2),\rho(b''_1))\\
=(z\rho(a'_1),z^{-1}\rho(a'_2),\rho(b'_1),z\rho(a''_1),z^{-1}\rho(a''_2),\rho(b''_1)).
\end{multline*}
Then $\Tr(W_r)$ is invariant under the $\mathbb{G}_m$-action.  We define the space $\mathcal{U}_n\subset \mathcal{N}_n^{\circ\circ}$ by the condition that $\det(\rho(a''_2))=1$.  Consider the morphism 
\[
J\colon \mathcal{U}_n\times\mathbb{G}_m\rightarrow \mathcal{N}_n^{\circ\circ}
\]
given by restricting the $\mathbb{G}_m$-action.  This is a Galois cover of the open subscheme $\mathcal{Y}\subset\mathcal{N}_n^{\circ\circ}$ defined by the condition that $\rho(a''_2)$ is invertible.  Let $\mathcal{V}_n\subset\mathcal{U}_n$ be the subvariety defined by the condition that $\rho(a''_2)$ has only one eigenvalue.  Then as a $\mu_n$-equivariant variety, there is an isomorphism
\[
\mathcal{V}_n\cong\mathcal{W}_n\times\mu_n
\]
where the action on the first factor is trivial, and is the action of group multiplication on the second.  Here $\mathcal{W}_n\subset \mathcal{U}_n$ is the subvariety defined by the condition that $\rho(a_2'')$ is unipotent.

Recall from Section \ref{direct_QW} the function $h_n$ on $\mathcal{N}_n^{\circ\circ}$.  The function $h_n\circ J$ factors through a function $\tilde{h}$ on $\mathcal{U}_n$, since $h$ is $\mathbb{G}_m$-invariant.  By Proposition \ref{BittnerProp}
\[
(\mathcal{Y}\rightarrow \mathcal{N}^{\circ\circ}_n)^*\bigl[\phi_{h_n}\bigr]_{\mathcal{N}^{\circ\circ}_n}=(\mathcal{Y}\rightarrow \mathcal{N}^{\circ\circ}_n)^*\pi_{\mu_n}J_!([\phi_{\tilde{h}}]_{\mathcal{U}_n}\boxtimes [\mathbb{G}_m\xrightarrow{\id_{\mathbb{G}_m}}\mathbb{G}_m]).
\]
It follows that 
\begin{align*}
    l_n^*\tilde{\kappa}_n^*\iota_n^*\HC_![\phi_{h_n}]_{\mathcal{N}_n^{\circ\circ}}&=l_n^*\tilde{\kappa}_n^*\iota_n^*\HC_!\pi_{\mu_n}J_!([\phi_{\tilde{h}}]_{\mathcal{U}_n}\boxtimes [\mathbb{G}_m\xrightarrow{\id_{\mathbb{G}_m}}\mathbb{G}_m])\\
    &=l_n^*\tilde{\kappa}_n^*\iota_n^*\HC_!\pi_{\mu_n}J_!([\phi_{\tilde h}]_{\mathcal{V}_n}\boxtimes [\mathbb{G}_m\xrightarrow{\id_{\mathbb{G}_m}}\mathbb{G}_m])\\
    &=l_n^*\tilde{\kappa}_n^*\iota_n^*\HC_!J_!([\phi_{\tilde h}]_{\mathcal{W}_n}\boxtimes [\mathbb{G}_m\xrightarrow{\id_{\mathbb{G}_m}}\mathbb{G}_m]),
\end{align*}
implying the triviality along the $\mathbb{G}_m$-factor required for part (1).  The second equality follows from the fact that only the part of the motive over $\mathcal{V}_n$ can contribute to the pullback along $\tilde{\kappa}_n$, since by definition this is the pullback along the locus where $\rho(a''_2)$ has only one eigenvalue.

\smallbreak
For part (2), consider the quiver $Q'_{\alpha}$ defined as follows.  Set
\begin{align*}
(Q'_{\alpha})_0&=\Set{\underline{2}_0,\ldots,\underline{2}_{l(\alpha)},\underline{\infty}}\\
(Q'_{\alpha})_1&=\Set{a'_{1,i},a'_{2,i},b'_{1,i},a''_{1,i}}_{0\leq i\leq l(\alpha)}\cup\Set{b''_{1,i,j},a''_{2,i,j}}_{0\leq i,j\leq l(\alpha)}
\end{align*}
with 
\begin{align*}
t(a'_{1,i})=t(a'_{2,i})=s(b'_{1,i})&=\underline{\infty}\\
s(a'_{1,i})=s(a'_{2,i})=t(b'_{1,i})=t(a''_{2,i})=s(a''_{2,i})&=\underline{2}_i\\
t(a''_{1,i,j})=t(b''_{1,i,j})&=\underline{2}_j\\
s(a''_{1,i,j})=s(b''_{1,i,j})&=\underline{2}_i.
\end{align*}
We give this quiver the dimension vector $\underline{d}=(n',1^{\alpha_1},\ldots,r^{\alpha_r},1)$ and the stability condition $(-1,\ldots,-1,n)$.  As in the proof of Lemma \ref{fixed_BBS}, a (stable) $Q'_{\alpha}$-representation corresponds to a ($\zeta'$-stable) $Q_r$-representation $\rho$ along with a direct sum decomposition of $\rho_2$ that is respected by the operator $\rho(a''_2)$.  A point in $\mathcal{Q}^{n',\alpha}$ gives rise to a $\underline{d}$-dimensional $Q'_{\alpha}$-representation, in a formal neighborhood of which the forgetful map is a Galois covering --- from this point the proof proceeds exactly as in the proof of Lemma \ref{fixed_BBS}.
\end{proof}

The following is proved in the same way as Corollary \ref{firstc}.

\begin{corollary}
\label{secondc}
Define classes $\Psi^{\#}_n \in \mathcal{M}_{\Sym^n(\mathbb{A}^1)}$ for $\#\in \set{\nilp,\unip}$, via
\[
\sum_{n\geq 0}\,(-1)^n\bigl[Q_L^{\bullet,\#,\nilp,n}\xrightarrow{\HC^x}\Sym^n(\mathbb{A}^1)\bigr]_{\vir}=\prExp_{\cup}\left(\sum_{n\geq 1}\Psi_n^{\#}\right).
\]
Then 
\begin{multline*}
\sum_{n\geq 0}\,(-1)^n\bigl[Q_L^{\bullet,\bullet,\nilp,n}\xrightarrow{\HC^{xy}}\Sym^n(\mathbb{A}^2)\bigr]_{\vir}\\
=\prExp_{\cup}\left(\sum_{n\geq 1}\left(\Sym^n(\mathbb{A}^1)\times\mathbb{G}_m\xrightarrow{\tilde{\kappa}_n\circ l_n}\Sym^n(\mathbb{A}^2)\right)_!\bigl(\Psi^{\unip}_n\boxtimes [\mathbb{G}_m\xrightarrow{\id}\mathbb{G}_m]\bigr) \right.\\
\left. {} +\left(\Sym^n(\mathbb{A}^1)\xrightarrow{\kappa^0_n}\Sym^n(\mathbb{A}^2)\right)_!\Psi^{\nilp}_n \vphantom{} \right).
\end{multline*}
\end{corollary}

We are now two thirds of the way towards showing that the relative DT invariants are fully punctual over $\Sym(\mathbb{A}^3)$.  Since the proof of Lemma \ref{finalred1} is almost identical to that of Lemma \ref{twothirds}, while the proof of Lemma \ref{finalred2} is strictly simpler, we omit them.  Before stating them, we introduce one last stratification; for $n'\leq n$, and $\alpha\vdash n-n'$, let
\[
\mathcal{P}^{n',\alpha}\subset Q_L^{\bullet,\nilp,\nilp,n}
\]
be the preimage under the map
\[
Q_L^{\bullet,\nilp,\nilp,n}\xrightarrow{\HC^x}\Sym^n(\mathbb{A}^1)
\]
of the space of tuples of length $n'$ at the origin, and for which the partition type of the tuple away from the origin is $\alpha$.  Likewise, for $\alpha\vdash n$ we define
\[
\mathcal{U}^{\alpha}\subset Q_L^{\bullet,\unip,\nilp,n}
\]
to be the preimage of $\Sym^{\alpha}(\mathbb{A}^1)$ under $Q_L^{\bullet,\unip,\nilp,n}\xrightarrow{\HC^x}\Sym^n(\mathbb{A}^1)$. Let also $\Delta_n\colon \A^1 \ra \Sym^n(\A^1)$ be the diagonal.

\begin{lemma}\label{finalred1}
\begin{enumerate}
    \item 
    There is an equality
    \begin{multline*}
    (\mathbb{G}_m\xrightarrow{\Delta_n\lvert_{\mathbb{G}_m}}\Sym^n(\mathbb{A}^1))^*\bigl[Q_L^{\bullet,\nilp,\nilp,n}\xrightarrow{\HC^x} \Sym^n(\mathbb{A}^1)\bigr]_{\vir}\\
    =(\{1\}\xrightarrow{\Delta_n\lvert_{\{1\}}}\Sym^n(\mathbb{A}^1))^*\bigl[Q_L^{\bullet,\nilp,\nilp,n}\xrightarrow{\HC^x} \Sym^n(\mathbb{A}^1)\bigr]_{\vir}\boxtimes \bigl[\mathbb{G}_m\xrightarrow{\id}\mathbb{G}_m\bigr]
    \end{multline*}
    in $\mathcal{M}_{\G_m}$.  In other words, away from $0$, the motive 
    $\Delta_n^*[Q_L^{\bullet,\nilp,\nilp,n}\rightarrow \Sym^n(\mathbb{A}^1)]_{\vir}$ is constant.
    \item
    Moreover, there is an equality 
    \begin{multline*}
    \left[\mathcal{P}^{n',\alpha}\xrightarrow{\HC^x}\Sym^n(\mathbb{A}^1)\right]_{\vir}\\
    = \left[\mathcal{P}^{n',\emptyset}\xrightarrow{\HC^x}\Sym^{n'}(\mathbb{A}^1)\right]_{\vir}\boxtimes_{\cup}\cup_!\pi_{G_{\alpha}}\left(\underset{i\lvert \alpha_i\neq 0}{\Boxtimes}
    \left[\mathcal{P}^{0,(i)} \xrightarrow{\HC^x} \Sym^{(i)}(\mathbb{A}^1)\right]^{\otimes\alpha_i}_{\vir}\right).
    \end{multline*}
\end{enumerate}
\end{lemma}

\begin{lemma}
\label{finalred2}
\begin{enumerate}
    \item 
    There is an equality
    \begin{multline*}
    (\mathbb{A}^1\xrightarrow{\Delta_n}\Sym^n(\mathbb{A}^1))^*\left[Q_L^{\bullet,\unip,\nilp,n}\xrightarrow{\HC^x} \Sym^n(\mathbb{A}^1)\right]_{\vir}\\
    =(\{0\}\xrightarrow{\Delta_n\lvert_{\{0\}}}\Sym^n(\mathbb{A}^1))^*\left[Q_L^{\bullet,\unip,\nilp,n}\xrightarrow{\HC^x} \Sym^n(\mathbb{A}^1)\right]_{\vir}\boxtimes \bigl[\mathbb{A}^1\xrightarrow{\id}\mathbb{A}^1\bigr].
    \end{multline*}
    in $\mathcal{M}_{\mathbb{A}^1}$.  In other words, the motive $\Delta_n^*[Q_L^{\bullet,\unip,\nilp,n}\rightarrow \Sym^n(\mathbb{A}^1)]_{\vir}$ is constant.
    \item
    Moreover, there is an equality
    \begin{align*}
    \left[\mathcal{U}^{\alpha}\xrightarrow{\HC^x}\Sym^n(\mathbb{A}^1)\right]_{\vir}=&\cup_!\pi_{G_{\alpha}}\left(\underset{i\lvert \alpha_i\neq 0}{\Boxtimes}\left[\mathcal{U}^{(i)} \xrightarrow{\HC^x} \Sym^{(i)}(\mathbb{A}^1)\right]^{\otimes\alpha_i}_{\vir}\right).
    \end{align*}
\end{enumerate}
\end{lemma}

By the same argument as Corollary \ref{firstc} we deduce the following.
\begin{corollary}
\label{thirdc1}
Define classes $\Lambda^{\#}_n$ for $\#\in \set{\nilp,\unip}$ in $\mathcal{M}_{\mathbb{C}}$ via
\begin{equation}
    \label{tceq1}
\sum_{n\geq 0}\,(-1)^n\bigl[Q_L^{\#,\nilp,\nilp,n}\bigr]_{\vir}=\prExp\left(\sum_{n\geq 1}\Lambda_n^{\#}\right).
\end{equation}
Then
\begin{multline*}
\sum_{n\geq 0}\,(-1)^n\bigl[Q_L^{\bullet,\nilp,\nilp,n}\xrightarrow{\HC^x}\Sym^n(\mathbb{A}^1)\bigr]_{\vir}\\
=\prExp_{\cup}\left(\sum_{n\geq 1}\Lambda_n^{\unip}\boxtimes \bigl[\mathbb{G}_m\xrightarrow{\Delta_n\lvert_{\mathbb{G}_m}}\Sym^n(\A^1)\bigr]+\Lambda_n^{\nilp}\boxtimes \bigl[0\hookrightarrow \Sym^n(\A^1)\bigr]\right).
\end{multline*}
\end{corollary}

\begin{corollary}
\label{thirdc2}
Define $\Lambda_n\in\mathcal{M}_{\mathbb{C}}$ via
\begin{equation}
\label{tceq2}
\sum_{n\geq 0}\,(-1)^n\bigl[Q_L^{\nilp,\unip,\nilp,n}\bigr]_{\vir}=\prExp\left(\sum_{n\geq 1}\Lambda_n\right).
\end{equation}
Then
\[
\sum_{n\geq 0}\,(-1)^n\bigl[Q_L^{\bullet,\unip,\nilp,n}\xrightarrow{\HC^x}\Sym^n(\mathbb{A}^1)\bigr]_{\vir}=\prExp_{\cup}\left(\sum_{n\geq 1}\Lambda_n\boxtimes \bigl[\A^1\xrightarrow{\Delta_n}\Sym^n(\A^1)\bigr]\right).
\]
\end{corollary}

\begin{lemma}
\label{sewing_lemma}
In the notation of Corollaries \ref{thirdc1} and \ref{thirdc2}, there is an equality
\[
\Lambda_n^{\unip}=\Lambda_n
\]
for all $n$.
\end{lemma}
\begin{proof}
There is an automorphism $\mathsf S_n$ of $\mathcal{N}^{\circ\circ}_n$, for all $n$, defined by
\begin{align*}
    \rho(a''_1)&\mapsto \rho(a''_2)\\
    \rho(a''_2)&\mapsto -\rho(a''_1)\\
    \rho(a'_1)&\mapsto \rho(a'_2)\\
    \rho(a'_2)&\mapsto-\rho(a'_1)
\end{align*}
which leaves $\Tr(W_r)$ invariant, and so preserves $[Q_L^n\rightarrow Q_L^n]_{\vir}$.  Now
\[
\mathsf S_{n,!}\left[Q_L^{\unip,\nilp,\nilp,n}\rightarrow Q_L\right]_{\vir}=\left[Q_L^{\nilp,\unip,\nilp,n}\rightarrow Q_L\right]_{\vir}
\]
and so $\bigl[Q_L^{\unip,\nilp,\nilp,n}\bigr]_{\vir}=\bigl[Q_L^{\nilp,\unip,\nilp,n}\bigr]_{\vir}$.  The lemma then follows from the defining equations (\ref{tceq1}) and (\ref{tceq2}).
\end{proof}

\begin{definition}\label{defin:Fully_Punctual}
Let us set 
\[
\Omega_{\pt}^n=\Lambda_n^{\unip},\quad \Omega_{\curv}^n=\Lambda_n^{\nilp}.
\]
These fully punctual motives express the contribution of points away from the curve and embedded on the curve, respectively.
\end{definition}
 Let
\[
    \iota_{L}\colon L\hookrightarrow \A^3,\quad 
    \iota_{C_0}\colon C_0\hookrightarrow X
\]
denote the closed inclusions, and let
\[
    u_{L}\colon \A^3\setminus L\hookrightarrow \A^3,\quad 
    u_{C_0}\colon X\setminus C_0\hookrightarrow X
\]
denote the inclusions of the open complements.
\begin{theorem}
\label{punctual_thm}
There is an equality in $\mathcal{M}_{\Sym(\A^3)}$:
\begin{multline}
    \label{firstpu}
\sum_{n\geq 0}\,(-1)^n\bigl[Q^n_L\xrightarrow{\HC_{\A^3}}\Sym^n(\A^3)\bigr]_{\vir}\\
=\prExp_{\cup}\left(\sum_{n\geq 1}\left(\Omega_{\pt}^n\boxtimes \bigl[\A^3\setminus L\xrightarrow{\Delta_n u_L}\Sym^n(\A^3)\bigr]+\Omega_{\curv}^n\boxtimes\bigl[L\xrightarrow{\Delta_n\iota_L}\Sym^n(\A^3)\bigr]\right)\right),
\end{multline}
and an equality in $\mathcal{M}_{\Sym(X)}$:
\begin{multline}
    \label{secpu}
\sum_{n\geq 0}\,(-1)^n\bigl[Q^n_{C_0}\xrightarrow{\HC_X}\Sym^n(X)\bigr]_{\vir}\\
=\prExp_{\cup}\left(\sum_{n\geq 1}\left(\Omega_{\pt}^n\boxtimes \bigl[X \setminus C_0\xrightarrow{\Delta_n u_{C_0}}\Sym^n(X)\bigr]+\Omega_{\curv}^n\boxtimes\bigl[C_0\xrightarrow{\Delta_n\iota_{C_0}}\Sym^n(X)\bigr]\right)\right).
\end{multline}
\end{theorem}
\begin{proof}
By Corollary \ref{firstc} we have
\begin{multline*}
\sum_{n\geq 0}\,(-1)^n\bigl[Q^n_L\xrightarrow{\HC_{\A^3}}\Sym^n(\mathbb{A}^3)\bigr]_{\vir}\\
=\prExp_{\cup}\left(\sum_{n\geq 1}\left(\Sym^n(\mathbb{A}^2)\times\mathbb{A}^1\xrightarrow{\tilde{\iota}_n} \Sym^n(\mathbb{A}^3)\right)_!\bigl(\Phi_n\boxtimes\, \bigr[\mathbb{A}^1\xrightarrow{\id}\mathbb{A}^1\bigr]\bigr)\right).
\end{multline*}
Then by Corollary \ref{secondc} and Proposition \ref{exp_prop}, we deduce that
\begin{multline*}
\sum_{n\geq 0}\,(-1)^n\bigl[Q^n_L\xrightarrow{\HC_{\A^3}}\Sym^n(\mathbb{A}^3)\bigr]_{\vir} \\
=\prExp_{\cup}\left(\sum_{n\geq 1}\,\bigl(\Sym^n(\A^1)\times \mathbb{G}_m\times\A^1\hookrightarrow \Sym^n(\A^3)\bigr)_!\Psi_n^{\unip}\boxtimes \bigl[\mathbb{G}_m\times \A^1\xrightarrow{\id}\mathbb{G}_m\times \A^1\bigr] \right. \\
+\left. {} \left(\Sym^n(\A^1)\times\{0\}\times\A^1\hookrightarrow \Sym^n(\A^3)\right)_!\Psi_n^{\nilp}\boxtimes \bigl[\{0\}\times\A^1\xrightarrow{\id}\{0\}\times\A^1\bigr] \vphantom{} \right).
\end{multline*}
Then by Corollaries \ref{thirdc1} and \ref{thirdc2} and Proposition \ref{exp_prop} again, we deduce that 
\begin{multline*}
\sum_{n\geq 0}\,(-1)^n\bigl[Q^n_L\xrightarrow{\HC_{\A^3}}\Sym^n(\mathbb{A}^3)\bigr]_{\vir}=\\
\prExp_{\cup}\Biggl(\sum_{n\geq 1}\left(\A^1\times\mathbb{G}_m\times\A^1\hookrightarrow\Sym^n(\A^3)\right)_!\Omega_{\pt}^n\boxtimes \bigl[\A^1\times\mathbb{G}_m\times\A^1\xrightarrow{\id}\A^1\times\mathbb{G}_m\times\A^1\bigr] \\
+\left(\mathbb{G}_m\times\{0\}\times\A^1\hookrightarrow\Sym^n(\A^3)\right)_!\Omega_{\pt}^n\boxtimes \bigl[\mathbb{G}_m\times\{0\}\times\A^1\xrightarrow{\id}\mathbb{G}_m\times\{0\}\times\A^1\bigr]\\
+\left(\{0\}\times\{0\}\times\A^1\hookrightarrow\Sym^n(\A^3)\right)_!\Omega_{\curv}^n\boxtimes \bigl[\{0\}\times\{0\}\times\A^1\xrightarrow{\id}\{0\}\times\{0\}\times\A^1\bigr]\Biggr),
\end{multline*}
and the first statement follows.  For the second statement, we observe that the right hand side of (\ref{secpu}) satisfies conditions (2) and (3) of Lemma \ref{res_lemma}, and is obviously equal to the right hand side of (\ref{firstpu}) after restriction to $\Sym(\A^3)$.  Then the second statement follows from Corollary \ref{res_cor}.
\end{proof}

\section{The virtual motive of the Quot scheme}

In this section we define a virtual motive
\[
\bigl[Q^n_C\bigr]_{\vir} \in \mathcal M_{\C}
\]
for an arbitrary smooth curve $C\subset Y$ in a smooth $3$-fold $Y$. Before getting to this point, we give an explicit formula for the generating series
\begin{align*}
    \mathsf Q_{L/\A^3}(t) &= \sum_{n\geq 0}\,\bigl[Q_L^n\bigr]_{\vir}\cdot t^n \\
    \mathsf Q_{C_0/X}(t)  &= \sum_{n\geq 0}\,\bigl[Q_{C_0}^n\bigr]_{\vir}\cdot t^n
\end{align*}
encoding the local absolute virtual motives attached to $L\subset \A^3$ and $C_0\subset X$. We shall also prove Theorems \ref{thm:newB} and \ref{thm:thm2} from the introduction.

\subsection{The local absolute virtual motives}

Let $X$ be, as usual, the resolved conifold. Then Theorem \ref{punctual_thm} implies that for $\iota\colon x\hookrightarrow X$ the inclusion of a point, the absolute motive 
\[
(\Sym^n\iota)^*\bigl[Q^n_{C_0}\xrightarrow{\HC_{X}}\Sym^n(X)\bigr]_{\vir}
\]
only depends on whether $x$ is in $C_0$ or not.  We define \bitem
\item
$\bigl[\mathcal{P}_{\pt}^{n}\bigr]_{\vir}=(\Sym^n\iota)^*\bigl[Q^n_{C_0}\xrightarrow{\HC_{X}}\Sym^n(X)\bigr]_{\vir}$ for $x\notin C_0$,

\item
$\bigl[\mathcal{P}_{\curv}^{n}\bigr]_{\vir}=(\Sym^n\iota)^*\bigl[Q^n_{C_0}\xrightarrow{\HC_{X}}\Sym^n(X)\bigr]_{\vir}$ for $x\in C_0$.
\eitem
Explicitly, these motives are determined by the identities
\[
\sum_{n\geq 0}(-1)^n\bigl[\mathcal{P}_{\#}^{n}\bigr]_{\vir}\cdot t^n=\prExp\left(\sum_{n>0}\Omega_{\#}^n\cdot t^n\right)
\]
for $\# \in \set{\pt,\curv}$. 
We define the generating functions
\[
\mathsf F_{\#}(t)=\sum_{n\geq 0}\,\bigl[\mathcal{P}_{\#}^{n}\bigr]_{\vir}\cdot t^n.
\]
Via the morphisms defining power structures for varieties, Lemmas \ref{fixed_BBS}, \ref{twothirds}, \ref{finalred1}, \ref{finalred2} and \ref{sewing_lemma} can be neatly summed up in the following corollary.

\begin{corollary}\label{corollary_Punctual}
There are equalities in $\mathcal{M}_{\C}\llbracket t \rrbracket$
\begin{align*}
    \mathsf Q_{L/\A^3}(-t)&=\mathsf F_{\pt}(-t)_{\pr}^{\mathbb{L}^3-\mathbb{L}}\cdot \mathsf F_{\curv}(-t)_{\pr}^\mathbb{L}\\
    \mathsf Q_{C_0/X}(-t)&=\mathsf F_{\pt}(-t)_{\pr}^{\mathbb{L}^3+\mathbb{L}^2-\mathbb{L}-1}\cdot \mathsf F_{\curv}(-t)_{\pr}^{\mathbb{L}+1}.
\end{align*}
\end{corollary}

\begin{remark}
Note that the exponents in the statement of the corollary are effective (despite the minus signs). Indeed, $\mathbb L^3-\mathbb L$ is the class of $\mathbb A^3\setminus L$, and $\mathbb{L}^3+\mathbb{L}^2-\mathbb{L}-1$ is the class of $X\setminus C_0$. The `$\pr$' subscripts are in the statement because we are yet to prove that the classes $(-1)^n[\mathcal{P}_{\pt}^{n}]_{\vir}$ and $(-1)^n[\mathcal{P}_{\curv}^{n}]_{\vir}$ are effective. 
As a result of Corollaries \ref{Omega_pt_calc} and \ref{Omega_curv_calc}, we will be able to remove them and the statements remain true, now expressed properly in terms of the power structures on the Grothendieck rings of varieties.
\end{remark}

The rest of this subsection is devoted to removing the decorations `$\pr$' from the formulas in Corollary \ref{corollary_Punctual}. We start with $\# = \pt$.

\begin{prop}\label{prop:remove_pr_pt}
For all $n\geq 0$ there is an identity
\[
\bigl[\mathcal{P}_{\pt}^{n}\bigr]_{\vir}=\bigl[\Hilb^n(\mathbb{A}^3)_0\bigr]_{\vir} \in \mathcal M_{\C}.
\]
\end{prop}

\begin{proof}
By definition 
\[
\bigl[\mathcal{P}_{\pt}^{n}\bigr]_{\vir}=j_n^*\bigl[Q_L^n\xrightarrow{\HC}\Sym^n(\mathbb{A}^3)\bigr]_{\vir}
\]
where $j_n\colon p\hookrightarrow \Sym^n(\A^3)$ is the inclusion of the point $p=((1,1,1),\ldots,(1,1,1))$. Consider the open subvariety $\overline{U}_n\subset \Rep_{(n,1)}(Q_r)$ defined by the condition that $\rho(a''_1)$ and $\rho(a''_2)$ are invertible, and the image of $\rho(a'_1)$ generates $\rho_{\underline{2}}$ under the action of $\rho(a''_1)$, $\rho(a''_2)$ and $\rho(b''_1)$.  Note that this is a \textit{stronger} notion of stability than $\zeta'$-stability.  However, after restricting to the critical locus of $\Tr(W_r)$, using our invertibility assumptions, we have that for $\rho\in \overline{U}_n$ the relation
\[
\rho(a''_1)^{-1}\rho(a''_2)\rho(a'_1)=\rho(a'_2)
\]
holds, and so we deduce that $\rho$ satisfies our stronger notion of stability if it is $\zeta'$-stable, lies in $\crit(\Tr(W_r))$, and satisfies the above invertibility assumptions.  It follows that, after setting $U_n = \overline{U}_n/\GL_n$, one has
\[
\bigl[\mathcal{P}_{\pt}^{n}\bigr]_{\vir}=j_n^*\bigl[U_n\cap \crit(\Tr(W_r))\xrightarrow{\HC}\Sym^n(\mathbb{A}^3)\bigr]_{\vir}.
\]
Let $Q_{\BBS}$ denote the quiver obtained from $Q_r$ by removing $a'_2$ and $b'_1$, in other words the framed $3$-loop quiver of Figure \ref{3LoopQuiver}.  Then $\NCHilb^n\defeq \Rep_{(n,1)}^{\zeta'}(Q_{\BBS})/\GL_n$ is the noncommutative Hilbert scheme considered by Behrend, Bryan and Szendr\H{o}i in \cite{BBS} (cf.~Section \ref{Section:BBS}), and $U_n=\overline{U}_n/\GL_n$ is a vector bundle over $\NCHilb_{\circ}^n$, the open subscheme of $\NCHilb^n$ defined by invertibility of $\rho(a''_1)$ and $\rho(a''_2)$.  

We introduce a new set of matrix coordinates on $\overline{U}_n$ by considering the matrices
\begin{align*}
\mathbb{A}''_1=\rho(a''_1)&&\mathbb{A}''_2=\rho(a''_2)&&\mathbb{B}''_1=\rho(b''_1)\\
\mathbb{A}'_2=\rho(a''_2)\rho(a'_1)-\rho(a''_1)\rho(a'_2)&&\mathbb{A}'_1=\rho(a''_1)\rho(a'_2)&&\mathbb{B}'_1=\rho(b'_1).
\end{align*}
With respect to these coordinates, we can write
\[
\Tr(W_r)\lvert_{U_n}=F+\Tr(\mathbb{B}'_1\mathbb{A}'_2)
\]
where
\[
F=\Tr(\mathbb{A}''_1)\Tr(\mathbb{B}''_1)\Tr(\mathbb{A}''_2)-\Tr(\mathbb{A}''_2)\Tr(\mathbb{B}''_1)\Tr(\mathbb{A}''_1)
\]
is a function on $\NCHilb^n$ considered as a function on $U_n$ by composition with the vector bundle projection $U_n\rightarrow \NCHilb^n$.  It follows from Lemma \ref{quad_lemma}\,(2) and the motivic Thom--Sebastiani theorem that 
\begin{align*}
j_n^*\bigl[U_n\cap \crit(\Tr(W_r))\xrightarrow{\HC}\Sym^n(\A^3)\bigr]_{\vir}&=j_n^*\bigl[\NCHilb^n\cap \crit(\Tr(F))\xrightarrow{\HC}\Sym^n(\A^3)\bigr]_{\vir}\\
&=j^*_n\bigl[\Hilb^n(\mathbb{A}^3)\xrightarrow{\HC}\Sym^n(\mathbb{A}^3)\bigr]_{\vir}\\
&=\bigl[\Hilb^n(\mathbb{A}^3)_0\bigr]_{\vir}
\end{align*}
as required.
\end{proof}

\begin{corollary}
\label{Omega_pt_calc}
For all $n\geq 0$ there is an identity
\[
\Omega^n_{\pt}=(-1)^n\mathbb{L}^{-\frac{3}{2}}\frac{\mathbb{L}^{\frac{n}{2}}-\mathbb{L}^{-\frac{n}{2}}}{\mathbb{L}^{\frac{1}{2}}-\mathbb{L}^{-\frac{1}{2}}} \in \mathcal M_{\C}.
\]
In particular, $\Omega_{\pt}^n$ is effective.
\end{corollary}

\begin{proof}
By definition we have
\begin{align*}
\sum_{n\geq 0}(-1)^n\bigl[\mathcal{P}_{\pt}^n\bigr]_{\vir}t^n=&\prExp\left(\sum_{n> 0}\Omega^n_{\pt}t^n\right).
\end{align*}
Since $[\mathcal{P}_{\pt}^n]_{\vir}=[\Hilb^n(\mathbb{A}^3)_0]_{\vir}$ we deduce (see Remark \ref{punctualBBS_effective}) that $(-1)^n[\mathcal{P}_{\pt}^n]_{\vir}$ is effective, and by Formula \eqref{punctual_BBS} we deduce also that
\begin{align*}
\sum_{n\geq 0}(-1)^n\bigl[\mathcal{P}_{\pt}^n\bigr]_{\vir}t^n=&\Exp\left(\sum_{n> 0}(-1)^n\mathbb{L}^{-\frac{3}{2}}\frac{\mathbb{L}^{\frac{n}{2}}-\mathbb{L}^{-\frac{n}{2}}}{\mathbb{L}^{\frac{1}{2}}-\mathbb{L}^{-\frac{1}{2}}}t^n\right).
\end{align*}
The result follows since $\prExp$ is injective, and agrees with $\Exp$ on effective motives.
\end{proof}

The corollary implies that in the right hand side of the formulas in Corollary \ref{corollary_Punctual}, we can remove the decoration `$\pr$' from the first factor. In other words, for all smooth quasi-projective $3$-folds $U$ one has 
\be\label{Points_no_pr_anymore}
\mathsf F_{\pt}(-t)^{[U]}_{\pr} = \mathsf F_{\pt}(-t)^{[U]} = \mathsf Z_0(-t)^{[U]} = \mathsf Z_U(-t).
\ee
Next we deal with $\# = \curv$.

\smallbreak
In \cite[Prop.~4.3]{RefConifold} the full motivic DT and PT theories of the resolved conifold $X$ are computed. The sign conventions in \emph{loc.~cit.}~are different from ours, but the discrepancy amounts to the substitution
\[
\L^{\frac{1}{2}} \ra -\L^{\frac{1}{2}}.
\]
After this change is done, the motivic partition function of the stable pair theory of $X$ reads
\[
\mathsf Z_{\PT}(-s,T) = \prod_{m\geq 1}\prod_{j=0}^{m-1}\left(1+\L^{-\frac{m}{2}+\frac{1}{2}+j}(-s)^mT\right).
\]
On the other hand, the DT partition $\mathsf Z_{\DT}$ function satisfies
\[
\mathsf Z_{\DT}(-s,T) = \mathsf Z_X(-s)\cdot \mathsf Z_{\PT}(-s,T),
\]
where $s$ is the point class and $T$ is the curve class.
Again, we have adjusted the sign in the point contribution $\mathsf Z_X$.
Extracting the coefficient of $T$ (which corresponds to picking out the contribution of the reduced curve class) and multiplying by $-s^{-1}$ (so that the $(-s)^n$ coefficient is $I_{n+1}(X,[C_0]) = Q^n_{C_0}$ on the DT side) yields an identity
\be\label{WC_formula}
\DDT_{C_0/X}(-s) = \mathsf Z_X(-s)\cdot \PPT_{C_0/X}(-s),
\ee
where 
\begin{align*}
\PPT_{C_0/X}(-s) &= \sum_{n\geq 0}\,(-1)^n\bigl[\Sym^n C_0\bigr]_{\vir}\cdot s^n \\
&= (1+\L^{-\frac{1}{2}}s)^{-(\L+1)} \\
&=(1-\mathbb{L}^{-\frac{1}{2}}s+\mathbb{L}^{-1}s^2-\cdots)^{\mathbb{L}+1}.
\end{align*}
On the other hand, 
\[
\DDT_{C_0/X}(-s) = \sum_{n\geq 0}\,(-1)^n\bigl[Q^n_{C_0}\bigr]_{\vir}\cdot s^n = \mathsf Q_{C_0/X}(-s),
\]
and the latter equals
\[
\mathsf Z_{X\setminus C_0}(-s) \cdot \mathsf F_{\curv}(-s)^{\L+1}_{\pr} = \mathsf Z_X(-s)\cdot \frac{\mathsf F_{\curv}(-s)^{\L+1}_{\pr}}{\mathsf Z_0(-s)^{\L+1}}
\]
by Corollary \ref{corollary_Punctual} and Proposition \ref{prop:remove_pr_pt}. It follows that
\begin{align} \label{F_factorized}
\mathsf F_{\curv}(-s)^{\L+1}_{\pr} &= \mathsf Z_0(-s)^{\L+1} \cdot (1-\L^{-\frac{1}{2}}s+\L^{-1} s^2-\cdots)^{\L+1},
\\
\nonumber
&=\left(\mathsf Z_0(-s) \cdot (1-\L^{-\frac{1}{2}}s+\L^{-1} s^2-\cdots)\right)^{\L+1}.
\end{align}
We shall need the following:

\begin{lemma}\label{lemma_F(s)}
There is an identity 
\[
\mathsf F_{\curv}(-s) = \mathsf Z_0(-s) \cdot (1+\L^{-\frac{1}{2}}s)^{-1} \in \mathcal{M}_{\mathbb{C}}\llbracket s\rrbracket.
\]
In particular, $\mathsf F_{\curv}(-s)$ is effective.
\end{lemma}
\begin{proof}
Since $\mathsf Z_0(-s) \cdot (1-\L^{-\frac{1}{2}}s+\L^{-1} s^2-\cdots)$ is effective we have
\[
\left(\mathsf Z_0(-s) \cdot (1-\L^{-\frac{1}{2}}s+\L^{-1} s^2-\cdots)\right)^{\mathbb{L}+1}=\left(\mathsf Z_0(-s) \cdot (1-\L^{-\frac{1}{2}}s+\L^{-1} s^2-\cdots)\right)_{\pr}^{\mathbb{L}+1}.
\]
Since $\L+1$ is invertible in $K(\Staff{\C})$, by Lemma \ref{pr_inj_lemma} and (\ref{F_factorized}) we deduce the lemma.
\end{proof}

\begin{corollary}
\label{Omega_curv_calc}
There is an equality of motives
\[
\Omega^n_{\curv}=
\begin{cases} 
-\mathbb{L}^{-\frac{1}{2}}-\mathbb{L}^{-\frac{3}{2}}&\textrm{if }n=1\\
(-1)^n\mathbb{L}^{-\frac{3}{2}}\displaystyle\frac{\L^{\frac{n}{2}}-\L^{-\frac{n}{2}}}{\L^{\frac{1}{2}}-\L^{-\frac{1}{2}}}&\textrm{otherwise.} 
\end{cases}
\]
In particular, $\Omega_{\curv}^n$ is effective.
\end{corollary}
\begin{proof}
From Lemma \ref{lemma_F(s)} and the equations
\begin{align*}
    \bigl(1+\mathbb{L}^{-\frac{1}{2}}s\bigr)^{-1}&=\Exp\bigl(-\mathbb{L}^{-\frac{1}{2}}s\bigr)\\
    \mathsf Z_0(-s)&=\Exp\left(\sum_{n\geq 1}\mathbb{L}^{-\frac{3}{2}}\displaystyle\frac{\L^{\frac{n}{2}}-\L^{-\frac{n}{2}}}{\L^{\frac{1}{2}}-\L^{-\frac{1}{2}}}(-s)^n\right)
\end{align*}
we deduce that $$\mathsf{F}_{\curv}(-s)=\Exp\left(P\right)$$
where 
\[
P=\left(\sum_{n\geq 1}\mathbb{L}^{-\frac{3}{2}}\displaystyle\frac{\L^{\frac{n}{2}}-\L^{-\frac{n}{2}}}{\L^{\frac{1}{2}}-\L^{-\frac{1}{2}}}(-s)^n\right)-\mathbb{L}^{-\frac{1}{2}}s.
\]
Since $P$ is effective, $\prExp(P)=\Exp(P)$.  On the other hand, by definition we have the equality $\mathsf F_{\curv}(-s)=\prExp\left(\sum \Omega_{\curv}^ns^n\right)$.  The result then follows by injectivity of $\prExp$.
\end{proof}

Corollary \ref{corollary_Punctual} can now be restated as follows:

\begin{theorem}\label{Thm:VM_local}
The absolute virtual motives of $L\subset \A^3$ and $C_0\subset X$ are given by
\begin{align*}
    \mathsf Q_{L/\A^3}(-t) &= \mathsf Z_{\A^3\setminus L}(-t) \cdot \mathsf F_{\curv}(-t)^{\L} \\
    \mathsf Q_{C_0/X}(-t) &= \mathsf Z_{X\setminus C_0}(-t) \cdot \mathsf F_{\curv}(-t)^{\L+1}.
\end{align*}
\end{theorem}

\begin{proof}
Combining \eqref{Points_no_pr_anymore} with Lemma \ref{lemma_F(s)}, we get $\mathsf F_{\curv}(-t)^{\L+1} = \mathsf F_{\curv}(-t)^{\L+1}_{\pr}$.
\end{proof}

\subsection{Proof of Theorems \ref{thm:newB} and \ref{thm:thm2}}\label{sec:Thm_B&C}

\begin{proofof}{Theorem \ref{thm:newB}}
By \eqref{firstpu} there is an equality
\begin{multline*}
\sum_{n\geq 0}\,(-1)^n\bigl[Q^n_L\xrightarrow{\HC_{\A^3}}\Sym^n(\A^3)\bigr]_{\vir} \\
=\prExp_{\cup}\left(\sum_{n\geq 1}\left(\Omega_{\pt}^n\boxtimes \bigl[\A^3\setminus L\xrightarrow{\Delta_n u_L}\Sym^n(\A^3)\bigr]+\Omega_{\curv}^n\boxtimes\bigl[L\xrightarrow{\Delta_n\iota_L}\Sym^n(\A^3)\bigr]\right)\right).
\end{multline*}
Plugging in the results of Corollaries \ref{Omega_curv_calc} and \ref{Omega_pt_calc} we deduce
\begin{align*}
    \sum_{n\geq 0}\,(-1)^n&\bigl[Q^n_L\xrightarrow{\HC_{\A^3}}\Sym^n(\A^3)\bigr]_{\vir}\\
&=\prExp_{\cup}\Biggl(\sum_{n\geq 1}\Bigl(\Omega_{\BBS}^n\boxtimes \bigl[\A^3\setminus L\xrightarrow{\Delta_n u_L}\Sym^n(\A^3)\bigr] \\
& \qquad \qquad +\Omega_{\BBS}^n\boxtimes\bigl[L\xrightarrow{\Delta_n\iota_L}\Sym^n(\A^3)\bigr]\Bigr)-\mathbb{L}^{-\frac{1}{2}}\bigl[L\xrightarrow{\Delta_1\iota_L}\Sym^1(\mathbb{A}^3)\bigr]\Biggr)\\
&=\Exp_{\cup}\Biggl(\sum_{n\geq 2}\left(\Omega_{\BBS}^n\boxtimes \bigl[\A^3\xrightarrow{\Delta_n }\Sym^n(\A^3)\bigr]\right)\\
&\qquad\qquad -\mathbb{L}^{-\frac{1}{2}}\bigl[L\xrightarrow{\Delta_1\iota_L}\Sym^1(\mathbb{A}^3)\bigr]-\mathbb{L}^{-\frac{3}{2}}\bigl[\A^3\xrightarrow{\Delta_1 }\Sym^1(\A^3)\bigr]\Biggr)
\end{align*}
where the removal of the `$\pr$' comes from the effectiveness statements in Corollaries \ref{Omega_curv_calc} and \ref{Omega_pt_calc}.
\end{proofof}

Guided by Theorem \ref{Thm:VM_local}, we can now define a virtual motive for an arbitrary Quot scheme $Q^n_C$.

\begin{definition}\label{def:VM_arbitrary_curve}
Let $Y$ be a smooth $3$-fold. For a smooth curve $C\subset Y$, we define classes $[Q^n_{C}]_{\vir} \in \mathcal M_{\C}$ by the identity
\[
\sum_{n\geq 0}\,\bigl[Q^n_{C}\bigr]_{\vir} (-t)^n = \mathsf Z_{Y\setminus C}(-t)\cdot\mathsf F_{\curv}(-t)^{[C]}.
\]
We also define 
\[
\mathsf Q_{C/Y}(t) = \sum_{n\geq 0}\,\bigl[Q^n_C\bigr]_{\vir}\cdot t^n.
\]
\end{definition}

At this point, it is not yet clear that $[Q^n_C]_{\vir}$ is a virtual motive for $Q^n_C$, but we incorporate this in the proof of Theorem \ref{thm:thm2} from the introduction:

\begin{theorem}\label{thm:C}
For all $n\geq 0$, there is an equality of motives
\[
\bigl[Q^n_C\bigr]_{\vir} = \sum_{j=0}^n\,\bigl[\Hilb^{n-j}Y\bigr]_{\vir}\cdot \bigl[\Sym^jC\bigr]_{\vir}.
\]
In other words, we have a product decomposition
\[
\mathsf Q_{C/Y}(t) = \mathsf Z_Y(t)\cdot \mathsf Z_C(t)
\]
where for a smooth variety $X$ of dimension at most $3$, $\mathsf Z_X(t)$ denotes the motivic partition function of the Hilbert scheme of points of $X$.
\end{theorem}

\begin{proof}
Combining the power structure with Lemma \ref{lemma_F(s)}, we find
    \begin{align*}
    \mathsf Q_{C/Y}(-t) &= \mathsf Z_{Y\setminus C}(-t)\cdot\mathsf F_{\curv}(-t)^{[C]} \\
    &=\mathsf Z_{0}(-t)^{[Y]-[C]}\cdot \mathsf Z_0(-t)^{[C]}\cdot \bigl(1+\L^{-\frac{1}{2}}t\bigr)^{-[C]} \\
    &= \mathsf Z_Y(-t) \cdot \mathsf Z_C(-t).
\end{align*}
In particular, the classes $[Q^n_{C}]_{\vir}$ are virtual motives, because
\[
\chi \mathsf Q_{C/Y}(t) = \chi\mathsf Z_Y(t) \cdot \chi \mathsf Z_C(t) = M(-t)^{\chi(Y)}\cdot (1+t)^{-\chi(C)}
\]
and this equals $\sum_n\widetilde\chi(Q^n_C)t^n$ by \cite[Prop.~5.1]{LocalDT}.
\end{proof}

The proof of Theorem \ref{thm:thm2} is complete.

\subsubsection{Equivalent formulations}
Using the power structure on $\mathcal M_\C$ and the explicit formulas for $\mathsf Z_X(t)$ available from \cite[Section 4]{BBS}, we find an identity
\[
\mathsf Q_{C/Y}(t) = \left( \prod_{m=1}^\infty\prod_{k=0}^{m-1}\left(1-\L^{k-1-\frac{m}{2}}t^m\right)^{-[Y]}\right)\cdot \left(1-\L^{-\frac{1}{2}}t\right)^{-[C]}.
\]
Another equivalent way to express the same identity is via motivic exponentials. If $t$ is the variable used in the definition of motivic exponential (cf.~Section \ref{subsec:Exp}), then
\begin{align*}
    \mathsf Q_{C/Y}(-t) &= \Exp\left(\frac{-t[Y]_{\vir}}{\bigl(1+\L^{\frac{1}{2}}t\bigr)\bigl(1+\L^{-\frac{1}{2}}t\bigr)}-t[C]_{\vir}\right) \\
    &=\Exp\left(-t[Y]_{\vir}\Exp(-t[\P^1]_{\vir})-t[C]_{\vir}\right).
\end{align*}

\subsection{Local Donaldson--Thomas invariants}\label{Sec:DT}
Let $Y$ be a projective Calabi--Yau $3$-fold, and let $C\subset Y$ be a smooth curve of genus $g$. Recall the $C$-local DT invariants
\[
\DDT^n_{C} = \int_{Q^n_C} \nu_I \,\mathrm{d}\chi
\]
defined by restricting the Behrend function of the Hilbert scheme $I=I_{1-g+n}(Y,[C])$ to its closed subset $|Q^n_C|\subset I$. The BPS number $n_{g,C}$ of $C\subset Y$ is the integer
\[
n_{g,C} = \nu_{I_{1-g}(Y,[C])}(\mathscr I_C) \in \Z.
\]

Theorem \ref{thm:C} immediately implies the following:

\begin{corollary}\label{cor_BPS}
Let $Y$ be a projective Calabi--Yau $3$-fold, $C\subset Y$ a smooth curve with BPS number $n_{g,C}=1$. Then 
\[
\chi\left[Q^n_C\right]_{\vir} = \DDT^n_{C}.
\]
\end{corollary}

\begin{proof}
The main result of  \cite{Ricolfi2018} proves that
\[
\DDT^n_{C} = n_{g,C}\cdot \widetilde\chi(Q^n_C).
\]
By the proof of Theorem \ref{thm:C} we know that $[Q^n_C]_{\vir}$ is a virtual motive, so the result follows.
\end{proof}
 
When $C$ is infinitesimally rigid in $Y$, i.e.~$H^0(C,N_{C/Y})=0$, the integer $\DDT^n_{C}$ is the degree of the virtual fundamental class 
\[
\left[Q^n_C\right]^{\vir}\in A_0(Q^n_C),
\]
naturally defined (by restriction) on the \emph{connected component} 
\[
Q^n_C\subset I_{1-g+n}(Y,[C]).
\]
So, by Corollary \ref{cor_BPS}, the class $[Q^n_{C}]_{\vir}\in\mathcal M_\C$ can be seen as a motivic Donaldson--Thomas invariant.

\begin{remark}\label{rem:vm3eq74}
In \cite[Example $5.7$]{CohDT} one can find an example of a cohomological DT invariant in the projective case.
We are not aware of other explicit examples of motivic DT invariants for \emph{projective} Calabi--Yau $3$-folds, in a setting where the moduli space parameterises \emph{curves and points}.
Without a curve in the picture, there is the virtual motive $[\Hilb^nY]_{\vir}$ constructed in \cite{BBS} for arbitrary $3$-folds, and if $Y$ is an \emph{open} Calabi--Yau there are plenty of examples, see for instance \cite{RefConifold,MorrNagao,BenSven2,BenSven3,Mozgovoy1}.
\end{remark}

\begin{remark}
The formula  $\mathsf Q_{L/\A^3} = \mathsf Z_{\A^3}\cdot \mathsf Z_L$ was conjectured in the second author's PhD thesis \cite{ThesisR}, where the problem was reduced to proving a motivic identity (with no ``virtualness'' left to deal with) involving essentially only the stack of finite length coherent sheaves over $\A^2$. As it turned out, such an identity could in principle be checked by hand after performing a complete classification of finite length modules over the polynomial ring $\C[x,y]$, which is known to be a \emph{wild} problem. This classification was accomplished by Moschetti and the second author in \cite{MR18} for modules of length $n\leq 4$. As a consequence, the motivic wall-crossing formula $\mathsf Q_{L/\A^3} = \mathsf Z_{\A^3}\cdot \mathsf Z_L$ could be proven by hand up to order $4$ using this classification.
\end{remark}

\section{Categorification}
In this final section, we outline a programme for future work on a \textit{categorified} DT/PT correspondence.  Just as, since any smooth threefold analytically locally looks like $\mathbb{A}^3$, one may patch together the full cohomological understanding of the DT theory of degree zero DT theory on $\A^3$ obtained in \cite{BenPreproj} to try to understand the degree zero cohomological DT theory of \textit{any} Calabi--Yau threefold (see \cite[Sec.~3.1]{TodaGVWCF}), analytically locally, the inclusion of a smooth curve inside a smooth Calabi--Yau threefold can be modelled by $L\subset \A^3$, and so the intention is that by proving the conjectures below we can approach a cohomological version of the DT/PT correspondence.
\subsection{Further directions I: categorification}
\label{further_sec_1}
By our Theorem \ref{thm:thm1}, the moduli space $Q_L^n$ arises as the critical locus of the function $\Tr(W_r)$ on the moduli space $\mathcal{N}^{\circ\circ}_n$.  Shifting the associated virtual motives, we define the generating function
\[
Q_{L/\mathbb{A}^3}^{\twist}=\sum_{n\geq 0}\mathbb{L}^{-\frac{n}{2}}\bigl[Q_L^n\xrightarrow{\HC_n}\Sym^n\A^3\bigr]_{\vir}\in\mathcal{M}_{\Sym(\A^3)}.
\]
Then by Theorem \ref{thm:newB} we can write
\begin{equation}
    \label{pre_sym}
Q_{L/\mathbb{A}^3}^{\twist}=\Exp_\cup\left(\sum_{n>0} \Delta_{n\,!}\left(\Omega_n^{\twist}\boxtimes \bigl[\A^3\xrightarrow{\id}\A^3\bigr]   \right)\right) \boxtimes_{\cup} 
\Exp_\cup\left(\Delta_{1\,!}\left(\mathbb{L}^{-1}\boxtimes\left[L\hookrightarrow \A^3\right]\right)\right)
\end{equation}
where 
\[
\Omega^{\twist}_n=\mathbb{L}^{-2}(1+\mathbb{L}^{-1}+\cdots+\mathbb{L}^{1-n})
\]
and $\Delta_n:\mathbb{A}^3\rightarrow \Sym^n(\A^3)$ is, as ever, the inclusion of the small diagonal.

For the rest of this section we will make free use of the language and foundational results concerning monodromic mixed Hodge modules, see \cite{Saito89,Saito90} for background on mixed Hodge modules, \cite{KSCOHA} for an introduction to monodromic mixed Hodge structures, and \cite{BenSvenQEA} for the theory of monodromic mixed Hodge modules in cohomological Donaldson--Thomas theory.  In particular, we will use the same functor 
\[
\phim{\Tr(W_r)}\colon \MHM(\mathcal{N}^{\circ\circ}_n)\rightarrow \MMHM(\mathcal{N}^{\circ\circ}_n)
\]
considered in \cite[Sec.2.1]{BenSvenQEA}, following the discussion of a monodromic mixed Hodge structure on vanishing cycle cohomology in \cite[Sec.~7.4]{KSCOHA}.

In what follows, for a space $X$ we write $\underline{\mathbb{Q}}_X$ for the constant complex of mixed Hodge modules on $X$, which we think of as a complex of monodromic mixed Hodge modules by endowing it with the trivial monodromy operator.  Where the choice of space $X$ is clear, we may drop the subscript.

Given an element $[Y\xrightarrow{f}Z]\in \mathcal{M}^{\muhat}_Z$, where $Z$ is a variety, pick $n\in \mathbb{N}$ such that the $\muhat$ action on $Y$ factors through the projection $\muhat\rightarrow \mu_n$.  Then we form the mapping torus 
\[
Y\times_{\mu_n}\mathbb{G}_m=(Y\times \mathbb{G}_m)/\mu_n
\]
where $\mu_n$ acts via $z\cdot (y,z')=(z\cdot y,z^{-1}z')$.  We define
\begin{align*}
    \tilde{f}\colon Y\times_{\mu_n}\mathbb{G}_m & \rightarrow Z\times\mathbb{A}^1\\
    (y,z') & \mapsto (f(y),z'^n).
\end{align*}
Set $\mathcal{F}=\tilde{f}_!\underline{\mathbb{Q}}_{Y\times_{\mu_n}\mathbb{G}_m}\in \Db{\MHM(Z\times \mathbb{A}^1)}$.  By construction, the cohomology sheaves of this direct image are locally constant along the fibers of the projection $Z\times\mathbb{A}^1\rightarrow Z$, away from the zero fiber.  In particular, $\mathcal{F}$ is a bounded complex of \textit{monodromic mixed Hodge modules}, and so $[\mathcal{F}]\in\KK_0(\MMHM(Z))$.  Sending $[Y\rightarrow Z]\mapsto [\mathcal{F}]$ defines a group homomorphism
\[
\Psi\colon\mathcal{M}^{\muhat}_Z\rightarrow \KK_0(\MMHM(Z))
\]
which is a $\lambda$-ring homomorphism in the event that $Z$ is also a commutative monoid.

Now consider $\Psi(Q^{\twist}_{L/\A^3})$.  This is the class of the complex of monodromic mixed Hodge modules
\be\label{relDef}
\mathcal{F}_{L/\A^3}\defeq \bigoplus_n \HC_{n!}\phim{\Tr(W_r)}\underline{\mathbb{Q}}_{\mathcal{N}^{\circ\circ}_n}\otimes (\mathbb{T}^{1/2})^{-4n-2n^2}
\ee
where $\mathbb{T}^{1/2}$ is a tensor square root of the complex of mixed Hodge structures $\mathrm{H}_c(\A^1,\mathbb{Q})$ --- i.e.~a half Tate twist, concentrated in cohomological degree one.  From (\ref{pre_sym}) we deduce that
\begin{align}
    \nonumber
\Psi(Q^{\twist}_{L/\A^3})&=[\mathcal{F}_{L/\A^3}]\\
\label{pre_conj}
&=\left[\Sym\left((\Delta_{1\,!}\underline{\mathbb{Q}}_{\A^1}\otimes\mathbb{T}^{-1})\oplus \left(\bigoplus_{n\geq 1}\Delta_{n\,!}\underline{\mathbb{Q}}_{\mathbb{A}^3}\otimes \left(\bigoplus_{i=-2}^{-1-n}\mathbb{T}^i\right)\right)\right)\right].
\end{align}
It follows as in \cite[Prop.~3.5]{BenSvenQEA} from finiteness of the map $\cup \colon \Sym(\A^3)\times\Sym(\A^3)\rightarrow\Sym(\A^3)$ that the object in the square brackets on the right hand side is \textit{pure}, in the sense that its $i$th cohomology is pure of weight $i$.  From semisimplicity of the category of pure mixed Hodge modules (proved by Saito, see above references), the following two statements are equivalent:
\begin{itemize}
    \item The complex of monodromic mixed Hodge modules $\mathcal{F}_{L/\A^3}$ is pure.
    \item
    There is an isomorphism 
    \begin{equation}
        \label{cat_thm_B}
    \mathcal{F}_{L/\A^3}\cong \Sym\left((\Delta_{1\,!}\underline{\mathbb{Q}}_{\A^1}\otimes\mathbb{T}^{-1})\oplus \left(\bigoplus_{n\geq 1}\Delta_{n\,!}\underline{\mathbb{Q}}_{\mathbb{A}^3}\otimes \left(\bigoplus_{i=-2}^{-1-n}\mathbb{T}^i\right)\right)\right).
    \end{equation}
\end{itemize}
\begin{conjecture}
\label{PurityConj}
The above two statements are true.  In particular, there is an isomorphism of $\mathbb{Z}$-graded mixed Hodge structures
\[
\bigoplus_{n\geq 0}\HO_c(Q_L^n,\phim{\Tr(W_r)}\underline{\mathbb{Q}}_{\mathcal{N}^{\circ\circ}_n})\otimes\Tate^{-2n-n^2}\cong \Sym\left(\bigoplus_{n\geq 1}\left(\bigoplus_{i=-2}^{-n-1}\mathbb{T}^i\right)\right)\otimes \Sym(V),
\]
with $n$ keeping track of the degree on both sides, and where $V$ is a one dimensional pure weight zero Hodge structure placed in degree 1.  The $\mathbb{Z}$-grading on the right hand side comes from the grading on the object we are taking the symmetric algebra of (i.e.~we ignore the fact that $\Sym$ of any object acquires an extra $\mathbb{Z}$-grading).
\end{conjecture}
The isomorphism in (\ref{cat_thm_B}) would categorify our Theorem \ref{thm:newB}, and would lift it from a formula to an isomorphism (i.e. categorify it).  Since it would take quite some time to even fill in the requisite definitions for the above discussion, we leave this conjecture to future work.
\subsection{Further directions II: A new CoHA module}
We recall some very general theory regarding critical cohomological Hall algebras and their representations.  

First, we fix a quiver $Q$ and a potential $W$.  The Euler form for $Q$ is defined by
\begin{align*}
\chi_Q\colon \mathbb{Z}^{Q_0}\times\mathbb{Z}^{Q_0}&\rightarrow\mathbb{Z}\\
(\gamma,\gamma')&\mapsto \sum_{i\in Q_0}\gamma_i\gamma'_i-\sum_{a\in Q_1}\gamma_{s(a)}\gamma'_{t(a)}.
\end{align*}
A stability condition $\zeta\in\mathbb{Q}^{Q_0}$ is called generic if for any two nonzero dimension vectors of the same slope, we have $\langle \gamma,\gamma'\rangle=0$.  Recall that a quiver is called symmetric if for all vertices $i,j\in Q_0$ we have
\[
\#\bigl\{a\;\lvert\;s(a)=i,\;t(a)=j\bigr\}\;=\;\#\bigl\{a\;\lvert\;s(a)=j,\;t(a)=i\bigr\}.
\]
Note that the genericity condition is vacuous for symmetric quivers such as the three loop quiver obtained by removing the framing from the quiver $Q_{\BBS}$ from Section \ref{Section:BBS}, and the quiver $Q_{\con}$ from Section \ref{conifold_section}.

Fix a  generic stability condition $\zeta\in \mathbb{Q}^{Q_0}$ and a slope $\theta\in (0,\infty)$.  We set $\Lambda_{\theta}^{\zeta}\subset\mathbb{N}^{Q_0}$ to be the submonoid of dimension vectors of slope $\theta$.  Then as in \cite{KSCOHA} we endow the graded monodromic mixed Hodge structure
\be
    \label{COHA_under}
\mathcal{H}_{Q,W,\theta}^{\zeta}\defeq \bigoplus_{\dd\in\Lambda_{\theta}^{\zeta}}\HO(\Rep^{\zeta}_{\dd}(Q)/\GL_{\dd},\phim{\Tr(W)}\underline{\mathbb{Q}})\otimes\Tate^{\chi(\dd,\dd)/2}
\ee
with the Hall algebra product arising from the stack of short exact sequences of right $\mathbb{C}Q$-modules.  The multiplication respects the monodromic mixed Hodge structure on (\ref{COHA_under}) --- in other words (\ref{COHA_under}) is made into an algebra object in the tensor category of $\Lambda_{\theta}^{\zeta}$-graded monodromic mixed Hodge structures.

In fact we will only consider the cohomological Hall algebra for which $Q$ is symmetric, and $\zeta=(0,\ldots,0)$ is the \textit{degenerate} stability condition, with $\theta=0$.  So we will drop $\zeta$ and $\theta$ from the notation and just write $\mathcal{H}_{Q,W}$ for the cohomological Hall algebra associated to $Q$ and $W$.

Let $Q\subset Q_{\fra}$ be an inclusion of quivers, with $Q$ a full subquiver.  We do not assume that $Q_{\fra}$ is symmetric.  Let $I\subset \mathbb{C}Q_{\fra}$ be the two-sided ideal generated by all paths in $Q_{\fra}$ not contained in $Q$.  Then $\mathbb{C}Q\cong\mathbb{C}Q_{\fra}/I$, and we let $q\colon \mathbb{C}Q_{\fra}\rightarrow\mathbb{C}Q$ be the induced surjection.

Let $W_{\fra}$ be a potential extending $W$, in the sense that $qW_{\fra}=W$.  Let $\zeta\in\mathbb{Q}^{Q_{\fra,0}}$ be a stability condition extending the degenerate stability condition, i.e. such that $\zeta\lvert_{Q_0}=0$.  Let $\ff\in\mathbb{N}^{Q_{\fra,0}\setminus Q_0}$ be a \textit{framing} dimension vector.  Define 
\be
\label{N_module_def}
\mathcal{N}_{Q_{\fra},\ff}^{\zeta}\defeq \bigoplus_{\dd\in\mathbb{N}^{Q_0}}\HO(\Rep^{\zeta}_{(\dd,\ff)}(Q_{\fra})/\GL_{\dd},\phim{\Tr(W_{\fra})}\underline{\mathbb{Q}})\otimes\Tate^{\chi_{Q}(\dd,\dd)/2-\chi_{Q_{\fra}}((\dd,0),(0,\ff))}.
\ee
Via the usual correspondences, it is standard to check that $\mathcal{N}^{\zeta}_{Q_{\fra},\ff}$ is a module for $\mathcal{H}_{Q,W}$.

Now we make this setup more specific.  Let $Q$ be the three loop quiver, obtained by removing $\underline{\infty}$ and all arrows containing it from $Q_r$.  Considering the inclusion $Q\subset Q_{\BBS}$, and setting $\zeta_{\infty}=-1$ and $\ff=1$, the module $\mathcal{N}_{Q_{\fra},1}^{\zeta}$ is precisely the vanishing cycle cohomology of $\Hilb(\mathbb{A}^3)$.  It was shown in version one of \cite{BenSvenQEA} that in fact this module is cyclic over the CoHA $\mathcal{H}_{Q,W}$.  For a more recent example of a geometrically motivated class of modules for this CoHA, that falls under the general construction above, the reader may consult \cite{RSYZ19}, where an action on the space of spiked instantons is considered.  

The purpose of this subsection is to add one more geometrically motivated example to the list, namely, we consider the inclusion $Q\subset Q_r$, the framing vector $\ff=1$, and the stability condition $\zeta_{\underline{\infty}}=-1$.  Then by Proposition \ref{prop:critQL}, we obtain 
\be
    \label{ModuleDef}
\mathcal{N}_{Q_{\fra},\ff}^{\zeta}=\bigoplus_{n\geq 0}\HO(Q_L^n,\phim{\Tr(W_r)}\underline{\mathbb{Q}}_{\mathcal{N}^{\circ\circ}_n})\otimes\Tate^{-n-n^2}
\ee
i.e. the module constructed for the above data is precisely the vanishing cycle cohomologies of all of the quot schemes $Q_L$.  Keeping track of the various Tate twists, (\ref{ModuleDef}) is the hypercohomology of the Verdier dual of the monodromic mixed Hodge module 
\[
\phim{\Tr(W_r)}\underline{\mathbb{Q}}_{\mathcal{N}^{\circ\circ}_n}\otimes \mathbb{T}^{-2n-n^2}
\]
considered in (\ref{relDef}).  According to the conjecture of the previous subsection, this module should itself be isomorphic to the underlying graded mixed Hodge module of a symmetric algebra.  In fact by utilising the factorization sheaf structure implicit in our calculations in Section \ref{sec:mstr} it is not hard to find a cocommutative coproduct on (\ref{ModuleDef}), and a candidate for a compatible product, leading to our final conjecture:
\begin{conjecture}
\label{UEAconj}
The graded mixed Hodge structure $\mathcal{N}_{Q_{\fra},\ff}^{\zeta}$ is a universal enveloping algebra.
\end{conjecture}
The connection with Conjecture \ref{PurityConj} is that by proving a version of Conjecture \ref{UEAconj} over the base $\Sym(\mathbb{A}^3)$ one would deduce the purity conjecture as in \cite[Thm.~A]{BenPreproj}.
\begin{remark}
The above module $\mathcal{N}^{\zeta}_{Q_{\fra},1}$ is in fact the third in a natural sequence, the first two elements of which will be well-known to the reader.  Firstly, we can remove the arrow $a'_2$ from $Q_r$ to obtain a new quiver $Q'_r$.  The pullback of $\Tr(W_r)$ along the extension by zero morphism
\[
\Rep^{\zeta}_{(n,1)}(Q'_r)/\GL_n\rightarrow \Rep^{\zeta}_{(n,1)}(Q_r)/\GL_n 
\]
is induced by the potential $W_r'=a''_1b''_1a''_2-a''_2b''_1a''_1-a''_2b'_1a'_1$, and it is easy to check that there is an isomorphism
\[
\HO\left(\Rep^{\zeta}_{(n,1)}(Q'_r)/\GL_n,\phim{\Tr(W'_r)}\underline{\mathbb{Q}}\right)\otimes \Tate^{-n-n^2}\cong \HO(\Hilb^n\mathbb{A}^2,\mathbb{Q})\otimes\Tate^{-n}
\]
with the \textit{usual} cohomology of the Hilbert scheme for $\mathbb{A}^2$.  Going further, we can remove the arrow $b'_1$, recovering the framed BBS quiver, with its usual potential, and the MacMahon module provided by the (vanishing cycle) cohomology of $\Hilb^n\mathbb{A}^3$.  That the cohomology of $\Hilb^n\mathbb{A}^2$ finds itself sandwiched between the vanishing cycle cohomology of the Hilbert scheme of $\mathbb{A}^3$ and the quot scheme $Q_L^n$ in this way is a mystery that we leave to future research to understand properly.
\end{remark}

\clearpage
\bibliographystyle{amsplain-nodash}
\bibliography{bib}

\ifx\undefined\bysame
\newcommand{\bysame}{\leavevmode\hbox to3em{\hrulefill}\,}
\fi
\begin{thebibliography}{10}

\bibitem{BR18}
Sjoerd Beentjes and Andrea~T. Ricolfi, {\em {Virtual counts on Quot schemes and
  the higher rank local DT/PT correspondence}}, To appear in Math. Res. Lett.
  \href{https://arxiv.org/abs/1811.09859}{ArXiv:1811.09859}, 2018.

\bibitem{Beh}
Kai Behrend, {\em {Donaldson--Thomas type invariants via microlocal geometry}},
  Ann. of Math. {\bf 2} (2009), no.~170, 1307--1338.

\bibitem{BBS}
Kai Behrend, Jim Bryan, and Bal{{\'a}}zs Szendr{\H{o}}i, {\em Motivic degree
  zero {D}onaldson--{T}homas invariants}, Invent. Math. {\bf 192} (2013),
  no.~1, 111--160.

\bibitem{Bittner05}
Franziska Bittner, {\em On motivic zeta functions and the motivic nearby
  fiber}, Math. Z. {\bf 249} (2005), 63--83.

\bibitem{BriFMT}
Tom Bridgeland, {\em Equivalences of triangulated categories and fourier--mukai
  transforms}, Bull. London Math. Soc. {\bf 31} (1999), 25--34.

\bibitem{Bri}
Tom Bridgeland, {\em Hall algebras and curve counting invariants}, J. Amer.
  Math. Soc. {\bf 24} (2011), no.~4, 969--998.

\bibitem{BenPreproj}
Ben Davison, {\em {The integrality conjecture and the cohomology of
  preprojective stacks}}, \href{arXiv:1602.02110v3}{ArXiv:1602.02110v3}, 2017.

\bibitem{BenSven3}
Ben Davison and Sven Meinhardt, {\em Motivic {D}onaldson--{T}homas invariants
  for the one-loop quiver with potential}, Geom. Topol. {\bf 19} (2015), no.~5,
  2535--2555.

\bibitem{BenSven2}
Ben Davison and Sven Meinhardt, {\em {The motivic Donaldson--Thomas invariants
  of $(-2)$-curves}}, Algebra and Number Theory {\bf 11} (2017), no.~6,
  1243--1286.

\bibitem{BenSvenQEA}
Ben Davison and Sven Meinhardt, {\em {Cohomological Donaldson--Thomas theory of
  a quiver with potential and quantum enveloping algebras}}, Invent. Math. {\bf
  221} (2020), no.~3, 777--871.

\bibitem{DenefLoeser2}
Jan Denef and Fran{\c c}ois Loeser, {\em {Motivic exponential integrals and a
  motivic Thom--Sebastiani theorem}}, Duke Math. J {\bf 99} (1999), 285--309.

\bibitem{DenefLoeser1}
Jan Denef and Fran{\c c}ois Loeser, {\em {Geometry on arc spaces of algebraic
  varieties}}, {3rd European congress of mathematics (ECM), Barcelona, Spain,
  July 10--14, 2000. Volume I}, Basel: Birkh\"auser, 2001, pp.~327--348.

\bibitem{FMR_K-DT}
Nadir Fasola, Sergej Monavari, and Andrea~T. Ricolfi, {\em {Higher rank
  K-theoretic Donaldson--Thomas theory of points}}, Forum Math. Sigma {\bf 9}
  (2021), no.~E15, 1--51.

\bibitem{GLMps}
Sabir~M. {Gusein-Zade}, Ignacio {Luengo}, and Alejandro {Melle-Hern\'andez},
  {\em {A power structure over the Grothendieck ring of varieties}}, {Math.
  Res. Lett.} {\bf 11} (2004), no.~1, 49--57.

\bibitem{GLMHilb}
Sabir~M. {Gusein-Zade}, Ignacio {Luengo}, and Alejandro {Melle-Hern\'andez},
  {\em {Power structure over the Grothendieck ring of varieties and generating
  series of Hilbert schemes of points}}, {Mich. Math. J.} {\bf 54} (2006),
  no.~2, 353--359.

\bibitem{KW1}
Igor~R. Klebanov and Edward Witten, {\em Superconformal field theory on
  threebranes at a {C}alabi-{Y}au singularity}, Nuclear Phys. B {\bf 536}
  (1999), no.~1-2, 199--218.

\bibitem{KSCOHA}
Maxim Kontsevich and Yan Soibelman, {\em {Cohomological Hall algebra,
  exponential Hodge structures and motivic Donaldson--Thomas invariants}},
  Comm. in Num. Th. \& Phys. {\bf 5} (2011), no.~2, 231--352.

\bibitem{kreschcycle}
Andrew Kresch, {\em Cycle groups for {A}rtin stacks}, Invent. Math. {\bf 138}
  (1999), no.~3, 495--536.

\bibitem{LooijengaMM}
Eduard Looijenga, {\em {Motivic measures}}, {S\'eminaire Bourbaki. Volume
  1999/2000. Expos\'es 865--879}, Paris: Soci\'et\'e Math\'ematique de France,
  2002, pp.~267--297, ex.

\bibitem{RefConifold}
Andrew Morrison, Sergey Mozgovoy, Kentaro Nagao, and Bal{{\'a}}zs
  Szendr{\H{o}}i, {\em Motivic {D}onaldson--{T}homas invariants of the conifold
  and the refined topological vertex}, Adv. Math. {\bf 230} (2012), no.~4-6,
  2065--2093.

\bibitem{MorrNagao}
Andrew Morrison and Kentaro Nagao, {\em Motivic {D}onaldson--{T}homas
  invariants of small crepant resolutions}, Algebra Number Theory {\bf 9}
  (2015), no.~4, 767--813.

\bibitem{MR18}
Riccardo Moschetti and Andrea~T. Ricolfi, {\em On coherent sheaves of small
  length on the affine plane}, Journal of Algebra {\bf 516} (2018), 471--489.

\bibitem{Mozgovoy1}
Sergey Mozgovoy, {\em On the motivic {D}onaldson--{T}homas invariants of
  quivers with potentials}, Math. Res. Lett. {\bf 20} (2013), no.~1, 107--118.

\bibitem{NagNak}
Kentaro {Nagao} and Hiraku {Nakajima}, {\em {Counting invariant of perverse
  coherent sheaves and its wall-crossing}}, {Int. Math. Res. Not.} {\bf 2011}
  (2011), no.~17, 3885--3938.

\bibitem{PT}
Rahul Pandharipande and Richard~P. Thomas, {\em Curve counting via stable pairs
  in the derived category}, Invent. Math. {\bf 178} (2009), no.~2, 407--447.

\bibitem{BPS}
Rahul Pandharipande and Richard~P. Thomas, {\em Stable pairs and {BPS}
  invariants}, J. Amer. Math. Soc. {\bf 23} (2010), no.~1, 267--297.

\bibitem{APPP1}
Adam {Parusi\'nski} and Piotr {Pragacz}, {\em {Characteristic classes of
  hypersurfaces and characteristic cycles}}, {J. Algebr. Geom.} {\bf 10}
  (2001), no.~1, 63--79.

\bibitem{RSYZ19}
Miroslav Rap{\v c}{\'a}k, Yan Soibelman, Yaping Yang, and Gufang Zhao, {\em
  Cohomological hall algebras, vertex algebras and instantons}, Comm. Math.
  Phys. {\bf 376} (2020), no.~3, 1803--1873.

\bibitem{RS17}
Jie Ren and Yan Soibelman, {\em Cohomological {H}all algebras, semicanonical
  bases and {D}onaldson--{T}homas invariants for 2-dimensional {C}alabi--{Y}au
  categories (with an appendix by {B}en {D}avison)}, Algebra, geometry, and
  physics in the 21st century, Progr. Math., vol. 324, Birkh\"{a}user/Springer,
  Cham, 2017, pp.~261--293.

\bibitem{ThesisR}
Andrea~T. Ricolfi, {\em {Local Donaldson--Thomas invariants and their
  refinements}}, Ph.D. thesis, University of Stavanger, 2017.

\bibitem{Ricolfi2018}
Andrea~T. Ricolfi, {\em The {DT}/{PT} correspondence for smooth curves}, Math.
  Z. {\bf 290} (2018), no.~1-2, 699--710.

\bibitem{LocalDT}
Andrea~T. Ricolfi, {\em {Local contributions to Donaldson--Thomas invariants}},
  Int. Math. Res. Not. IMRN {\bf 2018} (2018), no.~19, 5995--6025.

\bibitem{mot_quot}
Andrea~T. Ricolfi, {\em {On the motive of the Quot scheme of finite quotients
  of a locally free sheaf}}, J. Math. Pures Appl. {\bf 144} (2020), 50--68.

\bibitem{Ob_Mot}
Andrea~T. Ricolfi, {\em {Virtual classes and virtual motives of Quot schemes on
  threefolds}}, Adv. Math. {\bf 369} (2020), 107182.

\bibitem{Rydh1}
David Rydh, {\em {Families of cycles and the Chow scheme}}, Ph.D. thesis, KTH,
  Stockholm, 2008.

\bibitem{Saito89}
Morihiko Saito, {\em {Introduction to mixed Hodge modules}}, Ast\'erisque {\bf
  179} (1989), 889--921.

\bibitem{Saito90}
Morihiko Saito, {\em {Mixed Hodge modules}}, Publ. Res. Inst. Math. {\bf 26}
  (1990), 221--333.

\bibitem{Seg}
Ed~Segal, {\em {The $A_\infty$ Deformation Theory of a Point and the Derived
  Categories of Local Calabi-Yaus}}, J. Algebra {\bf 320} (2008), no.~8,
  3232--3268.

\bibitem{CohDT}
Bal\'{a}zs Szendr\H{o}i, {\em Cohomological {D}onaldson--{T}homas theory},
  String-{M}ath 2014, Proc. Sympos. Pure Math., vol.~93, Amer. Math. Soc.,
  Providence, RI, 2016, pp.~363--396.

\bibitem{MR2403807}
Bal{{\'a}}zs Szendr{\H{o}}i, {\em Non-commutative {D}onaldson--{T}homas
  invariants and the conifold}, Geom. Topol. {\bf 12} (2008), no.~2,
  1171--1202.

\bibitem{Toda1}
Yukinobu {Toda}, {\em {Curve counting theories via stable objects. I: DT/PT
  correspondence}}, {J. Am. Math. Soc.} {\bf 23} (2010), no.~4, 1119--1157.

\bibitem{TodaGVWCF}
Yukinobu Toda, {\em {Gopakumar--Vafa invariants and wall-crossing}},
  {\href{https://arxiv.org/abs/1710.01843}{ArXiv:1710.01843}}, 2017.

\end{thebibliography}

\end{document}